\documentclass[10pt,dvipsnames]{article}

\usepackage{amsmath}
\usepackage{amssymb}
\usepackage{amsthm}
\usepackage{bm}
\usepackage[backend=biber,bibencoding=utf8,bibstyle=alphabetic,citestyle=alphabetic,sorting=nyt,maxnames=4]{biblatex}
	\AtEveryBibitem{\clearfield{doi}}
	\AtEveryBibitem{\clearfield{isbn}}
	\AtEveryBibitem{\clearfield{issn}}
	\AtEveryBibitem{\clearfield{url}}
	
	\renewbibmacro{in:}{\ifentrytype{article}{}{\printtext{\bibstring{in}\intitlepunct}}}
	\AtEveryBibitem{\ifentrytype{book}{\clearfield{pages}}{}}
\usepackage[colorlinks=true,allcolors=purple]{hyperref}
\usepackage{cleveref}
	\crefformat{equation}{\textup{#2(#1)#3}}
\usepackage{dsfont}
\usepackage{enumitem}
\usepackage[OT2,OT1]{fontenc}
\usepackage[para]{footmisc}
\usepackage[cal=cm,scr=boondoxo]{mathalfa}
\usepackage{pifont}
\usepackage{refcount}
\usepackage{stmaryrd}
\usepackage{textcomp}
\usepackage{tikz}
	\usetikzlibrary{arrows}
	\usetikzlibrary{patterns}
	\usetikzlibrary{angles}
\usepackage{tikz-cd}
\usepackage[textwidth=3cm,colorinlistoftodos]{todonotes}
\usepackage{xfrac}

\usepackage{geometry}
\numberwithin{equation}{section}			
\swapnumbers						

\newcounter{Enum}				
\newenvironment{Enumerate}{\begin{enumerate}[label={\rm({\roman*})}]}{\end{enumerate}}

\newcommand{\descriptionlabelsave}{}		
\newenvironment{Itemize}{%
	\renewcommand{\descriptionlabelsave}{\descriptionlabel}\renewcommand{\descriptionlabel}{$\triangleright$}%
	\begin{description}[leftmargin=15pt,itemindent=-5.2pt]}{%
	\end{description}\renewcommand{\descriptionlabel}{\descriptionlabelsave}}

\newcounter{StepsCount}				
\newenvironment{Steps}{%
	\begin{list}{{\sf Step}\ \ding{\value{StepsCount}}\,:}{%
	\usecounter{StepsCount} \leftmargin=0pt \labelwidth=12pt \itemindent=\labelwidth%
	\itemsep=5pt\listparindent=\parindent} \setcounter{StepsCount}{191}}{\end{list}}
\newcounter{StepsRefCount}

\newcounter{ElistCount}				
\newenvironment{Elist}{%
	\begin{list}{{\rm(\roman{ElistCount})}}{%
	\usecounter{ElistCount} \leftmargin=0pt \labelwidth=0pt \itemindent=6pt%
	\itemsep=5pt\listparindent=\parindent} \setcounter{ElistCount}{0}}{\end{list}}
\newcommand{\Elistref}[1]{(\romannumeral\getrefnumber{#1})}

\newenvironment{Ilist}{
	\begin{list}{$\triangleright$}{\leftmargin=0pt \labelwidth=13pt \itemindent=\labelwidth%
	\itemsep=5pt\listparindent=\parindent}}{\end{list}}

\newenvironment{IIlist}{
	\begin{list}{$\triangleright\triangleright$}{\leftmargin=0pt \labelwidth=17pt \itemindent=\labelwidth%
	\itemsep=5pt\listparindent=\parindent}}{\end{list}}

\theoremstyle{plain}
	\newtheorem{lemma}{Lemma}[section]
	\newtheorem{proposition}[lemma]{Proposition}
	\newtheorem{theorem}[lemma]{Theorem}
	\newtheorem{corollary}[lemma]{Corollary}
\theoremstyle{definition}
	\newtheorem{definition}[lemma]{Definition}
\theoremstyle{remark}
	\newtheorem{remark}[lemma]{Remark}
	\newtheorem{example}[lemma]{Example}

\newenvironment{Proposition}{\begin{proposition}}{\par\noindent\rule{5em}{1pt}\end{proposition}}
\newenvironment{Theorem}{\begin{theorem}}{\par\noindent\rule{5em}{1pt}\end{theorem}}

\newenvironment{Definition}{\begin{definition}}{\hfill$\blacktriangleleft$\end{definition}}
\newenvironment{Remark}{\begin{remark}}{\hfill$\vartriangleleft$\end{remark}}
\newenvironment{Example}{\begin{example}}{\hfill$\vartriangleleft$\end{example}}

\newcommand{\mc}[1]{{\mathcal{#1}}}			
\newcommand{\bb}[1]{{\mathbb{#1}}}			
\newcommand{\ov}{\overline}				
\newcommand{\mr}{\mathring}				

\newcommand{\Dis}[1]{${\displaystyle{#1}}$}		
\DeclareMathOperator{\ran}{ran}				
\newcommand{\Dummy}{\text{\textvisiblespace\kern1pt}}	
\DeclareMathOperator{\Id}{id}				
\DeclareMathOperator{\Span}{span}			

\newcommand{\DS}{\mid\mkern3mu}				
\newcommand{\DSb}{\mkern4.5mu\big|\mkern7.5mu}		
\newcommand{\DSB}{\mkern4.5mu\Big|\mkern7.5mu}		
\newcommand{\DQ}{\mkern6mu}				
\newcommand{\DP}{{\mathop:\kern5pt}}			
\newcommand{\DF}{\colon}				
\newcommand{\DE}{\mathrel{\mathop:}=}			
\newcommand{\ED}{=\mathrel{\mathop:}}			
\newcommand{\DI}{\mathrel{\mathop:}\Leftrightarrow}	
\newcommand{\CAS}{&\text{if}\ }				
\newcommand{\CASO}{&\text{otherwise}}			

\newcommand{\Function}[5][2mm]{%
	\DF\left\{
	\begin{array}{rcl} 
		{#2} & \to & {#3} 
		\\[#1]
		{#4} & \mapsto & {#5} 
	\end{array}
	\right.
	}

\DeclareMathOperator{\HomCSL}{Hom_{\rm CSL}}		
\DeclareMathOperator{\HomCA}{Hom_{\rm CA}}		
\DeclareMathOperator{\HomSL}{Hom_{\rm SL}}		
\DeclareMathOperator{\Conv}{co}				
\DeclareMathOperator{\Cone}{cone}			
\DeclareMathOperator{\Edim}{e-dim}			
\DeclareMathOperator{\Gdim}{g-dim}			
\newcommand{\E}[1][]{\mr{e}^{#1}}			
\DeclareMathOperator{\Fin}{fin}				

\NewDocumentCommand{\Moewe}{ O{14} O{8} O{8.33} m }{%
	{#4}%
	\mkern-#1mu\raisebox{#3pt}{%
	\begin{tikzpicture}[scale=0.0833]
		\draw (0,0) arc (10:135:1);
		\draw (0,0) arc (170:45:1);
	\end{tikzpicture}
}\mkern-#2mu}

\begin{document}

\begin{flushleft}
	{\Large\bf Homomorphisms and subalgebras of free\\[2mm] convex semilattices}
	\\[5mm]
	\textsc{%
	Harald Woracek
	}
	\\[6mm]
	{\small
	\textbf{Abstract:}
		A convex semilattices is a convex algebra endowed with an additional semilattice operation that satisfies
		a distributive law. 
		We investigate free convex semilattices by studying the homomorphisms into products $\bb R^\Omega$
		and, more specifically, the homomorphisms between two free finitely generated convex semilattices. 
		We characterise the finitely generated subalgebras of a free finitely generated convex semilattice.
	\\[3mm]
	{\bf AMS MSC 2010:} 06\,F\,99,\ 52\,A\,01,\ 08\,B\,20
	}
\end{flushleft}


%
%
%
\section{Introduction}

Convex semilattices are an equational class of algebras that combine convex algebras and a semilattice operation by imposing a distributive 
law (all definitions will be recalled in Section~2.1). Besides their intrinsic algebraic relevance, on which we shall elaborate
below, convex semilattices have recently been used to understand the interplay of probabilistic phenomena and nondeterminism in the
analysis of automata \cite{bonchi.sokolova.vignudelli:2022}.

Given any algebra satisfying suitable laws one can consider the algebra of its subalgebras \cite{adaricheva.pilitowska.stanovski:2008}. This
algebra naturally carries a semilattice operation assigning to two subalgebras the subalgebra generated by their union. If one has
started from a convex algebra one will in this way obtain a convex semilattice. This construction not only produces some convex
semilattices, but 
is intrinsically related to that equational class: the algebra of all finitely generated subalgebras of the free convex algebra 
$\mc DA$ with basis $A$ is the free convex semilattice $\mc CA$ with basis $A$. A high-level (category theoretic) reason for this is
that convex semilattices occur from the composition of the ``finitely supported distributions''-monad with the ``finitely generated
subalgebras''-monad, 

In the present paper we investigate the structure of free convex semilattices by studying the homomorphisms of $\mc CA$ into some
product $\bb R^\Omega$ being a convex semilattice in a natural way. More specifically, we study homomorphisms between two free finitely
generated convex semilattices. Furthermore, we characterise finitely generated convex semilattices that can be embedded into a
free finitely generated one. Such algebras are sometimes called exact \cite{agliano.ugolini:2024,flaminio.ugolini:2026-arXiv}. 

Let us briefly describe the core content of the present paper.
We associate with a homomorphism $\Phi$ of $\mc CA$ into $\bb R^\Omega$ a subset 
$H(\Phi)$ of $\bb R^A$, and show that the kernel of $\Phi$ is characterised by this set, cf.\ \Cref{S3}. 
One important feature is that the set $H(\Phi)$ is defined using solely the action of $\Phi$ on the free generators of $\mc CA$.
This theorem can be seen as an indirect description of a part of the congruence lattice of $\mc CA$. 
We note that, while the congruence lattice of the free convex
algebra $\mc DA$ can be described in a complicated but explicit manner \cite{sokolova.woracek:pcacon}, an explicit description of
the full congruence lattice of $\mc CA$ is not known and currently seems out of reach. 
For finite $A$ we single out those subsets $H(\Phi)$ that arise from
homomorphisms into a free finitely generated convex semilattice, cf.\ \Cref{S24}. This description is explicit and achieved by a
number $\delta(H(\Phi))$ which quantifies a geometric property of the set $H(\Phi)$ as a subset of the euclidean space $\bb R^A$. 
Having at hand this computable quantity one can investigate exact convex semilattices. We determine whether a convex
semilattice is exact, and if it is give a formula for the minimal number $n$ such that it embeds into $\mc Cn$, cf.\ \Cref{S35}. 
We call this minimum the embedding dimension of the convex semilattice. 

One application of the formula given in \Cref{S35} is the following. 
Assume we have $n_1,\ldots,n_l\in\bb N\setminus\{0\}$ where $n_i\geq 3$ for at least one $i$. Then the product 
$\mc Cn_1\times\ldots\times\mc Cn_l$ is exact and the embedding dimension equals $\max\{n_1,\ldots,n_l\}$, cf.\ \Cref{S45}. 
Curiously, yet for a good reason, the product $\mc C2\times\mc C2$ is exact with embedding dimension $3$. 

The structuring of the paper follows the outline of its content given above. In order to make the presentation as self-contained as
possible, we include a list of general notation and conventions at the end of this introduction, and collect relevant definitions,
examples, and basic facts in a preliminary section. In Section~3.1 we give the definition of the set $H(\Phi)$ and provide some
properties. Section~3.2 contains \Cref{S3} and its proof, and in Section~3.3 we give a practical extension of the theorem to a more
flexible setting, cf.\ \Cref{S30}. 
Section~4 is devoted to \Cref{S24} and its proof. This is the technically most involved part of the paper. In
Section~5.1 we discuss the embedding dimension, and show \Cref{S35}. This theorem is an important assertion, but appears as
a fairly simple consequence of the two
previous theorems. In Section~5.2 we consider another natural notion of ``dimension'' of an algebra, namely, the mimimal cardinality
of a subset that generates the algebra. Interestingly, this number is related to the embedding dimension in a neat way, cf.\
\Cref{S51}. Finally, in Section~5.3, we discuss the embedding dimension in the equational class of convex algebras. This consideration
is included because it illustrates that free convex semilattices have a rather complicated structure compared to free convex algebras.

Last but not least let us add a few words about the generality of our setting. The high-level categorical interpretation mentioned 
above generalises
to other algebras than convex algebras, namely such whose operations satisfy entropic and idempotent laws, so-called modes, e.g.\
\cite{romanowska.smith:2002}. The construction of the algebra of subalgebras then leads to so-called modals, e.g.\
\cite{romanowska.smith:1985,smith:2011}. The category theoretic interpretation of these algebras as the Eilenberg-Moore algebras of the
composition of two monads is available \cite{kurpiel:1987}. 
On stark contrast the results of the present paper are most likely specific for convex algebras and convex semilattices. The reason
being that the geometry and topology of the euclidean space $\bb R^n$ (or, more generally, products $\bb R^\Omega$) is eminently important
for nearly all the proofs. 

\medskip\noindent
\textbf{Acknowledgement:}
The question whether a product of free convex semilattices can be embedded into some free convex semilattice was the starting
point of the present research. It came up in discussions with S.Milius and L.Schr\"oder during a visit to the FAU Erlangen whom we thank
for their hospitality.

\subsection{Notation and conventions}
\label{S79}

We list some notation and conventions that are used throughout the paper. We will give a reference to this list at the first
occurrance in the text. After that all what is listed here will be used without further notice.
\begin{Elist}
\item 
\label{S85}
	Whenever we speak of a linear space, we mean a linear space over the scalar field $\bb R$.
	If $M$ is a subset of a linear space, we denote by $\Span M$ the linear subspace generated by $M$, and by $\Cone M$ the cone
	generated by $M$.
\item Elements of a product $X^Y$ will interchangably be considered as tuples $(x_y)_{y\in Y}$ or as functions $f\DF Y\to X$.
\item
\label{S17}
	If $A$ is a nonempty set and $b\in A$, then $e_b$ denotes 
	\[
		e_b\DE(\delta_{a,b})_{a\in A}\in\bb R^A,
	\]
	where $\delta_{a,b}$ is the Kronecker-Delta.
	Moreover, we set 
	\[
		\E[A]\DE(1)_{a\in A}\in\bb R^A.
	\]
\item 
\label{S87}
	Using $A\DE\{1,\ldots,n\}$ and $j\in\{1,\ldots,n\}$ in the previous item, we in particular have that $e_j$ denotes the
	$j$-th canonical basis vector of $\bb R^n$ (independently of the dimension $n$), and that $\E[n]=\sum_{j=1}^ne_j$.
\item 
\label{S22}
	Denote 
	\[
		\bb R^A_{\Fin}\DE
		\big\{(\alpha_a)_{a\in A}\in\bb R^A\DSb \alpha_a=0\text{ for all but finitely many }a\in A\big\},
	\]
	and let $\llceil\Dummy,\Dummy\rrfloor\DF\bb R^A\times\bb R^A_{\Fin}\to\bb R$ be the bilinear form defined as 
	\[
		\llceil (\alpha_a)_{a\in A},(\beta_a)_{a\in A}\rrfloor\DE
		\sum_{\substack{a\in A\\ \beta_a\neq 0}}\alpha_a\beta_a
		\qquad\text{for }(\alpha_a)_{a\in A}\in\bb R^A,(\beta_a)_{a\in A}\in\bb R^A_{\Fin}.
	\]
\item 
	Using $A\DE\{1,\ldots,n\}$ in the previous item, the bilinear form $\llceil\Dummy,\Dummy\rrfloor$ reduces to the Euclidean 
	scalar product on $\bb R^n$, which we denote as
	\[
		\big((\alpha_j)_{j=1}^n,(\beta_j)_{j=1}^n\big)\DE\sum_{j=1}^n\alpha_j\beta_j
		\qquad\text{for }(\alpha_j)_{j=1}^n,(\beta_j)_{j=1}^n\in\bb R^n.
	\]
\item 
\label{S82}
	Unless stated differently, all topological notions in $\bb R^A$ refer to the product topology of the Euclidean topology on the
	factors $\bb R$. Note that this is the weak topology on $\bb R^A$ induced by $\bb R^A_{\Fin}$ via the bilinear form
	$\llceil\Dummy,\Dummy\rrfloor$. 
\item 
\label{S86}
	For a set $A$ denote by $\mc P(A)$ its powerset and 
	\[
		\mc P_{\Fin}(A)\DE\big\{B\in\mc P(A)\DSb B\text{ finite}\big\}.
	\]
\item 
\label{S38}
	For $n\in\bb N$ denote the \emph{unit sphere} in $\bb R^{n+1}$ as 
	\[
		S^n\DE\Big\{(\alpha_j)_{j=1}^{n+1}\in\bb R^{n+1}\DSB \sum_{j=1}^{n+1}\alpha_j^2=1\Big\}.
	\]
\item 
\label{S81}
	Let $L,L'$ be two linear spaces and $\phi\DF L\to L'$. 
	Then $\phi$ is called \emph{positively homogeneous}, if 
	\[
		\forall x\in L,\alpha\geq 0\DP \phi(\alpha x)=\alpha\phi(x).
	\]
	Every positively homogeneous function $\phi$ defined on $\bb R^{n+1}$ is determined by its values on $S^n$.
\item 
\label{S59}
	Let $L,L'$ be two linear spaces and $\phi\DF L\to L'$. Then $\phi$ is called \emph{piecewise linear}, if there exist finitely 
	many finitely generated cones $C_1,\ldots,C_N\subseteq L$ and linear maps $\xi_1,\ldots,\xi_N\DF L\to L'$ such that 
	\[
		\bigcup_{i=1}^N C_i=L
		\quad\text{and}\quad
		\forall i\in\{1,\ldots,N\}\DP \phi|_{C_i}=\xi_i|_{C_i}.
	\]
	Every piecewise linear function is positively homogeneous. 
\end{Elist}

\section{Preliminaries}

In this section we introduce and discuss the main objects of our present study.

\subsection{The equational class of convex semilattices}

Convex semilattices are obtained
by combining the algebraic structures of a convex algebra and a semilattice with an appropriate distributive law, and for 
the sake of completeness we recall also those.

\begin{Definition}
\label{S60}
	A \emph{convex algebra} is a set $X$ together with a family $(+_p)_{p\in[0,1]}$ of binary operations on $X$ that 
	satisfy
	\begin{Ilist}
	\item[] \Dis{\forall x\in X,p\in(0,1)\DP x+_px=x}
		\hspace*{25mm}(\emph{idempotence})
	\item[] \Dis{\forall x,y\in X,p\in(0,1)\DP x+_py=y+_{1-p}x}
		\hspace*{8.5mm}(\emph{parametric commutativity})
	\item[] \Dis{\forall x,y,z\in X,p,q\in(0,1)\DP (x+_py)+_qz=x+_{pq}\Big(y+_{\frac{(1-p)q}{1-pq}}z\Big)}
		\\[1mm] \hspace*{84mm}(\emph{parametric associativity})
	\item[] \Dis{\forall x,y\in X\DP x+_1y=x\wedge x+_0y=y}
		\hspace*{15mm}(\emph{projection axiom})
	\end{Ilist}
	We write a convex algebra as $\langle X,+_p\rangle$.

	The set of all homomorphisms of a convex algebra $X$ into another one $Y$ is denoted as $\HomCA(X,Y)$.
\end{Definition}

\begin{Definition}
\label{S61}
	A \emph{semilattice} is a set $X$ together with a binary operation $\oplus$ on $X$ that satisfies 
	\begin{Ilist}
	\item[] \Dis{\forall x\in X\DP x\oplus x=x}
		\hspace*{46mm}(\emph{idempotence})
	\item[] \Dis{\forall x,y\in X\DP x\oplus y=y\oplus x}
		\hspace*{35mm}(\emph{commutativity})
	\item[] \Dis{\forall x,y,z\in X\DP (x\oplus y)\oplus z=x\oplus(y\oplus z)}
		\hspace*{10mm}(\emph{associativity})
	\end{Ilist}
	We write a semilattice as $\langle X,\oplus\rangle$.

	The set of all homomorphisms of a semilattice $X$ into another one $Y$ is denoted as $\HomSL(X,Y)$.
\end{Definition}

\noindent
Now we combine these two structures.

\begin{Definition}
\label{S62}
	A \emph{convex semilattice} is a set $X$ together with binary operations $(+_p)_{p\in[0,1]}$ and $\oplus$ on $X$
	that satisfy
	\begin{Enumerate}
	\item $\langle X,+_p\rangle$ is a convex algebra,
	\item $\langle X,\oplus\rangle$ is a semilattice,
	\item Each $+_p$ is \emph{distributive over} $\oplus$, i.e., 
		\[
			\forall x,y,z\in X,p\in[0,1]\DP (x\oplus y)+_pz=(x+_pz)\oplus(y+_pz)
		\]
	\end{Enumerate}
	We write a convex semilattice as $\bb X=\langle X,+_p,\oplus\rangle$.

	The set of all homomorphisms of a convex semilattice $\bb X$ into another one $\bb Y$ is denoted as $\HomCSL(\bb X,\bb Y)$.
\end{Definition}

\noindent
Let us comment on these definitions.

\begin{Remark}
\label{S78}
	\phantom{}
	\begin{Elist}
	\item In the literature there are also other commonly used names for convex algebras, for example 
		convex spaces \cite{swirszcz:1974,flood:1981} or barycentric algebras \cite{stone:1949,komorowski.romanowska.smith:2019}.
		Equally, there exist different equivalent axiomatisations and we mention two.
		\begin{Ilist}
		\item One can define a ``convex algebra'' using only operations $+_p$ with $p\in(0,1)$, and require idempotence,
			parametric commutativity, and parametric associativity. 
			Then one can define operations $+_1$ and $+_0$ by the
			formula in the projection axiom, and obtain a convex algebra in the sense of \Cref{S60}.
		\item One can define a ``convex algebra'' using a family of operations 
			\[
				\Big\{+_{(p_i)_{i=1}^n}\DSB n\in\bb N\setminus\{0\},p_1,\ldots,p_n\geq 0,\sum_{i=1}^np_i=1\Big\},
			\]
			where $+_{(p_i)_{i=1}^n}$ is an $n$-ary operation, and require a general 
			\emph{projection axiom} and a \emph{barycentric axiom}, as in
			 \cite[Definition~3.1]{sokolova.woracek:pcacon}. 
			The connection with \Cref{S60} is made as follows. Given operations 
			$+_{(p_i)_{i=1}^n}$, one defines $+_p\DE+_{(p_i)_{i=1}^n}$ where $n\DE 2$, $p_1\DE p$, $p_2\DE 1-p$.
			Conversely, given $+_p$, one defines $+_{(p_i)_{i=1}^n}$ recursively by 
			\begin{align*}
				& +_{(p_i)_{i=1}^1}(x_1)\DE x_1,
				\\
				& +_{(p_i)_{i=1}^2}(x_1,x_2)\DE x_1+_{p_1}x_2,
				\\
				& +_{(p_i)_{i=1}^{n+1}}(x_1,\ldots,x_{n+1})\DE
				\begin{cases}
					x_1+_{p_1}\big(+_{(\frac{p_i}{1-p_1})_{i=2}^{n+1}}(x_2,\ldots,x_{n+1})\big)
					\CAS p_1\neq 1,
					\\
					x_1 \CAS p_1=1.
				\end{cases}
			\end{align*}
			Instead of writing $+_{(p_i)_{i=1}^n}$ it is customary to use the notation 
			\[
				+_{(p_i)_{i=1}^n}(x_1,\ldots,x_n)\ED\sum_{i=1}^np_ix_i.
			\]
		\end{Ilist}
	\item Using the projection axiom it follows that in every convex algebra the idempotence and parametric commutativity
		laws hold for all $p\in[0,1]$ and that the parametric associativity law holds for all 
		$(p,q)\in[0,1]^2\setminus\{(1,1)\}$. 
	\item A semilattice can also be defined as a partially ordered set $\langle X,\leq\rangle$ in which each two elements
		have a supremum. The connection is given by 
		\[
			x\leq y \DI x\oplus y=y\quad\text{and}\quad x\oplus y\DE\sup\{x,y\},
		\]
		respectively.
	\item In the definition of a convex semilattice one could equivalently require the distributive law only for 
		$p\in(0,1)$. For $p=0$ and $p=1$ it is a consequence of the projection axiom for $+_p$ and idempotence of $\oplus$. 
	\item In a convex semilattice the operation $\oplus$ is not necessarily distributive over $+_p$.
		In fact, in those convex semilattices that we deal with in the present paper
		this other distributive law does not hold.
	\end{Elist}
\end{Remark}

\begin{Example}
\label{S103}
	We define binary operations $+_p$ and $\oplus$ on $\bb R$ as 
	\begin{Ilist}
	\item[] \Dis{x+_py\DE px+(1-p)y}\qquad for $x,y\in\bb R$, $p\in(0,1)$,
	\item[] \Dis{x\oplus y\DE\max\{x,y\}}\qquad for $x,y\in\bb R$.
	\end{Ilist}
	Then $\langle\bb R,+_p,\oplus\rangle$ is a convex semilattice.
	We will always use these operations on $\bb R$ when speaking of $\bb R$ as a convex semilattice (or convex algebra or
	semilattice).

	We observe that translations
	\[
		T_a\Function{\bb R}{\bb R}{x}{x+a}\quad\text{where $a\in\bb R$},
	\]
	and rescalings 
	\[
		M_\lambda\Function{\bb R}{\bb R}{x}{\lambda x}\quad\text{where $\lambda>0$},
	\]
	are automorphisms of $\langle\bb R,+_p,\oplus\rangle$. 
	Furthermore, a subset $Y\subseteq\bb R$ is a subalgebra of the convex semilattice $\langle\bb R,+_p,\oplus\rangle$,
	if and only if $Y$ is an interval. 
\end{Example}

\noindent
We present two more standard examples of convex semilattices.
The first is simply powers of $\langle\bb R,+_p,\oplus\rangle$.

\begin{Example}
\label{S63}
	Let $\Omega\neq\emptyset$. The algebra $\langle\bb R,+_p,\oplus\rangle^\Omega$ has carrier set 
	$\bb R^\Omega$ and operations acting as
	\begin{Ilist}
	\item[]	\Dis{(f+_pg)(\omega)\DE pf(\omega)+(1-p)g(\omega)}\qquad for $x,y\in\bb R$, $p\in(0,1)$, $\omega\in\Omega$,
	\item[] \Dis{(f\oplus g)(\omega)\DE\max\{f(\omega),g(\omega)\}}\qquad for $x,y\in\bb R$, $\omega\in\Omega$.
	\end{Ilist}
	Due to this pointwise nature of the operations many properties of $\langle\bb R,+_p,\oplus\rangle$ are inherited. 
	For instance, all translations $T_a\DF f\mapsto f+a$ where $a\in\bb R^\Omega$ and rescalings $M_\lambda\DF f\mapsto\lambda f$ 
	where $\lambda>0$ are automorphisms.

	Clearly, each product of intervals is a subalgebra of $\langle\bb R^\Omega,+_p,\oplus\rangle$. 
	Contrasting the one-dimensional case, the subalgebra lattice of $\langle\bb R^\Omega,+_p,\oplus\rangle$ contains many more sets
	than boxes. For example, let $y=(y_\omega)_{\omega\in\Omega}\in\bb R^\Omega_{\Fin}$, and consider the 
	half-space\footnote{\Cref{S79}\Elistref{S22}.}  
	\[
		H_y\DE\big\{x\in\bb R^\Omega\DSb \llceil x,y\rrfloor\leq 0\big\}.
	\]
	Then $H_y$ is a subalgebra of the convex semilattice $\langle\bb R^\Omega,+_p,\oplus\rangle$, if and only if there exists at
	most one $\omega\in\Omega$ such that $y_\omega>0$. To prove this, assume first we have $\omega_1,\omega_2\in\Omega$ with 
	$\omega_1\neq\omega_2$ and $y_{\omega_1},y_{\omega_2}>0$. Set $x_1\DE 0$ and $x_2\DE(\xi_\omega)_{\omega\in\Omega}$ where 
	\[
		\xi_\omega\DE
		\begin{cases}
			-y_{\omega_2} \CAS \omega=\omega_1
			\\
			y_{\omega_1} \CAS \omega=\omega_2
			\\
			0 \CASO
		\end{cases}
	\]
	Then $x_1,x_2\in H_y$ but $x_1\oplus x_2\notin H_y$. Now assume that there exists at most one $\omega\in\Omega$ with 
	$y_\omega>0$. Let $\omega_0$ be the element with $y_{\omega_0}>0$ if one such exists, and pick $\omega_0$ arbitrary otherwise.
	It is clear that $H_y$ is convex. Let $x_1=(\xi_{1,\omega})_{\omega\in\Omega},x_2=(\xi_{2,\omega})_{\omega\in\Omega}\in H_y$,
	and assume w.l.o.g.\ that $\xi_{1,\omega_0}\leq\xi_{2,\omega_0}$. Then 
	\[
		\llceil x_1\oplus x_2,y\rrfloor=\sum_{\omega\in\Omega}\max\{\xi_{1,\omega},\xi_{2,\omega}\}\cdot y_\omega
		\leq\sum_{\omega\in\Omega}\xi_{2,\omega}\cdot y_\omega\leq 0.
	\]
	This example of subalgebras easily expands by applying translations. This leads to half-spaces 
	$H_{y,\gamma}\DE\{x\in\bb R^\Omega\DS \llceil x,y\rrfloor\leq\gamma\}$ where $\gamma\in\bb R$ and $y$ satisfies the stated
	condition.
\end{Example}

\noindent
To formulate the second example, we use the following notation: if $\langle X,+_p\rangle$ is a convex algebra and 
$M\subseteq X$, we denote the subalgebra of $X$ generated by $M$ as $\Conv M$ and speak of the \emph{convex hull} of $M$.
Explicitly, $\Conv M$ is given as 
\[
	\Conv M=\Big\{\sum_{i=1}^np_ix_i\DSB 
	n\in\bb N\setminus\{0\},p_1,\ldots,p_n\geq 0,\sum_{i=1}^np_i=1,x_1,\ldots,x_n\in M\Big\}.
\]

\begin{Example}
\label{S64}
	Let $\langle X,+_p\rangle$ be a convex algebra. We define\footnote{\Cref{S79}\Elistref{S86}}
	\[
		\Moewe X\DE\big\{\Conv M\DSb M\in\mc P_{\Fin}(X)\setminus\{\emptyset\}\big\},
	\]
	and endow $\Moewe X$ with the binary operations 
	\begin{Ilist}
	\item[] \Dis{A+_pB\DE\big\{a+_pb\DS a\in A,b\in B\}}\qquad for $A,B\in\Moewe X$, $p\in[0,1]$,
	\item[] \Dis{A\oplus B\DE\Conv(A\cup B)}\qquad for $A,B\in\Moewe X$. 
	\end{Ilist}
	Then $\langle\Moewe X,+_p,\oplus\rangle$ is a convex semilattice. 

	Explicitly, the operation $\oplus$ is given as 
	\[
		A\oplus B=\big\{a+_pb\DS a\in A,b\in B,p\in[0,1]\big\}=\bigcup_{p\in[0,1]}(A+_pB).
	\]
	The construction of $\langle\Moewe X,+_p,\oplus\rangle$ establishes a functor from convex algebras to convex
	semilattices in the obvious way: given convex algebras $\langle X,+_p\rangle$ and $\langle X',+_p'\rangle$ and 
	$\varphi\in\HomCA(X,X')$, define 
	\[
		\Moewe[12][8][5.9]\varphi\DF\left\{
		\begin{array}{rcl}
			\Moewe X & \to & \Moewe[17][8][8.33]{X'}
			\\
			A & \mapsto & \varphi(A)
		\end{array}
		\right.
	\]
\end{Example}

\begin{Remark}
\label{S96}
	Let $\bb X$ be a convex semilattice and $M\subseteq X$ a nonempty subset. Then the subalgebra of $\bb X$ generated by $M$ 
	is given as 
	\[
		\Big\{x_1\oplus\ldots\oplus x_N\DSB N\in\bb N,\ x_1,\ldots,x_N\in\Conv M\Big\}.
	\]
	This follows since this set is a subalgebra due to the distributive law.
\end{Remark}

\noindent
The free algebras in the varieties of convex algebras, semilattices, and convex semilattices are well-known, and we recall these
results.

\begin{Proposition}
\label{S13}
	Let $A$ be a nonempty set.
	\begin{Enumerate}
	\item The free convex algebra with basis $A$, we denote it as $\mc DA$, has carrier set 
		\[
			\mc DA\DE \Big\{(\alpha_a)_{a\in A}\in\bb R^A_{\Fin}\DSb 
			\forall a\in A\DP\alpha_a\geq 0,\ \sum_{a\in A}\alpha_a=1\Big\}
		\]
		and convex operations given by linear combinations in $\bb R^A$, i.e., 
		\[
			(\alpha_a)_{a\in A}+_p(\beta_a)_{a\in A}\DE\big(p\alpha_a+(1-p)\beta_a\big)_{a\in A}.
		\]
		Here we identify $A$ with\footnote{\Cref{S79}\Elistref{S17}.}
		$\{e_a\DS a\in A\}$ via $a\mapsto e_a$.
	\item The free semilattice with basis $A$, we denote it as $\mc SA$, has carrier set $\mc P_{\Fin}(A)\setminus\{\emptyset\}$,
		and the operation $\oplus$ given by set-theoretic union. Here we identify $A$ with $\{\{a\}\DS a\in A\}$
		via $a\mapsto\{a\}$.
	\item The free convex semilattice with basis $A$, we denote it as $\mc CA$, is equal to $\Moewe[21][-2][8.33]{\mc DA}$. 
		Here we identify $A$ with $\{\{e_a\}\DS a\in A\}$ via $a\mapsto\{e_a\}$. 
	\end{Enumerate}
\end{Proposition}

\noindent
Item (i) and (ii) of the above proposition are well-known, and (iii) is shown e.g.\ in \cite[Theorem~4]{bonchi.sokolova.vignudelli:2022}.

Let us add some comments concerning free algebras with a finite basis.

\begin{Remark}
\label{S65}
	Let $n\in\bb N$, $n\geq 1$. 
	\begin{Elist}
	\item Let $L$ be a linear space\footnote{\Cref{S79}\Elistref{S85}} 
		and let $\varphi\in\HomCA(\mc Dn,L)$. Then there exists a unique linear map 
		$\tilde\varphi\DF\bb R^n\to L$ with $\tilde\varphi|_{\mc Dn}=\varphi$. 
	\item We can find an isomorphic copy of $\mc Dn$ in $\bb R^{n-1}$. Let $\kappa_n\DF\bb R^n\to\bb R^{n-1}$ be the projection
		onto the first $n-1$ coordinates, i.e., 
		\[
			\kappa_n\big((\alpha_i)_{i=1}^n\big)\DE(\alpha_i)_{i=1}^{n-1}.
		\]
		This map is linear, surjective, and\footnote{\Cref{S79}\Elistref{S87}} 
		$\ker\kappa_n=\Span\{e_n\}$. Hence, $\kappa_n|_{\mc Dn}$ is an injective convex algebra 
		homomorphism of $\mc Dn$ into $\bb R^{n-1}$. Its image is
		\[
			\kappa_n(\mc Dn)=\Conv\big(\{0\}\cup\{e_1,\ldots,e_{n-1}\}\big)\subseteq\bb R^{n-1}.
		\]
		Note that $\kappa_n(\mc Dn)$ contains a nonempty open subset of $\bb R^{n-1}$.
	\item Being equal to $\Moewe[21][-2][8.33]{\mc Dn}$, the free convex semilattice $\mc Cn$ is a subalgebra of 
		$\Moewe[20][0][9.16]{\bb R^n}$. 
		We can also find an isomorphic copy of $\mc Cn$ in $\Moewe[35][-15][9.16]{\bb R^{n-1}}$,
		namely, via $\Moewe[18][2][6.25]{\kappa_n}|_{\mc Cn}$. 
		Note here that $\Moewe[18][2][6.25]{\kappa_n}|_{\mc Cn}$ is injective since $\kappa_n|_{\mc Dn}$ is injective.
	\end{Elist}
\end{Remark}

\subsection{The support function representation of $\Moewe[20][0][8.33]{\bb R^n}$}

We recall an indispensible geometric concept, see e.g.\ \cite{hug.weil:2020} or \cite{rockafellar:1970}.

\begin{Definition}
\label{S66}
	Let $A\neq\emptyset$ and let $K\subseteq\bb R^A$ be nonempty compact convex. 
	Then the \emph{support function} $\sigma(K,\Dummy)\DF\bb R^A_{\Fin}\to\bb R$ of the set $K$ is defined as
	\[
		\sigma(K,\omega)\DE\max\big\{\llceil x,\omega\rrfloor\DS x\in K\big\}\qquad\text{for }\omega\in\bb R^A_{\Fin}.
	\]
\end{Definition}

\noindent
The maximum in the definition of $\sigma(K,\omega)$ exists 
since\footnote{\Cref{S79}\Elistref{S82}} $\llceil\Dummy,\omega\rrfloor$ is continuous.

\begin{Definition}
\label{S32}
	Let $A\neq\emptyset$. Then we denote 
	\[
		\rho_A\DF\left\{
		\begin{array}{rcl}
			\Moewe[20][0][9.16]{\bb R^A} & \to & \bb R^{(\bb R^A_{\Fin})}
			\\[2mm]
			K & \mapsto & (\omega\mapsto\sigma(K,\omega))
		\end{array}
		\right.
	\]
	and speak of the \emph{support function representation} of $\Moewe[20][0][9.16]{\bb R^A}$. 

	If $A$ is finite, say $|A|\ED n$, we write $\rho_n$ instead of $\rho_A$.
\end{Definition}

\noindent
We list some properties of support functions that will be used throughout the paper without further notice.

\begin{Remark}
\label{S74}
	Let $A\neq\emptyset$.
	\begin{Elist}
	\item The support function representation $\rho_A$ is a convex semilattice homomorphism. This is checked by writing out the
		definitions of convex operations.
	\item Let $K$ be a nonempty compact convex subset of $\bb R^A$. 
		Then $K$ can be recovered from $\sigma(K,\Dummy)$ as
		\begin{equation}
		\label{S15}
			A=\bigcap_{\omega\in\bb R^A_{\Fin}}\big\{x\in\bb R^A\DS \llceil x,\omega\rrfloor\leq\sigma(K,\omega)\big\}.
		\end{equation}
		This follows by the Hahn-Banach separation theorem in the space $\bb R^A$.
		As a consequence, $\rho_A$ is injective.
	\item The function $\sigma(K,\Dummy)$ is positively homogeneous\footnote{\Cref{S79}\Elistref{S81}.}. 
		Hence, the intersection \cref{S15} may be taken over smaller sets. For example
		\begin{equation}
		\label{S88}
			A=\bigcap_{\substack{\omega=(\omega_a)_{a\in A}\in\bb R^A_{\Fin}\\ \sum_{a\in A}|\omega_a|^2=1}}
			\mkern-6mu\big\{x\in\bb R^A\DS \llceil x,\omega\rrfloor\leq\sigma(K,\omega)\big\}.
		\end{equation}
	\item Assume that $A$ is finite. Then \cref{S88} is the intersection over the unit sphere\footnote{\Cref{S79}\Elistref{S38}.} 
		$S^{|A|-1}$. The function $\llceil x,\Dummy\rrfloor$ is continuous and hence it is enough to take the intersection
		over any dense subset $T$ of the sphere. As a consequence, for any such set $T$, the function 
		$(\Dummy|_T)\circ\rho_A\DF\Moewe[20][0][9.16]{\bb R^A}\to\bb R^{(\bb R^T)}$ is an injective convex semilattice homomorphism.
	\item 
	\label{S89}
		Since $\mc CA$ is a subalgebra of $\Moewe[20][0][9.16]{\bb R^A}$, the above variants of the support function representation
		give rise to representations of $\mc CA$. Note that for finite $A$, say $|A|\ED n$, we can reduce the dimension by $1$ by
		precomposing with $\Moewe[18][2][6.25]{\kappa_n}$.
	\item Let $M\subseteq\bb R^A$ be finite and nonempty. Then 
		\[
			\sigma(\Conv M,\omega)=\max\big\{\llceil y,\omega\rrfloor\DS y\in M\big\}.
		\]
		The function $\sigma(\Conv M,\Dummy)$ is piecewise linear\footnote{\Cref{S79}\Elistref{S59}.}. To see this 
		use the cones and linear maps 
		\[
			C_y\DE\big\{\omega\in\bb R^A_{\Fin}\DS\sigma(\Conv M,\omega)=\llceil y,\omega\rrfloor\big\},\ 
			\xi_y\DE\llceil y,\Dummy\rrfloor\quad\text{where }y\in M.
		\]
	\end{Elist}
\end{Remark}

\noindent
Let us elaborate a bit more on the particular case of $\mc C2$. This algebra is significantly simpler than $\mc Cn$ for $n\geq 3$; the
reason being that $S^0$ is much different from $S^n$, $n\geq 1$. 

\begin{proposition}
\label{S72}
	\phantom{}
	\begin{Enumerate}
	\item There exists an injective convex semilattice homomorphism of $\mc C2$ into $\bb R^2$.
	\item If $|A|\geq 3$ there exists no injective convex semilattice homomorphism of $\mc CA$ into any finite power $\bb R^d$.
	\end{Enumerate}
\end{proposition}
\begin{proof}
	Item (i) is contained in \Cref{S74}\Elistref{S89}: the map 
	\begin{equation}
	\label{S73}
		\tau\DE\big(\Dummy|_{S^0}\circ\rho_1\circ\Moewe[18][2][6.25]{\kappa_2}\big)\big|_{\mc Cn}
	\end{equation}
	is an injective convex semilattice homomorphism of $\mc C2$ into $\bb R^{\{-1,1\}}$. 

	The idea for the proof of item (ii) is that there are too many different convex shapes in the plane. To make this
	precise, we use roots of unity
	\begin{equation}
	\label{S99}
		\xi_n^k\DE\binom{\cos\frac{2\pi k}n}{\sin\frac{2\pi k}n}
		\quad\text{for }n\in\bb N,n\geq 1,\ k\in\{0,\ldots,n-1\}.
	\end{equation}
	For $n\geq 2$ and $m\in\{0,\ldots,n-1\}$ set 
	\[
		A_n^m\DE\Conv\Big(\big\{\xi_n^k\DS k\in\{0,\ldots,n-1\}\big\}\cup\big\{\xi_{2n}^{2m+1}\}\Big).
	\]
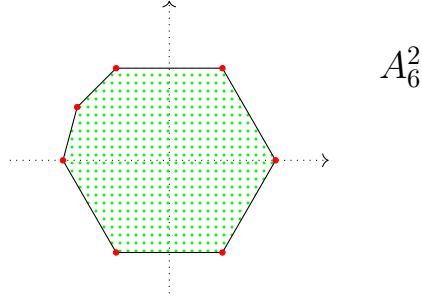
	\begin{center}
	\begin{tikzpicture}[x=1pt,y=1pt,scale=1,font=\fontsize{14}{14}]
		\draw[thin,dotted,->] (0,-50)--(0,60);
		\draw[thin,dotted,->] (-60,0)--(60,0);
		\draw[pattern=dots,pattern color=green,thin] 
			({40*cos(0)},{40*sin(0)})--({40*cos(60)},{40*sin(60)})--({40*cos(120)},{40*sin(120)})
			--({40*cos(150)},{40*sin(150)})
			--({40*cos(180)},{40*sin(180)})--({40*cos(240)},{40*sin(240)})--({40*cos(300)},{40*sin(300)})
			--({40*cos(360)},{40*sin(360)})
			;
		\draw[fill,red] ({40*cos(0)},{40*sin(0)}) circle [radius=1];
		\draw[fill,red] ({40*cos(60)},{40*sin(60)}) circle [radius=1];
		\draw[fill,red] ({40*cos(120)},{40*sin(120)}) circle [radius=1];
		\draw[fill,red] ({40*cos(150)},{40*sin(150)}) circle [radius=1];
		\draw[fill,red] ({40*cos(180)},{40*sin(180)}) circle [radius=1];
		\draw[fill,red] ({40*cos(240)},{40*sin(240)}) circle [radius=1];
		\draw[fill,red] ({40*cos(300)},{40*sin(300)}) circle [radius=1];
		\draw (87,35) node {$A_6^2$};
	\end{tikzpicture}
	\end{center}
	Then we have
	\[
		\sigma\big(A_n^m,\xi_{2n}^{2k+1}\big)=
		\begin{cases}
			1 \CAS k=m
			\\
			\sqrt{\frac 12(1+\cos\frac{2\pi}n)} \CAS k\in\{0,\ldots,n-1\}\setminus\{m\}
		\end{cases}
	\]
	Let $p_0,\ldots,p_{n-1}\geq 0$ with $\sum_{m=0}^{n-1}p_m=1$, then the above yields
	\[
		\sigma\Big(\sum_{m=0}^{n-1}p_mA_n^m,\xi_{2n}^{2k+1}\Big)=
		\sum_{m=0}^{n-1}p_n\sigma\big(A_n^m,\xi_{2n}^{2k+1}\big)=
		p_k+(1-p_k)\sqrt{\textstyle{\frac 12(1+\cos\frac{2\pi}n)}}.
	\]
	Thus the coefficient $p_k$ can be recovered from the set $\sum_{m=0}^{n-1}p_nA_n^m$. We conclude that every element of 
	$\Conv\{A_n^m\DS m\in\{0,\ldots,n-1\}\}$ has a unique representation as a convex combination $\sum_{m=0}^{n-1}p_mA_n^m$. In
	particular, the barycenter $\sum_{m=0}^{n-1}\frac 1nA_n^m$ is not contained in any set $\Conv\{A_n^m\DS m\in I\}$ where 
	$I\subsetneq\{0,\ldots,n-1\}$. This shows that
	\[
		\Conv\big\{A_n^m\DS m\in\{0,\ldots,n-1\}\big\}\supsetneq
		\bigcup_{I\subsetneq\{0,\ldots,n-1\}}\Conv\big\{A_n^m\DS m\in I\big\}.
	\]
	Now assume we have $d\in\bb N$ and $\Phi\in\HomCA(\mc C3,\bb R^d)$. Choose a bijective affine map 
	$\varphi\DF\bb R^2\to\bb R^2$ that maps the
	closed unit disk $\bb B\subseteq\bb R^2$ into $\kappa_3(\mc D3)$. Then we have the convex semilattice homomorphism
	$(\,\Moewe[18][2][6.25]{\kappa_3})^{-1}\circ\Moewe[12][8][5.9]\varphi\DF\Moewe[14][0][9.16]{\bb B}\to\mc C3$,
	and thus the convex algebra homomorphism $\lambda\DE\Phi\circ(\,\Moewe[18][2][6.25]{\kappa_3})^{-1}\circ\Moewe[12][8][5.9]\varphi$. 
	By Caratheodory's theorem 
	\[
		\Conv\big\{\lambda(A_{d+2}^m)\DS m\in\{0,\ldots,d+1\}\big\}=
		\bigcup_{\substack{I\subseteq\{0,\ldots,d+1\}\\ |I|=d+1}}\Conv\big\{\lambda(A_{d+2}^m)\DS m\in I\big\}.
	\]
	It follows that $\lambda$ is not injective, and in turn that $\Phi$ is not injective.

	The general case that $|A|\geq 3$ follows since $\mc C3$ embeds in $\mc CA$. 
\end{proof}

\begin{Remark}
\label{S68}
	Let us compute the action of the embedding \cref{S73}.
	An element $A\in\mc C2$ is a set of the form $A=\{pe_1+(1-p)e_2\DS p\in[a,b]\}$ with some $0\leq a\leq b\leq 1$. We have 
	\[
		\Moewe[18][2][6.25]{\kappa_2}(A)=[a,b]
		\quad\text{and}\quad
		[(\rho_1\circ\Moewe[18][2][6.25]{\kappa_2})(A)](\omega)=
		\begin{cases}
			b \CAS \omega=1,
			\\
			-a \CAS \omega=-1.
		\end{cases}
	\]
	Hence, we obtain the formula
	\[
		\tau(A)=
		\begin{pmatrix}
			[(\rho_1\circ\Moewe[18][2][6.25]{\kappa_2})(A)](-1)
			\\[2mm]
			[(\rho_1\circ\Moewe[18][2][6.25]{\kappa_2})(A)](1)
		\end{pmatrix}
		=\binom{-a}{b}
	\]
	which can be pictured as
	
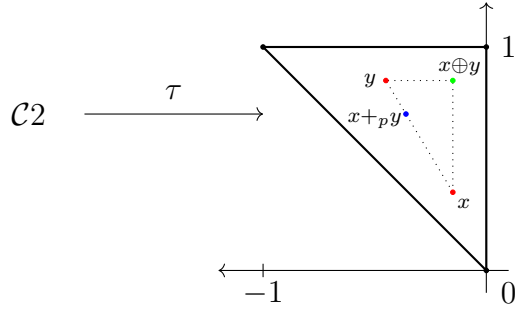
\begin{center}
	\begin{tikzpicture}[x=1.2pt,y=1.2pt,scale=0.7,font=\fontsize{12}{12}]
		\draw[->] (10,0) -- (-120,0);
		\draw[->] (0,-10) -- (0,120);
		\draw[thick] (0,0) -- (-100,100) -- (0,100) -- (0,0);
		\draw[fill] (0,100) circle [radius=1];
		\draw[fill] (0,0) circle [radius=1];
		\draw[fill] (-100,100) circle [radius=1];
		\draw (10,100) node {$1$};
		\draw (10,-10) node {$0$};
		\draw[thin] (-100,-3) -- (-100,3);
		\draw (-100,-10) node {$-1$};

		\draw[dotted] (-15,35) -- (-45,85);
		\draw[dotted] (-15,85) -- (-45,85);
		\draw[dotted] (-15,35) -- (-15,85);

		\draw[fill,color=green] (-15,85) circle [radius=1];
		\draw (-12,91) node {$\scriptstyle x\oplus y$};
		\draw[fill,color=red] (-15,35) circle [radius=1];
		\draw (-10,30) node {$\scriptstyle x$};
		\draw[fill,color=red] (-45,85) circle [radius=1];
		\draw (-52,85) node {$\scriptstyle y$};
		\draw[fill,color=blue] (-36,70) circle [radius=1];
		\draw (-50,68) node {$\scriptstyle x+_py$};

		\draw (-205,70) node {\large$\mc C2$};
		\draw (-140,80) node {$\tau$};
		\draw[->] (-180,70) -- (-100,70);
	\end{tikzpicture}
	\end{center}
\end{Remark}

\section{A description of kernels}

In this section we study the homomorphism set $\HomCSL(\mc CA,\bb R^\Omega)$, where $A$ and $\Omega$ are nonempty sets and 
$\bb R^\Omega$ is a convex semilattice as in \Cref{S63}. 
Our plan is to associate with every homomorphism $\Phi\DF\mc CA\to\bb R^\Omega$ a certain subset of $\bb R^A$, 
and show that this subset characterises inclusions of kernels, cf.\ \Cref{S3}.

\subsection{The set $H(\Phi)$}

\begin{Definition}
\label{S1}
	Let $A,\Omega\neq\emptyset$. For $\Phi\in\HomCSL(\mc CA,\bb R^\Omega)$ we define
	\begin{align*}
		& h_\Phi\DF\left\{
		\begin{array}{rcl}
			\Omega & \to & \bb R^A
			\\
			\omega & \mapsto & \big(\Phi(\{e_a\})(\omega)\big)_{a\in A}
		\end{array}
		\right.
		\\[2mm]
		& H(\Phi)\DE\ov{\bigcup_{\omega\in\Omega}\Cone\{h_\Phi(\omega)\}+\Span\{\E[A]\}}
		\\
		&\mkern48mu
		=\ov{\big\{\lambda h_\Phi(\omega)+\alpha\E[A]\DS \lambda\geq 0,\omega\in\Omega,\alpha\in\bb R\big\}}
		\subseteq\bb R^A.
	\end{align*}
	Moreover, denote
	\[
		H(A,\Omega)\DE\big\{H(\Phi)\DSb \Phi\in\HomCSL(\mc CA,\bb R^\Omega)\big\}\subseteq\mc P(\bb R^A).
	\]
\end{Definition}

\begin{Example}
\label{S83}
	\phantom{}
	\begin{Enumerate}
	\item For every constant map $\Phi\DF\mc CA\to\bb R^\Omega$ we have $H(\Phi)=\Span\{\E[A]\}$. 
		To see this write $\Phi(\{e_a\})=(y_\omega)_{\omega\in\Omega}$. Then 
		$h_\phi(\omega)=y_\omega\E[A]$. 
	\item For the support function representation $\rho_A\in\HomCSL(\mc CA,\bb R_{\Fin}^A)$ we have 
		$H(\rho_A)=\bb R_{\Fin}^A+\Span\{\E[A]\}$.
		To see this note that for each $\omega=(\omega_a)_{a\in A}\in\bb R_{\Fin}^A$ we have
		\[
			\sigma(\{e_a\},\omega)=\llceil e_a,\omega\rrfloor=\omega_a.
		\]
		Thus 
		\[
			h_\Phi(\omega)=\big(\rho_A(\{e_a\})(\omega)\big)_{a\in A}=(\omega_a)_{a\in A}=\omega.
		\]
		In particular, if $A$ is finite, we have $H(\rho_A)=\bb R^A$.
	\end{Enumerate}
\end{Example}

\noindent
By inspecting the definition of the sets $H(\Phi)$ it is easy to give an implicit yet useful description of the set $H(A,\Omega)$.

\begin{lemma}
\label{S12}
	Let $A,\Omega\neq\emptyset$. Then 
	\begin{equation}
	\label{S90}
		H(A,\Omega)=\Big\{
		\ov{\bigcup_{\omega\in\Omega}\Cone\{y_\omega\}+\Span\{\E[A]\}}
		\DSB (y_\omega)_{\omega\in\Omega}\in(\bb R^A)^\Omega
		\Big\}.
	\end{equation}
\end{lemma}
\begin{proof}
	The inclusion ``$\subseteq$'' holds by definition.
	Assume we are given $y_\omega\in\bb R^A$ for $\omega\in\Omega$, and write $y_\omega=(\eta_{\omega,a})_{a\in A}$. 
	Define a function $\varphi\DF A\to\bb R^\Omega$ as
	\[
		\varphi(a)\DE (\eta_{\omega,a})_{\omega\in\Omega},
	\]
	and let $\Phi\in\HomCSL(\mc CA,\bb R^\Omega)$ be the unique homomorphism with $\Phi(\{e_a\})=\varphi(a)$, $a\in A$. Then 
	$h_\Phi(\omega)=y_\omega$, and hence 
	\[
		H(\Phi)=\ov{\bigcup_{\omega\in\Omega}\Cone\{y_\omega\}+\Span\{\E[A]\}}.
	\]
\end{proof}

\noindent
If $\Omega$ is finite or $A$ has at most two elements, it is not necessary to take the closure in \cref{S90}. 

\begin{lemma}
\label{S2}
	\phantom{}
	\begin{Enumerate}
	\item \phantom{}\\[-15.3mm]
		\begin{multline*}
			\mkern-15mu
			\forall m\in\bb N\setminus\{0\}\DP 
			\\
			H(A,m)=
			\Big\{\bigcup_{i=1}^m\Cone\{y_i\}+\Span\{\E[A]\}\DSB y_1,\ldots,y_m\in\bb R^A\Big\}.
		\end{multline*}
	\item \Dis{%
		\forall\Omega\neq\emptyset\DP H(1,\Omega)=\{\bb R\}
		}.
	\item \Dis{%
			H(2,1)=\Big\{
			\textstyle{\Span\{\E[2]\},\big\{\binom\alpha\beta\in\bb R^2\DSb \alpha\leq\beta\big\},
				\big\{\binom\alpha\beta\in\bb R^2\DSb\alpha\geq\beta\big\}}\Big\}.
		}
	\item \phantom{}\\[-14.3mm]
		\begin{multline*}
			\mkern-15mu
			\forall |\Omega|\geq 2\DP
			\\
			H(2,\Omega)=\Big\{
			\textstyle{\Span\{\E[2]\},\big\{\binom\alpha\beta\in\bb R^2\DSb \alpha\leq\beta\big\},
				\big\{\binom\alpha\beta\in\bb R^2\DSb\alpha\geq\beta\big\},\bb R^2}\Big\}.
		\end{multline*}
	\end{Enumerate}
\end{lemma}
\begin{proof}
	As a first step we show that for each $y\in\bb R^A$ the set $\Cone\{y\}+\Span\{\E[A]\}$ is closed. 
	If $y\in\Span\{\E[A]\}$ we have $\Cone\{y\}+\Span\{\E[A]\}=\Span\{\E[A]\}$ and this is closed. 
	If $y\notin\Span\{\E[A]\}$, then $\{y,\E[A]\}$ is linearly independent, and hence the map 
	\[
		\psi\DF\left\{
		\begin{array}{rcl}
			\bb R^2 & \to & \bb R^A
			\\
			\binom\alpha\beta & \mapsto & \alpha y+\beta\E[A]
		\end{array}
		\right.
	\]
	is a homeomorphism onto its image. We have 
	\[
		\Cone\{y\}+\Span\{\E[A]\}=\psi\Big(\Big\{\binom\alpha\beta\in\bb R^2\DSB\alpha\geq 0\Big\}\Big),
	\]
	and hence $\Cone\{y\}+\Span\{\E[A]\}$ is a closed subset of $\Span\{y,\E[A]\}$. This linear span is closed in $\bb R^A$, and
	hence also $\Cone\{y\}+\Span\{\E[A]\}$ is closed in $\bb R^A$.

	Item (i) can now be deduced easily: every finite union 
	\[
		\bigcup_{i=1}^m\Big(\Cone\{y_i\}+\Span\{\E[A]\}\Big)
	\]
	is closed, and if $|\Omega|<\infty$ finite unions exhaust $H(A,\Omega)$. 

	Item (ii) follows from \Cref{S83}(i), since $\mc C1$ has only one element and hence every homomorphism 
	$\Phi\DF\mc C1\to\bb R^\Omega$ is constant. 

	To prove (iii) and (iv) we again refer to \Cref{S12}. 
	The four sets written in the assertion occur depending on the following case distinctions
	(here we write $y_\omega=(\eta_{\omega,j)})_{j=1}^2$):
	\[
		\begin{cases}
			\Span\{\E[2]\} \CAS \forall\omega\in\Omega\DP \eta_{\omega,1}=\eta_{\omega,2}
			\\
			\big\{\binom\alpha\beta\in\bb R^2\DSb \alpha\leq\beta\big\} \CAS 
			\forall\omega\in\Omega\DP \eta_{\omega,1}\leq\eta_{\omega,2}\wedge
			\exists\omega\in\Omega\DP \eta_{\omega,1}<\eta_{\omega,2}
			\\
			\big\{\binom\alpha\beta\in\bb R^2\DSb \alpha\geq\beta\big\} \CAS 
			\forall\omega\in\Omega\DP \eta_{\omega,1}\geq\eta_{\omega,2}\wedge
			\exists\omega\in\Omega\DP \eta_{\omega,1}>\eta_{\omega,2}
			\\
			\bb R^2 \CAS
			\exists\omega\in\Omega\DP \eta_{\omega,1}<\eta_{\omega,2}\wedge
			\exists\omega\in\Omega\DP \eta_{\omega,1}>\eta_{\omega,2}
		\end{cases}
	\]
	If $|\Omega|=1$, the last case cannot occur.
\end{proof}

\noindent
Next we show a monotonicity property.

\begin{lemma}
\label{S84}
	Let $A\neq\emptyset$. 
	\begin{Enumerate}
	\item If $\Omega,\Omega'\neq\emptyset$ and $|\Omega|\leq|\Omega'|$, then $H(A,\Omega)\subseteq H(A,\Omega')$.
	\item Assume that $A$ is infinite or has at most two elements. If $|\Omega|\geq|A|$, then $H(A,\Omega)=H(A,A)$.
	\item Assume that $A$ has at least three elements. Then 
		\[
			H(A,1)\subsetneq H(A,2)\subsetneq H(A,3)\subsetneq\cdots\subsetneq H(A,\bb N).
		\]
	\end{Enumerate}
\end{lemma}
\begin{proof}
	Item (i) is easy to see. If $\lambda\DF\Omega\to\Omega'$ is injective and $(y_\omega)_{\omega\in\Omega}\in(\bb R^A)^\Omega$, 
	we set 
	\[
		y'_{\omega'}\DE
		\begin{cases}
			y_\omega \CAS \omega'=\lambda(\omega),
			\\
			0 \CASO.
		\end{cases}
	\]
	Then
	\[
		\ov{\bigcup_{\omega\in\Omega}\Cone\{y_\omega\}+\Span\{\E[A]\}}=
		\ov{\bigcup_{\omega'\in\Omega'}\Cone\{y'_{\omega'}\}+\Span\{\E[A]\}}.
	\]
	We come to the proof of (ii). Since $A$ is infinite, the product topology of $\bb R^A$ has a base of cardinality $|A|$. 
	Now assume we have $\Omega$ with $|\Omega|\geq|A|$. Then, by (i), $H(A,\Omega)\supseteq H(A,A)$. Given $H\in H(A,\Omega)$ we
	choose a set $\{y_a\DS a\in A\}\subseteq H$ that is dense in $H$. Since $H$ is closed, invariant under taking nonnegative
	multiples, and invariant under adding scalar multiples of $\E[A]$, we have 
	\[
		H=\ov{\bigcup_{a\in A}\Cone\{y_a\}+\Span\{\E[A]\}}.
	\]
	\Cref{S12} implies that $H\in H(A,A)$.
	Item (iii) follows from \Cref{S2}(i).
\end{proof}

\noindent
If $\Omega$ is infinite and $A$ has at least three elements the situation is more complex as in \Cref{S2}. Still
one can obtain an illustrative description provided that $A$ is finite.

\begin{lemma}
\label{S91}
	Assume that $3\leq|A|<\infty$ and $\Omega$ is infinite. Then there exists an order preserving bijection of $H(A,\Omega)$
	onto the set of all closed subsets of $S^{|A|-2}$.
\end{lemma}
\begin{proof}
	Set $n\DE|A|$ and $L\DE\{(\alpha_j)_{j=1}^n\in\bb R^n\DS \alpha_n=0\}$. For each homomorphism $\Phi$ the set $H(\Phi)$ is
	closed under addition of scalar multiples of $\E[n]$ and under multiplication with nonnegative constants. Thus we can recover
	$H(\Phi)$ from the set $L\cap S^{n-1}\cap H(\Phi)$ by means of the formula
	\[
		H(\Phi)=
		\begin{cases}
			\Big[\bigcup_{\lambda\geq 0}\lambda\big(L\cap S^{n-1}\cap H(\Phi)\big)\Big]+\Span\{\E[n]\} \CAS 
			L\cap S^{n-1}\cap H(\Phi)\neq\emptyset
			\\[2mm]
			\Span\{\E[n]\} \CAS L\cap S^{n-1}\cap H(\Phi)=\emptyset
		\end{cases}
	\]
	The map $\Upsilon\DF H(\Phi)\mapsto L\cap S^{n-1}\cap H(\Phi)$ is thus injective. We can identify $L$ with $\bb R^{n-1}$ and 
	$L\cap S^{n-1}$ with $S^{n-2}$. Hence, $\Upsilon$ is an injective and order preserving map of $H(A,\Omega)$ into the set of all
	closed subsets of $S^{n-2}$. 

	Let $H$ be a closed subset of $S^{n-2}$. If $H=\emptyset$, we have $H=\Upsilon(H(\Phi))$ when $\Phi\DF\mc CA\to\bb R^\Omega$ 
	is any constant map. Assume that $H\neq\emptyset$. Consider an element 
	\[
		a=(\alpha_j)_{j=1}^n\in\bigcup_{\lambda\geq 0}\lambda H+\Span\{\E[n]\}.
	\]
	If $a\notin\Span\{\E[n]\}$ then $a$ can be represented as 
	\[
		a=\|a-\alpha_n\E[n]\|\cdot\frac{a-\alpha_n\E[n]}{\|a-\alpha_n\E[n]\|}+\alpha_n\E[n],
	\]
	and it follows that $\frac{a-\alpha_n\E[n]}{\|a-\alpha_n\E[n]\|}\in H$. Since $H$ is closed it follows that 
	$\bigcup_{\lambda\geq 0}\lambda H+\Span\{\E[n]\}$ is closed.

	Choose a countable subset of $\Omega$ and enumerate it as $\omega_1,\omega_2,\ldots$. 
	Choose a countable dense subset of $H$ and enumerate it as $y_1,y_2,\ldots$. According to \Cref{S12} we have 
	\[
		\ov{\bigcup_{k=1}^\infty\Cone\{y_k\}+\Span\{\E[A]\}}\in H(A,\Omega).
	\]
	Clearly, 
	\[
		H\subseteq\ov{\bigcup_{k=1}^\infty\Cone\{y_k\}+\Span\{\E[A]\}},
	\]
	and
	\[
		\bigcup_{k=1}^\infty\Cone\{y_k\}+\Span\{\E[A]\}\subseteq
		\bigcup_{\lambda\geq 0}\lambda H+\Span\{\E[n]\},
	\]
	and we conclude 
	\[
		\ov{\bigcup_{k=1}^\infty\Cone\{y_k\}+\Span\{\E[A]\}}=\bigcup_{\lambda\geq 0}\lambda H+\Span\{\E[n]\}.
	\]
\end{proof}

\subsection{Relating $\ker\Phi$ to $H(\Phi)$}

The next statement is the main result of this section and provides the basis for all what follows. It is a comparative result and 
tells us how to get hands on $\ker\Phi$ using only the action of $\Phi$ on the basis $\{\{e_a\}\DS a\in A\}$ of $\mc CA$.

\begin{Theorem}
\label{S3}
	Let $A,\Omega,\Omega'\neq\emptyset$ and let $\Phi\in\HomCSL(\mc CA,\bb R^\Omega)$ and $\Phi'\in\HomCSL(\mc CA,\bb R^{\Omega'})$. 
	Then 
	\[
		\ker\Phi\subseteq\ker\Phi'
		\ \Leftrightarrow\ 
		H(\Phi)\supseteq H(\Phi')
	\]
\end{Theorem}

\noindent
In order to prove this theorem we present a series of lemmata. The first one explains the significance of the function $h_\Phi$. 

\begin{lemma}
\label{S4}
	Let $\Phi\in\HomCSL(\mc CA,\bb R^\Omega)$, $s\in\bb N$, and $p_1,\ldots,p_s\in\mc DA$. Then
	\[
		\forall\omega\in\Omega\DP
		\Phi\big(\Conv\{p_1,\ldots,p_s\}\big)(\omega)
		=\max\big\{\llceil h_\Phi(\omega),p_i\rrfloor\DS i\in\{1,\ldots,s\}\big\}.
	\]
\end{lemma}
\begin{proof}
	Let $p\in\mc DA$ and write $p=(\alpha_a)_{a\in A}$. Then
	\begin{align*}
		\Phi\big(\{p\}\big)(\omega)= &\, 
		\Phi\Big(\Big\{\sum_{\substack{a\in A\\ \alpha_a\neq 0}}\alpha_ae_a\Big\}\Big)(\omega)
		=\Phi\Big(\sum_{\substack{a\in A\\ \alpha_a\neq 0}}\alpha_a\{e_a\}\Big)(\omega)
		\\
		= &\, 
		\sum_{\substack{a\in A\\ \alpha_a\neq 0}}\alpha_a\Phi\big(\{e_a\}\big)(\omega)=\llceil h_\Phi(\omega),p\rrfloor
		.
	\end{align*}
	If $p_1,\ldots,p_s\in\mc DA$, then 
	\begin{align*}
		\Phi\big(\Conv\{p_1,\ldots,p_s\}\big)= &\, \Phi\Big(\bigoplus_{i=1}^s\{p_i\}\Big)
		\\
		= &\, \bigoplus_{i=1}^s\Phi\big(\{p_i\}\big)=\max\big\{\Phi(\{p_i\})\DS i\in\{1,\ldots,s\}\big\},
	\end{align*}
	and the assertion follows. 
\end{proof}

\noindent
In order to relate $H(\Phi)$ to $\ker\Phi$ we use a man in the middle.

\begin{Definition}
\label{S5}
	Let $p_1,\ldots,p_s,p\in\bb R^A_{\Fin}$. Then we set 
	\[
		C(p_1,\ldots,p_s;p)\DE
		\big\{x\in\bb R^A\DS\forall j\in\{1,\ldots,s\}\DP\llceil x,p_j\rrfloor<\llceil x,p\rrfloor\big\}.
	\]
\end{Definition}

\noindent
We list some properties of the sets $C(p_1,\ldots,p_s;p)$.

\begin{Remark}
\label{S77}
	Let $p_1,\ldots,p_s,p\in\bb R^A_{\Fin}$. Then the following statements hold.
	\begin{Enumerate}
	\item \Dis{%
		\begin{aligned}[t]
			C(p_1,\ldots,p_s;p)= &\, 
			\big\{x\in\bb R^A\DS\forall y\in\Conv\{p_1,\ldots,p_s\}\DP\llceil x,y\rrfloor<\llceil x,p\rrfloor\big\}
			\\
			= &\, 
			\big\{x\in\bb R^A\DS\forall j\in\{1,\ldots,s\}\DP\llceil x,p_j-p\rrfloor<0\big\}.
		\end{aligned}
		}
	\item $0\notin C(p_1,\ldots,p_s,p)$.
	\item $C(p_1,\ldots,p_s,p)$ is open.
	\item \Dis{%
			\forall x\in\bb R^A,\lambda>0\DP 
			\Big(x\in C(p_1,\ldots,p_s,p)\Leftrightarrow \lambda x\in C(p_1,\ldots,p_s,p)\Big).
		}
	\item If $p_1,\ldots,p_s,p\in\mc DA$, then 
		\[
			\forall x\in\bb R^A,\alpha\in\bb R\DP 
			\Big(x\in C(p_1,\ldots,p_s,p)\Leftrightarrow x+\alpha\E[A]\in C(p_1,\ldots,p_s,p)\Big).
		\]
	\item \Dis{%
			C(p_1,\ldots,p_s;p)\cap\bb R^A_{\Fin}\neq\emptyset.
			\ \Leftrightarrow\ 
			C(p_1,\ldots,p_s;p)\neq\emptyset
			\ \Leftrightarrow\ 
			p\notin\Conv\{p_1,\ldots,p_s\}
		}
	\end{Enumerate}
	Only the last item requires an argument; it is seen by applying the Hahn-Banach separation theorem in the space $\bb R^A$.
\end{Remark}

\noindent
It is not difficult to relate certain elements of $\ker\Phi$ to $H(\Phi)$ using the sets $C(p_1,\ldots,p_s;p)$.

\begin{lemma}
\label{S6}
	Let $\Phi\in\HomCSL(\mc CA,\bb R^\Omega)$ and $p_1,\ldots,p_s,p\in\mc DA$. Then
	\begin{align*}
		\Phi\big(\Conv\{p_1,\ldots,p_s\}\big)=\Phi\big(\Conv\{p_1,\ldots,p_s,p\}\big)
		&\ \Leftrightarrow\ 
		\\
		C(p_1,\ldots,p_s;p)\cap h_\Phi(\Omega)=\emptyset
		&\ \Leftrightarrow\ 
		C(p_1,\ldots,p_s;p)\cap H(\Phi)=\emptyset.
	\end{align*}
\end{lemma}
\begin{proof}
	By \Cref{S4} we have 
	\begin{align*}
		\Phi\big(\Conv\{p_1,\ldots,p_s\}\big)(\omega)= &\, 
		\max\big\{\llceil h_\Phi(\omega),p_i\rrfloor\DS i=1,\ldots,s\},
		\\
		\Phi\big(\Conv\{p_1,\ldots,p_s,p\}\big)(\omega)= &\, 
		\max\Big(\big\{\llceil h_\Phi(\omega),p_i\rrfloor\DS i=1,\ldots,s\}\cup\{\llceil h_\Phi(\omega),p\rrfloor\}\Big).
	\end{align*}
	Thus
	\begin{align*}
		\Phi\big(\Conv\{p_1,\ldots,p_s\}\big)= \mkern2.5mu&\, \Phi\big(\Conv\{p_1,\ldots,p_s,p\}\big)
		\\
		\Leftrightarrow
		&\, \forall\omega\in\Omega\DQ\exists i\in\{1,\ldots,s\}\DP 
		\llceil h_\Phi(\omega),p_i\rrfloor\geq\llceil h_\Phi(\omega),p\rrfloor
		\\
		\Leftrightarrow
		&\, C(p_1,\ldots,p_s;p)\cap h_\Phi(\Omega)=\emptyset.
	\end{align*}
	The backwards implication in the second equivalence asserted in the lemma is trivial. 
	To show the forward implication assume that $C(p_1,\ldots,p_s;p)\cap H(\Phi)\neq\emptyset$. 
	We use the properties listed in \Cref{S77}. Since $C(p_1,\ldots,p_s;p)$ is open, it follows that 
	\[
		C(p_1,\ldots,p_s;p)\cap
		\big\{\lambda h_\Phi(\omega)+\alpha\E[A]\DS \lambda\geq 0,\omega\in\Omega,\alpha\in\bb R\big\}\neq\emptyset.
	\]
	Choose $\lambda\geq 0$, $\omega\in\Omega$, $\alpha\in\bb R$ such that 
	$\lambda h_\Phi(\omega)+\alpha\E[A]\in C(p_1,\ldots,p_s;p)$. Then also $\lambda h_\Phi(\omega)\in C(p_1,\ldots,p_s;p)$.
	Hence, $\lambda>0$, and it follows that $h_\Phi(\omega)\in C(p_1,\ldots,p_s;p)$.
\end{proof}

\noindent
It is a more involved fact that $H(\Phi)$ can be fully described using the sets $C(p_1,\ldots,p_s;p)$. 

\begin{Proposition}
\label{S8}
	Let $\Phi\in\HomCSL(\mc CA,\bb R^\Omega)$. Then 
	\begin{multline*}
		\bb R^A\setminus H(\Phi)=\bigcup\Big\{C(p_1,\ldots,p_{n-1};p)\,\Big|
		\\
		p_1,\ldots,p_{n-1},p\in\mc DA,C(p_1,\ldots,p_{n-1};p)\cap H(\Phi)=\emptyset\Big\}.
	\end{multline*}
\end{Proposition}

\noindent
The proof of this statement relies on an elementary geometric fact.

\begin{lemma}
\label{S7}
	Let $d\in\bb N$ and let $L$ be a linear space with $\dim L=d$. Moreover, let $(\Dummy,\Dummy)$ be a scalar
	product on $L$ and denote by $\|\Dummy\|$ the norm induced by $(\Dummy,\Dummy)$. 
	Assume we have $x_0\in L\setminus\{0\}$ and $\varepsilon>0$. Then there exist linearly independent vectors 
	$c_1,\ldots,c_d\in L$ such that, with the notation
	\begin{align*}
		& C\DE\big\{x\in L\DS\forall j\in\{1,\ldots,d\}\DP(x,c_j)<0\big\},
		\\
		& U^L_\varepsilon(x_0)\DE\big\{x\in L\DS \|x-x_0\|<\varepsilon\big\},
	\end{align*}
	it holds that 
	\begin{equation}
	\label{S14}
		x_0\in C\subseteq\big\{\lambda x\DS \lambda>0,x\in U^L_\varepsilon(x_0)\big\}.
	\end{equation}
\end{lemma}

\noindent
For the sake of completeness we provide a proof.

\begin{proof}
	The assertion is invariant under applying unitary linear maps and positive multiples of the identity. 
	Hence, we may assume w.l.o.g.\ that 
	$L=\bb R^d$, that $(\Dummy,\Dummy)$ is the Euclidean scalar product, and that $x_0=\E[d]$.

	Let $\rho>0$ be a parameter (which will be specified later) and set 
	\[
		f_j\DE \E[d]+\rho e_j\text{ for }j\in\{1,\ldots,d\}.
	\]
	Then $\{f_1,\ldots,f_d\}$ is linearly independent: To see this, let $\lambda_1,\ldots,\lambda_d\in\bb R$ and
	compute, denoting $\mu\DE\sum_{j=1}^d\lambda_j$, 
	\[
		\sum_{j=1}^d\lambda_jf_j=\mu \E[d]+\sum_{j=1}^d\lambda_j\rho e_j=\sum_{j=1}^d[\mu+\rho\lambda_j]e_j.
	\]
	If this linear combination vanishes, then $\mu+\rho\lambda_j=0$ for all $j$. Summing over $j$ yields $\mu(d+\rho)=0$,
	hence $\mu=0$ and in turn $\lambda_j=0$ for all $j$. 

	Let $c_l$ be the unique solution of the system of linear equations
	\[
		(c_l,f_j)=
		\begin{cases}
			-1 \CAS l=j
			\\
			0 \CASO
		\end{cases}
	\]
	For any $\lambda_1,\ldots,\lambda_d\in\bb R$ it holds that 
	\[
		\Big(\sum_{j=1}^d\lambda_jc_j,f_l\Big)=-\lambda_l.
	\]
	In particular, $\{c_1,\ldots,c_d\}$ is linearly independent. 

	We check that for all sufficiently small values of $\rho$ the properties \cref{S14} hold.
	\begin{Ilist}
	\item 
		We have $\E[d]=\frac 1{d+\rho}\sum_{j=1}^df_j$, and hence 
		\[
			\forall l\in\{1,\ldots,d\}\DP (\E[d],c_l)=-\frac 1{d+\rho}.
		\]
		In particular, $\E[d]\in C$.
	\item 
		We have 
		\[
			\forall x\in\bb R^d\DP x=-\sum_{j=1}^d(x,c_j)f_j.
		\]
		In particular, $C\subseteq\Cone\{f_1,\ldots,f_d\}$. 

		We have $\lim_{\rho\to 0}f_j=\E[d]$ for all $j\in\{1,\ldots,d\}$. Hence, for all sufficiently small 
		$\rho>0$ all $f_j$ belong to $U^L_\varepsilon(\E[d])$. 
		Given $x\in C$, we set $\lambda\DE-\sum_{j=1}^d(x,c_j)$.
		Then $\lambda>0$ and 
		\[
			x=\lambda\cdot\sum_{j=1}^d\frac{-(x,c_j)}\lambda f_j.
		\]
		The sum on the right side is a convex combination of the $f_j$, and hence belongs to 
		$U^L_\varepsilon(\E[d])$. 
	\end{Ilist}
\end{proof}

\begin{proof}[Proof of \Cref{S8}]
	The inclusion ``$\supseteq$'' is obvious; we have to prove the reverse inclusion. 
	Assume we have $x_0\in\bb R^A\setminus H(\Phi)$. 
	Write $x_0=(\xi_{0,a})_{a\in A}$. Since $H(\Phi)$ is closed in the product topology, 
	we find $a_1,\ldots,a_n\in A$ and $\varepsilon>0$ such that the open neighbourhood 
	($\pi_a\DF\bb R^A\to\bb R$ denotes the projection on the $a$-th coordinate)
	\[
		U\DE\bigcap_{i=1}^n\pi_{a_i}^{-1}\big((\xi_{0,a_i}-\varepsilon,\xi_{0,a_i}+\varepsilon)\big)
	\]
	of $x_0$ satisfies $U\cap H(\Phi)=\emptyset$. 

	Denote by $P$ be the restriction map 
	\[
		P\DF\left\{
		\begin{array}{rcl}
			\bb R^A & \to & \bb R^n
			\\
			(\xi_a)_{a\in A} & \mapsto & (\xi_{a_i})_{i=1}^n
		\end{array}
		\right.
	\]
	and observe that 
	\begin{equation}
	\label{S75}
		\forall x,x'\in\bb R^A\DP\Big(Px=Px'\Rightarrow \big(x\in U\Leftrightarrow x'\in U\big)\Big).
	\end{equation}
	Moreover, note that $P\E[A]=\E[n]$. Set 
	\[
		L\DE\Big\{(\alpha_i)_{i=1}^n\in\bb R^n\DSB\sum_{i=1}^n\alpha_i=0\Big\}=\big\{\E[n]\big\}^\perp,
	\]
	and let $P_L\DF\bb R^n\to L$ be the orthogonal projection onto $L$. Explicitly, $P_L$ acts as 
	\[
		P_L\big((\alpha_i)_{i=1}^n\big)=(\alpha_i)_{i=1}^n-\Big(\frac 1n\sum_{i=1}^n\alpha_i\Big)\E[n]
		\qquad\text{for }(\alpha_i)_{i=1}^n\in\bb R^n.
	\]
	Consider now the linear space $L$ endowed with the scalar product inherited from $\bb R^n$. We have 
	$x_0\in U$ and $\Span\{\E[A]\}\cap U=\emptyset$, and hence \cref{S75} implies that $Px_0\notin\Span\{\E[n]\}$. 
	Thus $(P_L\circ P)(x_0)\in L\setminus\{0\}$.
	
	The set 
	\[
		P(U)=\prod_{i=1}^n(\xi_{0,a_i}-\varepsilon,\xi_{0,a_i}+\varepsilon)
	\]
	is open in $\bb R^n$, and $P_L$ is an open map. Hence, $(P_L\circ P)(U)$ is an open neighbourhood of $(P_L\circ P)(x_0)$
	in $L$, and we find $\varepsilon>0$ such that 
	\[
		U^L_\varepsilon\big((P_L\circ P)(x_0)\big)\subseteq(P_L\circ P)(U).
	\]
	\Cref{S7} provides us with elements $c_1,\ldots,c_{n-1}\in L$ such that, with the notation 
	\[
		C\DE\big\{y\in L\DS \forall j\in\{1,\ldots,n-1\}\DP (y,c_j)<0\big\},
	\]
	it holds that
	\begin{align*}
		(P_L\circ P)(x_0)\in C\subseteq &\, 
		\big\{\mu y\DS \mu>0,y\in U^L_\varepsilon\big((P_L\circ P)(x_0)\big)\big\}
		\\
		\subseteq &\, \big\{\mu y\DS \mu>0,y\in(P_L\circ P)(U)\big\}.
	\end{align*}
	Write $c_j=(\gamma_{j,i})_{i=1}^n$ and set 
	\[
		\eta_{j,a}\DE
		\begin{cases}
			\frac 1n+\rho\gamma_{j,i}\CAS a=a_i,
			\\
			0 \CASO,
		\end{cases}
		\qquad
		\eta_a\DE
		\begin{cases}
			\frac 1n\CAS a=a_i,
			\\
			0 \CASO,
		\end{cases}
	\]
	where $\rho>0$ is sufficiently small so that all $\eta_{j,a_i}$ are positive. Then 
	\[
		p_j\DE(\eta_{j,a})_{a\in A},p\DE(\eta_a)_{a\in A}\in\mc DA.
	\]
	For each $x=(\xi_a)_{a\in A}\in\bb R^A$ we have
	\begin{multline*}
		x\in C(p_1,\ldots,p_{n-1};p)
		\ \Leftrightarrow\ \forall j\in\{1,\ldots,n-1\}\DP
		\mkern-15mu\underbrace{\sum_{a\in A}\xi_a(\eta_{j,a}-\eta_a)}_{
		=\rho\sum_{i=1}^n\xi_{a_i}\gamma_{j,i}=\rho(Px,c_j)}\mkern-15mu<0
		\\[2mm]
		\Leftrightarrow\ \forall j\in\{1,\ldots,n-1\}\DP
		\mkern-40mu\underbrace{(Px,c_j)}_{{=(Px,P_Lc_j)=(P_LPx,c_j)}}\mkern-40mu<0
		\ \Leftrightarrow\ (P_L\circ P)(x)\in C.
	\end{multline*}
	In particular, $x_0\in C(p_1,\ldots,p_{n-1};p)$. 

	Now assume towards a contradiction that $C(p_1,\ldots,p_{n-1};p)\cap H(\Phi)\neq\emptyset$. 
	Since $C(p_1,\ldots,p_{n-1};p)$ is open and does not intersect $\Span\{\E[n]\}$, 
	we find $\lambda>0,\omega\in\Omega,\alpha\in\bb R$ such that 
	\[
		x\DE\lambda h_\Phi(\omega)+\alpha\E[A]\in C(p_1,\ldots,p_{n-1};p).
	\]
	This yields that $(P_L\circ P)(x)\in C$, and hence we find $\mu>0,z\in U$ such that 
	\[
		(P_L\circ P)(x)=\mu(P_L\circ P)(z).
	\]
	Since $P_L$ is the orthogonal projection with kernel $\Span\{\E[n]\}$, we find $\beta_x,\beta_z\in\bb R$ such that 
	\[
		(P_L\circ P)(x)=P(x-\beta_x\E[A]),\quad (P_L\circ P)(z)=P(z-\beta_z\E[A]).
	\]
	In turn we have 
	\begin{align*}
		Pz= &\, (P_L\circ P)(z)+\beta_zP\E[A]=\frac 1\mu(P_L\circ P)(x)+\beta_zP\E[A]
		\\
		= &\, P\Big(\frac 1\mu x-\frac 1\mu\beta_x\E[A]+\beta_z\E[A]\Big)
		=P\Big(\frac\lambda\mu h_\Phi(w)+\Big[\frac\alpha\mu-\frac{\beta_x}\mu+\beta_z\Big]\E[A]\Big).
	\end{align*}
	Now \cref{S75} yields $\frac\lambda\mu h_\Phi(\omega)+\big[\frac\alpha\mu-\frac{\beta_x}\mu+\beta_z\big]\E[A]\in U$. 
	This is a contradiction, since $U\cap H(\Phi)=\emptyset$. 
\end{proof}

\noindent
Finally, we note a simple and general algebraic fact.

\begin{lemma}
\label{S9}
	Let $\bb X$ be a convex semilattice and $\Phi\in\HomCSL(\mc CA,\bb X)$. 
	Then $\ker\Phi$ is the smallest equivalence relation containing the set 
	\begin{multline}
	\label{S10}
		\Big\{\big(\Conv\{p_1,\ldots,p_s\},\Conv\{p_1,\ldots,p_s,p\}\big)\DSB s\in\bb N,
		p_1,\ldots,p_s,p\in\mc DA,
		\\
		\Phi\big(\Conv\{p_1,\ldots,p_s\}\big)=\Phi\big(\Conv\{p_1,\ldots,p_s,p\}\big)\Big\}
		.
	\end{multline}
\end{lemma}
\begin{proof}
	The equivalence relation generated by \cref{S10} is clearly contained in the kernel of $\Phi$. We have to prove the
	reverse inclusion. Let $(B,B')\in\ker\Phi$, then 
	\[
		\Phi(B\oplus B')=\Phi(B)\oplus\Phi(B')=\Phi(B)=\Phi(B')
		.
	\]
	Thus 
	\[
		\forall C\in\mc CA\DP\Big(B\subseteq C\subseteq B\oplus B'\vee B'\subseteq C\subseteq B\oplus B'\Big)
		\Rightarrow\Phi(C)=\Phi(B\oplus B').
	\]
	Now write $B=\Conv\{p_1,\ldots,p_s\}$ and $B'=\Conv\{q_1,\ldots,q_r\}$. Then each of the pairs 
	\begin{align*}
		& \big(\Conv\{p_1,\ldots,p_s,q_1,\ldots,q_{k-1}\},\Conv\{p_1,\ldots,p_s,q_1,\ldots,q_k\}\big)
		\text{ for }k=1,\ldots,r,
		\\
		& \big(\Conv\{q_1,\ldots,q_r,p_1,\ldots,p_{l-1}\},\Conv\{q_1,\ldots,q_r,p_1,\ldots,p_l\}\big)
		\text{ for }l=1,\ldots,s,
	\end{align*}
	belongs to the set \cref{S10}.
\end{proof}

\noindent
The proof of \Cref{S3} is now easily obtained by plugging together the above lemmata.

\begin{proof}[Proof of \Cref{S3}]
	Assume that $H(\Phi)\supseteq H(\Phi')$. Let $p_1,\ldots,p_s,p\in\mc DA$ with 
	$\Phi(\Conv\{p_1,\ldots,p_s\})=\Phi(\Conv\{p_1,\ldots,p_s,p\})$. Then, by \Cref{S6}, we have 
	$C(p_1,\ldots,p_s;p)\cap H(\Phi)=\emptyset$, and hence also $C(p_1,\ldots,p_s;p)\cap H(\Phi')=\emptyset$.
	Referring again to \Cref{S6} we obtain 
	\[
		\big(\Conv\{p_1,\ldots,p_s\},\Conv\{p_1,\ldots,p_s,p\}\big)\in\ker\Phi'.
	\]
	Now \Cref{S9} yields $\ker\Phi\subseteq\ker\Phi'$.

	Assume that $\ker\Phi\subseteq\ker\Phi'$. Let $x\in\bb R^A\setminus H(\Phi)$. By \Cref{S8} we
	find $p_1,\ldots,p_s,p\in\mc DA$ with 
	\[
		x\in C(p_1,\ldots,p_s;p)\subseteq\bb R^A\setminus H(\Phi).
	\]
	\Cref{S6} yields 
	\[
		\big(\Conv\{p_1,\ldots,p_s\},\Conv\{p_1,\ldots,p_s,p\}\big)\in\ker\Phi.
	\]
	It follows that this pair also belongs to $\ker\Phi'$, and again by \Cref{S6} we obtain 
	$C(p_1,\ldots,p_s;p)\cap H(\Phi')=\emptyset$. In particular, $x\notin H(\Phi')$. 
	We see that $\bb R^n\setminus H(\Phi)\subseteq\bb R^n\setminus H(\Phi')$, i.e., $H(\Phi)\supseteq H(\Phi')$.
\end{proof}

\subsection{Lifting \Cref{S3} to the class $\mc W$}

We present a more flexibel variant of the notion $H(\Phi)$ and of \Cref{S3}.

\begin{Definition}
\label{S80}
	Denote by $\mc W$ the class of all convex semilattices $\bb X$ that are isomorphic to a subalgebra of some power $\bb R^\Omega$
	where $\Omega\neq\emptyset$.
\end{Definition}

\noindent
The property that a convex semilattice belongs to $\mc W$ could be formulated differently: we have 
\begin{align*}
	\bb X\in\mc W\ \Leftrightarrow\ &\, \exists\Omega\neq\emptyset\DP\HomCSL(\bb X,\bb R^\Omega)\text{ contains an injective map}
	\\
	\Leftrightarrow\ &\, \HomCSL(\bb X,\bb R)\text{ is point separating}.
\end{align*}
Here the condition that $\HomCSL(\bb X,\bb R)$ is \emph{point separating} means that
\[
	\forall x,y\in X,x\neq y\DQ\exists f\in\HomCSL(\bb X,\bb R)\DP f(x)\neq f(y).
\]
\Cref{S3} has the following immediate consequence.

\begin{corollary}
\label{S26}
	Let $A\neq\emptyset$, let $\bb X\in\mc W$, and let $\Phi\in\HomCSL(\mc CA,\bb X)$. 
	If $\Omega$ and $\Omega'$ are nonempty sets, and $\tau\in\HomCSL(\bb X,\bb R^\Omega)$ and 
	$\tau'\in\HomCSL(\bb X,\bb R^{\Omega'})$ are both injective, then 
	\[
		H(\tau\circ\Phi)=H(\tau'\circ\Phi).
	\]
\end{corollary}
\begin{proof}
	Since $\tau$ and $\tau'$ are injective, we have 
	\[
		\ker(\tau\circ\Phi)=\ker\Phi=\ker(\tau'\circ\Phi).
	\]
	Now apply \Cref{S3}.
\end{proof}

\noindent
This justifies the following definition.

\begin{Definition}
\label{S27}
	Let $A\neq\emptyset$, let $\bb X\in\mc W$, and let $\Phi\in\HomCSL(\mc CA,\bb X)$. 
	Choose $\Omega\neq\emptyset$ and an injective map $\tau\in\HomCSL(\bb X,\bb R^\Omega)$, and set 
	\[
		{\sf H}(\Phi)\DE H(\tau\circ\Phi).
	\]
\end{Definition}

\noindent
The benefit of working in this slightly more general setting is the gain in flexibility that stems from the 
freedom of the choice of $\tau$. 

Clearly, the analogue of \Cref{S3} in this setting holds.

\begin{theorem}
\label{S30}
	Let $A\neq\emptyset$, let $\bb X,\bb X'\in\mc W$, and let $\Phi\in\HomCSL(\mc CA,\bb X)$ and 
	$\Phi'\in\HomCSL(\mc CA,\bb X')$. Then 
	\[
		\ker\Phi\subseteq\ker\Phi'\ \Leftrightarrow\ {\sf H}(\Phi)\supseteq {\sf H}(\Phi').
	\]
\end{theorem}
\begin{proof}
	Choose $\Omega$ and $\Omega'$ and injective homomorphisms $\tau\in\HomCSL(\bb X,\bb R^\Omega)$ and 
	$\tau'\in\HomCSL(\bb X',\bb R^{\Omega'})$. By \Cref{S3} it holds that
	\[
		\ker(\tau\circ\Phi)\subseteq\ker(\tau'\circ\Phi')
		\ \Leftrightarrow\ 
		H(\tau\circ\Phi)\supseteq H(\tau'\circ\Phi').
	\]
	Since $\tau$ and $\tau'$ are both injective, we have
	\[
		\ker\Phi=\ker(\tau\circ\Phi),\quad \ker\Phi'=\ker(\tau'\circ\Phi'),
	\]
	and by definition ${\sf H}(\Phi)=H(\tau\circ\Phi)$ and ${\sf H}(\Phi')=H(\tau'\circ\Phi')$.
\end{proof}

\noindent
We obtain a useful criterion for existence of factorisations of homomorphisms. 

\begin{corollary}
\label{S31}
	Let $A\neq\emptyset$, and $\bb X,\bb Y\in\mc W$.
	Assume we have homomorphisms $\Lambda\in\HomCSL(\mc CA,\bb X)$ and $\Phi\in\HomCSL(\mc CA,\bb Y)$ where $\Lambda$
	is surjective. 
	\begin{Enumerate}
	\item There exists a homomorphism $\Psi\in\HomCSL(\bb X,\bb Y)$ with $\Phi=\Psi\circ\Lambda$
		\begin{equation}
		\label{S76}
			\begin{tikzcd}[column sep=large]
				\mc CA \arrow[r,two heads,"\Lambda"] \arrow[rd,swap,"\Phi"]
				& \bb X \arrow[d,dashed,"\Psi"]
				\\
				& \bb Y
			\end{tikzcd}
		\end{equation}
		if and only if ${\sf H}(\Lambda)\supseteq {\sf H}(\Phi)$. 
	\item There exists an injective homomorphism $\Psi\in\HomCSL(\bb X,\bb Y)$ with \cref{S76}, if and only if 
		${\sf H}(\Lambda)={\sf H}(\Phi)$. 
	\end{Enumerate}
\end{corollary}
\begin{proof}
	Since $\Lambda$ is surjective, $\Psi$ exists if and only if $\ker\Lambda\subseteq\ker\Phi$, and $\Psi$ will be injective
	if and only if equality holds. Now apply \Cref{S30}.
\end{proof}

\section{Construction of homomorphisms}

In this section we deal with homomorphisms between free finitely generated convex semilattices, and consider the set 
\[
	\mc H(n,k)\DE\big\{{\sf H}(\Phi)\DSb\Phi\in\HomCSL(\mc Cn,\mc Ck)\big\},
\]
where $n,k\in\bb N\setminus\{0\}$. Our aim is to find an explicit description of $\mc H(n,k)$. 

To start with we observe two simple properties.
\begin{Ilist}
\item We have
	\[
		\forall n,k\in\bb N\setminus\{0\}\DP \mc H(n,k)\subseteq\mc H(n,k+1)
	\]
	This follows since $\mc Ck$ can be embedded in $\mc C(k+1)$.
\item We have 
	\[
		\Span\{\E[n]\}\in\mc H(n,k)\subseteq H(n,\bb R^k).
	\]
	This follows since we always have constant maps as homomorphisms, and since we can compose with the support function
	representation to compute ${\sf H}(\Phi)$. 
\end{Ilist}
For $n\in\{1,2\}$ or $k\in\{1,2\}$ the set $\mc H(n,k)$ can be determined easily.

\begin{lemma}
\label{S55}
	\phantom{}
	\begin{Enumerate}
	\item \Dis{%
		\forall n,k\in\bb N\setminus\{0\}\DP 
		\big(n=1\vee k=1\big)\ \Rightarrow\ \mc H(n,k)=\big\{\Span\{\E[n]\}\big\}
		}
	\item \Dis{%
		\forall n\in\bb N\setminus\{0\}\DP \mc H(n,2)=H(n,2)
		}
	\item \phantom{}\\[-15mm]
		\begin{multline*}
		\mkern-14mu\forall n,k\in\bb N\setminus\{0,1\}\DP\big(n=2\vee k=2\big)\ \Rightarrow\ 
		\\
		\mc H(n,k)=\Big\{\big(\Cone\{y_1\}\cup\Cone\{y_2\}\big)+\Span\{\E[n]\}\DS y_1,y_2\in\bb R^n\Big\}
		\end{multline*}
	\end{Enumerate}
\end{lemma}
\begin{proof}
	The sets $\HomCSL(\mc Cn,\mc C1)$ and $\HomCSL(\mc C1,\mc Ck)$ contain only constant maps, and \Cref{S83}(i) yields the assertion
	in item (i).

	The inclusion ``$\subseteq$'' in (ii) holds since $\mc C2$ can be embedded into $\bb R^2$ using the map $\tau$ from \cref{S73}.
	To show the reverse inclusion, assume we have $\Phi\in\HomCSL(\mc Cn,\bb R^2)$. 
	The image of $\tau$ contains a nonempty open subset of $\bb R^2$. 
	From what was said about subalgebras in \Cref{S63} we obtain that $\Phi(\mc Cn)$ is bounded. Hence, there exist 
	$\lambda>0$ and $a\in\bb R^2$ such that $(T_a\circ M_\lambda)(\Phi(\mc Cn))\subseteq\tau(\mc C2)$. We have 
	\[
		\Psi\DE\big[\tau^{-1}\circ T_a\circ M_\lambda\big]\circ\Phi\in\HomCSL(\mc Cn,\mc C2),
	\]
	and 
	\[
		{\sf H}(\Psi)=H\big(\big[M_{\frac 1\lambda}\circ T_{-a}\circ\tau\big]\circ\Psi\big)=H(\Phi).
	\]
	Item (iii) follows easily. If $k=2$, \Cref{S2}(i) yields the assertion. If $n=2,k\geq 2$, we use \Cref{S2}(iv) to compute 
	\begin{align*}
		H(2,\bb R^k)= &\, H(2,2)=\mc H(2,2)\subseteq\mc H(2,k)\subseteq H(2,\bb R^k)
		\\
		= &\, \Big\{\textstyle{\Span\{\E[2]\},\big\{\binom\alpha\beta\in\bb R^2\DSb \alpha\leq\beta\big\},
		\big\{\binom\alpha\beta\in\bb R^2\DSb\alpha\geq\beta\big\},\bb R^2}\Big\}
		\\
		= &\, \Big\{\big(\Cone\{y_1\}\cup\Cone\{y_2\}\big)+\Span\{\E[n]\}\DS y_1,y_2\in\bb R^n\Big\}.
	\end{align*}
\end{proof}

\noindent
The decisive quantity required for a description of $\mc H(n,k)$ when $n,k\geq 3$ is the following.

\begin{Definition}
\label{S25}
	Let $n\in\bb N\setminus\{0\}$.
	We define a map $\delta\DF\mc P(\bb R^n)\setminus\{\emptyset\}\to\bb N\cup\{\infty\}$ as follows.
	Let $H$ be a nonempty subset of $\bb R^n$. If there exist finitely many finitely generated cones $C_1,\ldots,C_N$ such
	that 
	\begin{equation}
	\label{S11}
		H=\bigcup_{i=1}^NC_i,
	\end{equation}
	then set
	\begin{equation}
	\label{S41}
		\delta(H)\DE\max\big\{\dim(\Span C_i)\DS i\in\{1,\ldots,N\}\big\}
		.
	\end{equation}
	If $H$ cannot be represented as a union \cref{S11}, set $\delta(H)\DE\infty$. 
\end{Definition}

\noindent
First of all we have to argue that $\delta(H)$ is well-defined, and we right away include more properties of
$\delta$ into the statement.

\begin{lemma}
\label{S42}
	Let $n\in\bb N\setminus\{0\}$.
	\begin{Enumerate}
	\item Assume we have finitely generated cones $C_1,\ldots,C_N$ and $C'_1,\ldots,C'_{N'}$ with 
		\[
			\bigcup_{i=1}^N C_i\subseteq\bigcup_{j=1}^{N'} C'_j.
		\]
		Then 
		\begin{multline*}
			\max\big\{\dim\big(\Span C_i\big)\DS i\in\{1,\ldots,N\}\big\}
			\\
			\leq\max\big\{\dim\big(\Span C'_j\big)\DS j\in\{1,\ldots,N'\}\big\}.
		\end{multline*}
		In particular, the map $\delta$ is well-defined.
	\item For any $H\in\mc P(\bb R^n)\setminus\{\emptyset\}$ we have $\delta(H)\in\{0,\ldots,n\}\cup\{\infty\}$. Moreover, 
		\begin{align*}
			& \delta(H)=0\,\Leftrightarrow\,H=\{0\}
			\\
			& \delta(H)=1\,\Leftrightarrow\,
			\exists m\geq 1,y_1,\ldots,y_m\in\bb R^n\setminus\{0\}\DP H=\bigcup_{j=1}^m\Conv\{y_j\}
		\end{align*}
	\item Assume that $H,H'\in\mc P(\bb R^n)\setminus\{\emptyset\}$ with $\delta(H),\delta(H')<\infty$. Then 
		\[
			\delta(H\cup H')=\max\{\delta(H),\delta(H')\},\quad
			\delta(H+H')\leq\delta(H)+\delta(H').
		\]
		In particular, $\delta$ is monotone in the sense that 
		\[
			H\subseteq H'\wedge\delta(H)<\infty\ \Rightarrow\ \delta(H)\leq\delta(H')
			.
		\]
	\item Let $H\in\mc P(\bb R^n)\setminus\{\emptyset\}$ and let $\kappa\DF\bb R^n\to\bb R^m$ be a linear map. Then 
		$\delta(\kappa(H))\leq\delta(H)$.
	\item If $H\in\mc P(\bb R^n)\setminus\{\emptyset\}$ and $\delta(H)<\infty$, then $H$ is closed.
	\end{Enumerate}
\end{lemma}
\begin{proof}
	\phantom{}
	\begin{Elist}
	\item
		Let $i\in\{1,\ldots,N\}$. Choose $x_1,\ldots,x_m\in\bb R^n$ such that
		$C_i=\Cone\{x_1,\ldots,x_m\}$, and choose a maximal linearly independent subset $\{x_{l_1},\ldots,x_{l_d}\}$ of 
		$\{x_1,\ldots,x_m\}$. Then
		\[
			\Span C_i=\Span\{x_{l_1},\ldots,x_{l_d}\},\quad d=\dim\big(\Span C_i\big).
		\]
		Consider the set 
		\[
			U\DE\Big\{\sum_{k=1}^d\beta_kx_{l_k}\DS\forall k\in\{1,\ldots,d\}\DP\beta_k>0\Big\}.
		\]
		Then we have 
		\[
			U\subseteq C_i=C_i\cap\bigcup_{j=1}^{N'} C'_j=\bigcup_{j=1}^{N'}(C_i\cap C'_j)
			\subseteq\bigcup_{j=1}^{N'}\Span(C_i\cap C'_j)\subseteq\Span C_i.
		\]
		Since $U$ is a nonempty open subset of the linear space $\Span C_i$, it follows that 
		\[
			\exists j\in\{1,\ldots,N'\}\DP\Span(C_i\cap C_j')=\Span C_i.
		\]
		Choosing $j$ with this property, we obtain
		\begin{align*}
			\dim\big(\Span C_i\big)= &\, \dim\Span(C_i\cap C_j')\leq\dim\Span C_j'
			\\
			\leq &\, \max\big\{\dim\Span C'_j\DS j\in\{1,\ldots,N'\}\big\}.
		\end{align*}
	\item
		Assume $\delta(H)<\infty$ and choose a decomposition \cref{S11}. Clearly, we have
		$\dim\Span C_i\leq n$. Moreover, if $\delta(H)=0$, then for each $i$ and $C_i$ as in \cref{S11} we have 
		$C_i\subseteq\Span C_i=\{0\}$. Since $H\neq\emptyset$, thus $H=\{0\}$. Conversely, if $H=\{0\}$, we can use 
		$N\DE 1$ and $C_1\DE\{0\}$ in \cref{S11} and \cref{S41}. 
	\item
		Write $H=\bigcup_{i=1}^N C_i$ and $H'=\bigcup_{j=1}^{N'} C'_j$. Then 
		\[
			H\cup H'=\bigcup_{i=1}^N C_i\cup\bigcup_{j=1}^{N'} C'_j,\quad
			H+H'=\bigcup_{i=1}^N\bigcup_{j=1}^{N'}\big(C_i+C'_j\big).
		\]
		We already see that $\delta(H\cup H')=\max\{\delta(H),\delta(H')\}$. Each $C_i+C'_j$ is a finitely generated cone and
		$\dim\Span(C_i+C'_j)\leq\dim\Span C_i+\dim\Span C'_j$. Therefore $\delta(H+H')\leq\delta(H)+\delta(H')$.
	\item
		If $\delta(H)=\infty$, there is nothing to prove. Assume $H=\bigcup_{i=1}^N C_i$ as in \cref{S11}. Then 
		\[
			\kappa(H)=\kappa\Big(\bigcup_{i=1}^N C_i\Big)=\bigcup_{i=1}^N\kappa(C_i).
		\]
		Each $\kappa(C_i)$ is a finitely generated cone and $\Span\kappa(C_i)=\kappa(\Span C_i)$. Thus
		$\dim\Span\kappa(C_i)\leq\dim\Span C_i$ and in turn $\delta(\kappa(H))\leq\delta(H)$.
	\item
		Every finitely generated cone is closed, and hence also every finite union of such cones is. 
	\end{Elist}
\end{proof}

\begin{Example}
\label{S43}
	Let $H$ be a linear subspace of $\bb R^n$. Then 
	\[
		\delta(H)=\dim H.
	\]
	To see this, choose a basis $\{x_1,\ldots,x_d\}$ of $H$, write 
	\[
		H=\bigcup\Big\{\Cone\big\{(-1)^{l_j}x_j\DS j\in\{1,\ldots,d\}\big\}\DSb (l_1,\ldots,l_d)\in\{0,1\}^d\Big\},
	\]
	and note that 
	\[
		\Span\big\{(-1)^{l_j}x_j\DS j\in\{1,\ldots,d\}\big\}=H.
	\]
\end{Example}

\noindent
We can now formulate the main result of this section.

\begin{Theorem}
\label{S24}
	Let $n,k\in\bb N\setminus\{0\}$. If $n\in\{1,2\}$ or $k\neq 2$, then 
	\begin{equation}
	\label{S58}
		\mc H(n,k)=\big\{\tilde H\in\mc P(\bb R^n)\setminus\{\emptyset\}\DSb
		\tilde H=\tilde H+\Span\{\E[n]\}\wedge\delta(\tilde H)\leq k\big\}.
	\end{equation}
	If $n\geq 3$ and $k=2$, then ``\,$\subsetneq$'' instead of ``\,$=$'' holds in \cref{S58}.
\end{Theorem}

\noindent
Note that the set on the right side of \cref{S58} is always nonempty since it contains $\Span\{\E[n]\}$.

\subsection*{Proof of \Cref{S24}}

We use Caratheodory's theorem in the following form.

\begin{lemma}
\label{S18}
	Let $H\subseteq\bb R^n$ with $\delta(H)<\infty$. Then there exist $T\in\bb N$ and linearly independent sets 
	$M_0,\ldots,M_T\subseteq\bb R^n$, such that 
	\[
		H=\bigcup_{t=0}^T\Cone M_t.
	\]
\end{lemma}
\begin{proof}
	Since $\delta(H)<\infty$, the set $H$ can be represented as a finite union of finitely generated cones 
	$C_1,\ldots,C_s$. Write $C_i=\Cone N_i$ with some finite sets $N_i$. By the variant of Caratheodory's theorem given in 
	\cite[Corollary~7.1i]{schrijver:1998} we have 
	\[
		\Cone N_i=\bigcup\big\{\Cone M\DS M\subseteq N_i, M\text{ linearly independent}\big\}.
	\]
	Since $N_i$ is finite, this is a finite union, and we obtain 
	\[
		H=\bigcup_{i=1}^s\bigcup_{\substack{M\subseteq N_i\\ M\text{ l.i.}}}\Cone M.
	\]
\end{proof}

\noindent
The next lemma gives a practical reformulation of the set on the right side of \cref{S58}. 

\begin{lemma}
\label{S39}
	Let $n\in\bb N\setminus\{0\}$ and $\tilde H\in\mc P(\bb R^n)\setminus\{\emptyset\}$. Assume that
	\begin{equation}
	\label{S36}
		\tilde H=\tilde H+\Span\{\E[n]\}.
	\end{equation}
	Then
	\[
		\delta(\tilde H)\!=\!\inf\big\{k\!\in\!\bb N\!\setminus\!\{0\}\DSb
		\exists H\!\in\!\mc P(\bb R^n)\!\setminus\!\{\emptyset\}\DP
		\delta(H)\!\leq k\!-\!1\wedge\tilde H\!=\!H\!+\!\Span\{\E[n]\}\big\}.
	\]
	In particular, we have 
	\begin{equation}
	\label{S56}
		\delta(\tilde H)=1\ \Leftrightarrow\ \tilde H=\Span\{\E[n]\}
	\end{equation}
\end{lemma}
\begin{proof}
	To prove ``$\leq$'' assume we have $k\in\bb N\setminus\{0\}$ and $H\subseteq\bb R^n$ with $\delta(H)\leq k-1$ and 
	$\tilde H=H+\Span\{\E[n]\}$. Choose finitely generated cones $C_1,\ldots,C_n$ such that $H=\bigcup_{i=1}^N C_i$. 
	Then we can write 
	\[
		\tilde H=\bigcup_{i=1}^N\Cone\big(C_i\cup\{\E[n]\}\cup\{-\E[n]\}\big).
	\]
	We have
	\[
		\dim\Big(\Span\big(C_i\cup\{\E[n]\}\cup\{-\E[n]\}\big)\Big)\leq\dim\big(\Span C_i\big)+1\leq\delta(H)+1\leq k,
	\]
	and hence $\delta(\tilde H)\leq k$. 

	We come to the reverse inequality. If $\delta(\tilde H)=\infty$ there is nothing to prove, hence assume that 
	$\delta(\tilde H)<\infty$. Let $P\DF\bb R^n\to\bb R^n$ the projection with $\ker P=\Span\{\E[n]\}$ and $\ran P=\{\E[n]\}^\perp$. 
	Due to \cref{S36} we have 
	\[
		\tilde H=\tilde H+\Span\{\E[n]\}=P(\tilde H)+\Span\{\E[n]\}.
	\]
	Choose finitely generated cones $\tilde C_1,\ldots,\tilde C_n$ such that $\tilde H=\bigcup_{i=1}^N \tilde C_i$. Then
	$P(\tilde H)=\bigcup_{i=1}^N P(\tilde C_i)$ and each $P(\tilde C_i)$ is a finitely generated cone. Hence, 
	$\delta(P(\tilde H))<\infty$. By \Cref{S18} we can write $P(\tilde H)=\bigcup_{t=0}^T\Cone M_t$ with linearly independent sets 
	$M_t$. Set $L_t\DE\Span(M_t\cup\{\E[n]\})$, then 
	\[
		\Cone\big(M_t\cup\{\E[n]\}\big)\subseteq L_t\cap\tilde H\subseteq\bigcup_{i=1}^N\big(L_t\cap\Span\tilde C_i\big)
		\subseteq L_t.
	\]
	Since $\Cone(M_t\cup\{\E[n]\})$ contains a nonempty open subset of $L_t$, there must exist $i\in\{1,\ldots,N\}$ with 
	$L_t\cap\Span\tilde C_i=L_t$. It follows that 
	\[
		\dim L_t\leq\dim\Span\tilde C_i\leq\delta(\tilde H).
	\]
	However, since $M_t\subseteq\ran P$, the set $M_t\cup\{\E[n]\}$ is linearly independent and thus $\dim L_t=|M_t|+1$. 
	We see that $\dim\Span M_t\leq\delta(\tilde H)-1$, and conclude that 
	$\delta(P(\tilde H))\leq\delta(\tilde H)-1$. 

	The equivalence \cref{S56} follows remembering \Cref{S42}(ii).
\end{proof}

\noindent
As a consequence of this lemma we have 
\begin{multline}
\label{S71}
	\forall n,k\in\bb N\setminus\{0\}\DP
	\big\{\tilde H\in\mc P(\bb R^n)\setminus\{\emptyset\}\DSb
	\tilde H=\tilde H+\Span\{\E[n]\}\wedge\delta(\tilde H)\leq k\big\}
	\\
	=\big\{H+\Span\{\E[n]\}\DSb H\in\mc P(\bb R^n)\setminus\{\emptyset\}\wedge\delta(H)\leq k-1\big\}.
\end{multline}
Further on in the proof of \Cref{S24} we work with the set \cref{S71} in place of the set on the right side of \cref{S58}.

Due to \Cref{S55} the ``boundary cases'' in \Cref{S24} that $n\in\{1,2\}$ or $k\in\{1,2\}$ are easy to see.

\begin{proof}[Proof of \Cref{S24}; $n\in\{1,2\}\vee k\in\{1,2\}$]
	\phantom{}
	\begin{Ilist}
	\item Case $k=1$: The only possible choice for $H$ in \cref{S71} is $\{0\}$, and hence 
		\[
			\text{\cref{S71}}=\{\Span\{\E[n]\}\}=\mc H(n,1).
		\]
	\item Case $n=1$: We have $\Span\{\E[n]\}=\bb R$, and hence 
		\[
			\text{\cref{S71}}=\{\bb R\}=\mc H(1,k).
		\]
	\item Case $n=2$: Using \Cref{S42}(ii) and \Cref{S55}(iii) we compute 
		\begin{align*}
			\text{\cref{S71}}= &\,
			\Big\{\bigcup_{i=1}^m\Cone\{y_i\}+\Span\{\E[A]\}\DSB y_1,\ldots,y_m\in\bb R^A\Big\}
			\\
			= &\, \Big\{\textstyle{\Span\{\E[2]\},\big\{\binom\alpha\beta\in\bb R^2\DSb \alpha\leq\beta\big\},
			\big\{\binom\alpha\beta\in\bb R^2\DSb\alpha\geq\beta\big\},\bb R^2}\Big\}
			\\
			= &\, \Big\{\big(\Cone\{y_1\}\cup\Cone\{y_2\}\big)+\Span\{\E[n]\}\DS y_1,y_2\in\bb R^n\Big\}.
		\end{align*}
	\item Case $n\geq 3\wedge k=2$: In this case we have 
		\begin{align*}
			\text{\cref{S71}}= &\,
			\Big\{\bigcup_{i=1}^m\Cone\{y_i\}+\Span\{\E[A]\}\DSB y_1,\ldots,y_m\in\bb R^A\Big\}
			\\
			\supsetneq &\, \Big\{\big(\Cone\{y_1\}\cup\Cone\{y_2\}\big)+\Span\{\E[n]\}\DS y_1,y_2\in\bb R^n\Big\}
			=\mc H(n,2).
		\end{align*}
	\end{Ilist}
\end{proof}

\noindent
Next we prove the inclusion ``$\subseteq$'' in \cref{S58}.

\begin{proof}[Proof of ``\,$\subseteq$'' in \cref{S58}]
	Let $\Phi\in\HomCSL(\mc Cn,\mc Ck)$ be given. In order to compute ${\sf H}(\Phi)$ 
	we use the following variant of the support function representation:
	\[
		\tau\DE\rho_{k-1}\circ\Moewe[18][2][6.25]{\kappa_k}\DF\mc Ck\to\bb R^{(\bb R^{k-1})}
	\]
	cf.\ \Cref{S74}\Elistref{S89}. The function $h_{\tau\circ\Phi}$ computes as 
	\[
		h_{\tau\circ\Phi}(\omega)=
		\Big(\sigma\big((\Moewe[18][2][6.25]{\kappa_k}\circ\Phi)(\{e_j\}),\omega\big)\Big)_{j=1}^n
		\qquad\text{for }\omega\in\bb R^{k-1}
	\]
	Each function $\sigma((\Moewe[18][2][6.25]{\kappa_k}\circ\Phi)(\{e_j\}),\Dummy)$ is piecewise linear, 
	and hence also $h_{\tau\circ\Phi}$ has this property. Choose finitely many finitely generated cones 
	$E_1,\ldots,E_N$ with $\bigcup_{i=1}^NE_i=\bb R^{k-1}$, such that for each $i\in\{1,\ldots,N\}$ the restriction 
	$h_{\tau\circ\Phi}|_{E_i}$ coincides with the restriction to $E_i$ of some linear function 
	$\xi_i\DF\bb R^{k-1}\to\bb R^n$. 

	Set $C_i\DE\xi_i(E_i)$, then $C_i$ is a finitely generated cone. Since $\Span C_i\subseteq\xi_i(\bb R^{k-1})$, we have 
	$\dim(\Span C_i)\leq k-1$. Using that 
	\[
		h_{\tau\circ\Phi}(\bb R^{k-1})=\bigcup_{i=1}^N h_{\tau\circ\Phi}(E_i)=\bigcup_{i=1}^NC_i,
	\]
	we obtain $\delta(h_{\tau\circ\Phi}(\bb R^{k-1}))\leq k-1$. 

	Since $h_{\tau\circ\Phi}$ is positively homogeneous, we can write 
	\begin{align*}
		\big\{\lambda h_{\tau\circ\Phi}(\omega)+ &\, \alpha\E[n]\DSb 
		\lambda\geq 0,\omega\in\bb R^{k-1},\alpha\in\bb R\big\}
		=h_{\tau\circ\Omega}(\bb R^{k-1})+\Span\{\E[n]\}
		\\
		= &\, \bigcup_{i=1}^NC_i+\Span\{\E[n]\}
		=\bigcup_{i=1}^N\Cone\big(C_i\cup\{\E[n]\}\big)\cup\bigcup_{i=1}^N\Cone\big(C_i\cup\{-\E[n]\}\big).
	\end{align*}
	Each of the cones in the last union is finitely generated and hence closed. 
	Thus also their union is closed, and we conclude that 
	\begin{equation}
	\label{S93}
		{\sf H}(\Phi)=h_{\tau\circ\Phi}(\bb R^{k-1})+\Span\{\E[n]\}.
	\end{equation}
\end{proof}

\noindent
The proof of ``$\supseteq$'' in \Cref{S24} is the involved part; 
we have to construct homomorphisms $\Phi$ whose set ${\sf H}(\Phi)$ is of a prescribed form. 

We start with an observation that points the direction we are heading to.

\begin{lemma}
\label{S16}
	Let $d,N\in\bb N$ with $1\leq d\leq N$. 
	Let $v_0,\ldots,v_N\in\bb R^d$, let $\gamma_0,\ldots,\gamma_N>0$, and let 
	$\mc J\subseteq\mc P(\{0,\ldots,N\})\setminus\{\emptyset\}$. 
	Assume that 
	\begin{itemize}
	\item[{\rm(A1)}] \Dis{\forall J\in\mc J\DP |J|=d\wedge\Span\{v_j\DS j\in J\}=\bb R^d}.
	\end{itemize}
	Note that {\rm(A1)} implies that the elements $v_j$, $j\in J$, are pairwise different and 
	$\{v_j\DS j\in J\}$ is linearly independent.
	For $J\in\mc J$ denote by $x_J\in\bb R^d$ the unique solution of the linear system
	\begin{equation}
	\label{S52}
		\forall j\in J\DP (x,v_j)=\gamma_j.
	\end{equation}
	Now assume furthermore that our data satisfies
	\begin{itemize}
	\item[{\rm(A2)}] \Dis{\forall J\in\mc J,k\in\{0,\ldots,N\}\setminus J\DP (x_J,v_k)<\gamma_k}.
	\end{itemize}
	Then the set 
	\begin{equation}
	\label{S54}
		K\DE\big\{x\in\bb R^d\DS\forall j\in\{0,\ldots,N\}\DP(x,v_j)\leq\gamma_j\big\}
	\end{equation}
	satisfies
	\begin{equation}
	\label{S34}
		\forall J\in\mc J\DQ\forall(\beta_j)_{j\in J}\in[0,\infty)^J\DP 
		\max\Big\{\Big(x,\sum_{j\in J}\beta_jv_j\Big)\DSB x\in K\Big\}=\sum_{j\in J}\beta_j\gamma_j.
	\end{equation}
\end{lemma}
\begin{proof}
	For each $x\in K$ and $\beta_j\geq 0$ we can estimate
	\begin{equation}
	\label{S33}
		\Big(x,\sum_{j\in J}\beta_jv_j\Big)=\sum_{j\in J}\beta_j(x,v_j)\leq\sum_{j\in J}\beta_j\gamma_j.
	\end{equation}
	Now note that $x_J\in K$ and that 
	for $x_J$ equality holds in \cref{S33}.
\end{proof}

\noindent
An important technical role is played by the following object.

\begin{Definition}
\label{S53}
	Given $\mc J\subseteq\mc P(\{0,\ldots,N\})\setminus\{\emptyset\}$, we define an undirected graph $\mc G$ as follows.
	The set of vertices of $\mc G$ is $\{0,\ldots,N\}$ and the set of edges of $\mc G$ is $\bigcup_{J\in\mc J}(J\times J)$. 
	If $\mc G$ is connected, we denote by 
	\[
		d_{\mc G}\DF\{0,\ldots,N\}\times\{0,\ldots,N\}\to\bb N
	\]
	the shortest-path distance of this graph.
\end{Definition}

\noindent
The next lemma contains the crucial construction; we build a skeleton from which in the proof of \Cref{S24} 
candidates for $\Phi(\{e_j\})$ will be derived.

\begin{lemma}
\label{S19}
	Let $d,N\in\bb N$ with $2\leq d\leq N$. 
	Then there exist $v_0,\ldots,v_N\in\bb R^d$, $\gamma_0,\ldots,\gamma_N>0$, and 
	$\mc J\subseteq\mc P(\{0,\ldots,N\})\setminus\{\emptyset\}$, that satisfy {\rm(A1)}, {\rm(A2)}, and 
	\begin{itemize}
	\item[{\rm(A3)}] \Dis{\bigcup_{J\in\mc J}\Cone\{v_j\DS j\in J\}=\bb R^d},
	\item[{\rm(A4)}] the graph $\mc G$ is connected and 
		\[
			\max\Big\{\sum_{j\in J}d_{\mc G}(j,k)\DSB J\in\mc J,k\in \{0,\ldots,N\}\Big\}\geq N.
		\]
	\end{itemize}
	Note that $\mc G$ being connected implies that $\bigcup_{J\in\mc J}J=\{0,\ldots,N\}$.
\end{lemma}
\begin{proof}
	We fix $d\geq 2$ and use induction on $N$. 
\begin{Ilist}
\item Base case $N=d$:

	For $N=d$ the corner points of a shifted $(d+1)$-simplex do the job: set 
	\begin{align}
		\label{S46}
		& c\DE\frac 1{d+1}\E[d],\qquad v_j\DE
		\begin{cases}
			-c \CAS j=0,
			\\
			e_j-c \CAS j\in\{1,\ldots,d\},
		\end{cases}
		\\
		\nonumber
		& \gamma_j\DE 1\text{ for }j\in\{0,\ldots,d\},
		\\[1mm]
		\nonumber
		& \mc J\DE\{J_0,\ldots,J_d\}\text{ with }J_k\DE\{0,\ldots,d\}\setminus\{k\}\text{ for }k\in\{0,\ldots,d\}.
	\end{align}
	We check that the data $v_0,\ldots,v_d$, $\gamma_0,\ldots,\gamma_d$, $\mc J$ has all required properties.
	\begin{IIlist}
	\item \textit{Property {\rm(A1)}}:
		Clearly, we have $|J_k|=d$ for all $k$. Let $k\in\{0,\ldots,d\}$. If $k\neq 0$, the linear space 
		$\Span\{v_j\DS j\in J_k\}$ contains the elements $\E[d]=-(d+1)(-c)$ and $e_j=(e_j-c)-(-c)$ for 
		$j\in\{1,\ldots,d\}\setminus\{k\}$. Thus it equals $\bb R^d$. The linear space $\{v_j\DS j\in J_0\}$ contains the
		elements $\E[d]=(d+1)\sum_{j=1}^d(e_j-c)$ and in turn all $e_j$. Thus it equals $\bb R^d$.
	\item \textit{Property {\rm(A2)}}:
		Set 
		\[
			y_l\DE
			\begin{cases}
				(d+1)^2\cdot c \CAS l=0,
				\\
				-(d+1)\cdot e_l \CAS l\in\{1,\ldots,d\}.
			\end{cases}
		\]
		Using that 
		\[
			(c,c)=\frac d{(d+1)^2},\quad (c,e_j)=\frac 1{d+1},
		\]
		we compute for $l,i\in\{0,\ldots,d\}$
		\[
			(y_l,v_i)=
			\begin{cases}
				-d \CAS l=0,i=0,
				\\
				1 \CAS l=0,i>0,
				\\
				1 \CAS l>0,i=0,
				\\
				1 \CAS l>0,i>0,\ l\neq i,
				\\
				-d \CAS l>0,i>0,\ l=i.
			\end{cases}
		\]
		We see that $y_l$ satisfies \cref{S52}, and thus $x_{J_l}=y_l$. We also see that {\rm(A2)} holds. 
	\item \textit{Property {\rm(A3)}}: 
		The linear map 
		\[
			\Phi\DF\left\{
			\begin{array}{rcl}
				\bb R^{d+1} & \to & \bb R^d
				\\
				(\lambda_j)_{j=0}^d & \mapsto & \sum_{j=0}^d\lambda_jv_j
			\end{array}
			\right.
		\]
		is surjective and $\Phi((1)_{j=0}^d)=0$. Thus, given $x\in\bb R^d$, we find
		$(\lambda_j)_{j=0}^d\in\bb R^{d+1}$ with 
		\[
			\Phi\big((\lambda_j)_{j=0}^d\big)=x\quad\text{and}\quad\min_{j=0,\ldots,d}\lambda_j=0.
		\]
		Choose $l$ such that $\lambda_l=0$. Then $x\in\Cone\{v_j\DS j\in J_l\}$.
	\item \textit{Property {\rm(A4)}}: 
		Since $d\geq 2$ each two vertices are connected by an edge. We see that $\mc G$ is connected and that
		\[
			d_{\mc G}(j,k)=
			\begin{cases}
				0 \CAS j=k,
				\\
				1 \CAS j\neq k.
			\end{cases}
		\]
		It follows that 
		\[
			\forall l\in\{0,\ldots,d\}\DP
			\sum_{j\in J_l}d_{\mc G}(j,k)=
			\begin{cases}
				d-1 \CAS k\neq l
				\\
				d \CAS k=l
			\end{cases}
		\]
	\end{IIlist}
\item Induction step $N\mapsto N+1$:

	Assume we have $N\geq d$ and $v_0,\ldots,v_N$, $\gamma_0,\ldots,\gamma_N$, $\mc J$ that satisfy 
	{\rm(A1)}, {\rm(A2)}, {\rm(A3)}, {\rm(A4)}. 
	The idea is to attach an additional simplex at an appropriate face of the polytope given by that data, 
	add the new corner point and create the new faces. To this end choose $\mr J\in\mc J$ and $\mr k\in\{0,\ldots,N\}$ such that 
	\[
		\sum_{j\in\mr J}d_{\mc G}(j,\mr k)\geq N.
	\]
	Now let $\varepsilon\geq 0$ be a parameter, and set 
	\begin{align*}
		& v_k'\DE
		\begin{cases}
			v_k\CAS k\in\{0,\ldots,N\},
			\\
			\sum_{j\in\mr J}v_j\CAS k=N+1,
		\end{cases}
		\\
		& \gamma_k'\DE
		\begin{cases}
			\gamma_k\CAS k\in\{0,\ldots,N\},
			\\
			\big(\sum_{j\in\mr J}\gamma_j\big)-\varepsilon\CAS k=N+1,
		\end{cases}
		\\
		& \mc J'\DE\big(\mc J\setminus\{\mr J\}\big)\cup\big\{J_l'\DS l\in\mr J\big\}\text{ where }
		J_l'\DE(\mr J\setminus\{l\})\cup\{N+1\}.
	\end{align*}
	We are going to show that for all sufficiently small positive $\varepsilon$ the data 
	$v_0',\ldots,v_{N+1}'$, $\gamma_0',\ldots,\gamma_{N+1}'$, $\mc J'$ has all required properties. 
	\begin{IIlist}
	\item \textit{Property {\rm(A1)}}: 
		For $J'\in\mc J\setminus\{\mr J\}$ the required property holds by the inductive hypothesis. 
		Let $l\in\mr J$. Since $|\mr J|=d$, also $|J_l|=d$. Since the coefficient of $v_l$ in $v_{N+1}$ is nonzero we have 
		\[
			\Span\{v_j\DS j\in J_l\}=\Span\{v_j\DS j\in\mr J\}=\bb R^d.
		\]
	\item \textit{Property {\rm(A2)}}: 
		We denote by $x_{J'}'$, $J'\in\mc J'$, the solutions of the systems \cref{S52} corresponding to the data 
		$v_j',\gamma_j',\mc J'$, and by $x_J$, $J\in\mc J$, the solutions of \cref{S52} corresponding to $v_j,\gamma_j,\mc J$. 
		Note that $x_{J'}'=x_{J'}$ if $J'\in\mc J\setminus\{\mr J\}$. 

		Let $J'\in\mc J'$, $i\in\{0,\ldots,N+1\}\setminus J'$. 
		The required property is checked by going through all possible cases. 
		\begin{IIlist}
		\item $J'\in\mc J\setminus\{\mr J\}$, $i\in\{0,\ldots,N\}$: 
			Then the inductive hypothesis yields 
			\[
				(x_{J'}',v_i')=(x_{J'},v_i)<\gamma_i=\gamma_i'.
			\]
		\item $J'\in\mc J\setminus\{\mr J\}$, $i=N+1$: 
			Since $\mr J\setminus J'\neq\emptyset$, we obtain
			\[
				(x_{J'}',v_{N+1}')=\sum_{j\in\mr J}(x_{J'},v_j)=
				\sum_{j\in\mr J\cap J'}\underbrace{(x_{J'},v_j)}_{=\gamma_j}+
				\sum_{j\in\mr J\setminus J'}\underbrace{(x_{J'},v_j)}_{<\gamma_j}
				<\sum_{j\in\mr J}\gamma_j.
			\]
			
			Therefore $(x_{J'}',v_{N+1}')<\gamma_{N+1}'$ for all sufficiently small $\varepsilon\geq 0$.
		\item $J'=J_l$ for some $l\in\mr J$, $i=l$:
			We compute 
			\[
				(x_{J'}',v_i')=\Big(x_{J_l}',v_{N+1}'-\sum_{j\in\mr J\setminus\{l\}}v_j\Big)=
				\underbrace{(x_{J_l}',v_{N+1}')}_{=\gamma_{N+1}}
				-\sum_{j\in\mr J\setminus\{l\}}\underbrace{(x_{J_l}',v_j)}_{=\gamma_j}=\gamma_l-\varepsilon.
			\]
			Therefore $(x_{J'}',v_i)<\gamma_i'$ for all $\varepsilon>0$.
		\item $J'=J_l$ for some $l\in\mr J$, $i\neq l$:
			Then $i\in\{0,\ldots,N\}\setminus\mr J$, and we write 
			\[
				(x_{J'}',v_i')=\big(x_{J_l}'-x_{\mr J},v_i\big)+\underbrace{(x_{\mr J},v_i)}_{<\gamma_i}.
			\]
			For the parameter $\varepsilon=0$ we have $x_{J_l}'=x_{\mr J}$, and hence 
			$\lim_{\varepsilon\to 0}x_{J_l}'=x_{\mr J}$. 
			Therefore $(x_{J'}',v_i)<\gamma_i'$ for all sufficiently small $\varepsilon\geq 0$.
		\end{IIlist}
	\item \textit{Property {\rm(A3)}}: 
		Assume $x\in\Cone\{v_j\DS j\in\mr J\}$, then $x=\sum_{j\in\mr J}\xi_jv_j$ with some $\xi_j\geq 0$. Choose 
		$l\in\mr J$ with $\xi_l=\min\{\xi_j\DS j\in\mr J\}$, then 
		\[
			x=\sum_{j\in\mr J\setminus\{l\}}(\xi_j-\xi_l)v_j+\xi_lv_{N+1}\in\Cone\{v_j\DS j\in J'_l\}.
		\]
		Thus $\Cone\{v_j\DS j\in\mr J\}\subseteq\bigcup_{l\in\mr J}\Cone\{v_j'\DS j\in J'_l\}$, and the inductive
		hypothesis yields
		\[
			\bb R^d
			=\Cone\{v_j\DS j\in\mr J\}\mkern7mu\cup\mkern-8mu
			\bigcup_{J\in\mc J\setminus\{\mr J\}}\mkern-10mu\Cone\{v_j\DS j\in J\}
			=\bigcup_{J\in\mc J'}\Cone\{v_j\DS j\in J\}.
		\]
	\item \textit{Property {\rm(A4)}}: 
		We denote by $\mc G$ the graph induced from $\mc J$ and by $\mc G'$ the graph induced from $\mc J'$. Furthermore, 
		let $\mc E$ and $\mc E'$ be the sets of edges of $\mc G$ and $\mc G'$, respectively, i.e. 
		\[
			\mc E\DE\bigcup_{J\in\mc J}(J\times J),\quad \mc E'\DE\bigcup_{J'\in\mc J'}(J'\times J').
		\]
		Observe that
		\[
			\mc E'\cap\{0,\ldots,N\}^2\subseteq\mc E,
		\]
		\[
			\forall i\in\{0,\ldots,N\}\DP (N+1,i)\in\mc E'\ \Leftrightarrow\ i\in\mr J.
		\]
		We first show that $\mc G'$ is connected. Let $i,j\in\{0,\ldots,N\}$ and choose a path $\pi$ in $\mc G$
		connecting $i$ and $j$. Substituting every edge $(k,l)\in\mr J\times\mr J$ by the sequence of edges 
		$(k,N+1),(N+1,l)$, we obtain a path $\pi'$ in $\mc G'$ connecting $i$ and $j$.
		The vertex $N+1$ is connected with every vertex in $\mr J$. 

		Second, we estimate $d_{\mc G'}$. Let $i,j\in\{0,\ldots,N\}$, and let $\pi'$ be a path in $\mc G'$ 
		connecting $i$ and $j$. If a sequence $(k,N+1),(N+1,l)$ occurs in $\pi'$, then $k,l\in\mr J$.
		Hence, we obtain a path $\pi$ in $\mc G$ connecting $i$ and $j$ by substituting each such sequence in $\pi'$ by the edge
		$(k,l)$. It follows that 
		\[
			\forall i,j\in\{0,\ldots,N\}\DP d_{\mc G}(i,j)\leq d_{\mc G'}(i,j)
		\]
		Consider $i\in\{0,\ldots,N\}$. Let $\pi'$ be a path in $\mc G'$ connecting $N+1$ and 
		$i$. Then there exists $j\in\mr J$ and a path $\pi'_0$ in $\mc G'$, such that $\pi'$ is the concatenation of the 
		edge $(N+1,j)$ with $\pi'_0$. It follows that 
		\[
			d_{\mc G'}(N+1,i)\geq 1+\min_{j\in\mr J}d_{\mc G'}(j,i)
			\geq 1+\min_{j\in\mr J}d_{\mc G}(j,i).
		\]
		Now we have all ingredients to finish the argument.
		Choose $l\in\mr J$ such that $d_{\mc G}(l,\mr k)=\min_{j\in\mr J}d_{\mc G}(j,\mr k)$. Then
		\begin{align*}
			\sum_{j\in J_l}d_{\mc G'}(j,\mr k)= &\, 
			\sum_{j\in\mr J\setminus\{l\}}d_{\mc G'}(j,\mr k)+d_{\mc G'}(N+1,\mr k)
			\\
			\geq &\, \sum_{j\in\mr J\setminus\{l\}}d_{\mc G}(j,\mr k)+\big(1+d_{\mc G}(l,\mr k)\big)
			=\sum_{j\in\mr J}d_{\mc G}(j,\mr k)+1\geq N+1.
		\end{align*}
	\end{IIlist}
\end{Ilist}
\end{proof}

\begin{Remark}
\label{S20}
	\phantom{}
	\begin{Elist}
	\item The conditions (A1), (A3), (A4) depend only on $v_0,\ldots,v_N$ and $\mc J$. 
		The values of $\gamma_0,\ldots,\gamma_N$ enter only in (A2). 
	\item Assume we have $v_0,\ldots,v_N\in\bb R^d$ and $\mc J\subseteq\mc P(\{0,\ldots,N\})\setminus\{\emptyset\}$ 
		such that (A1), (A3), (A4) hold. Then the set of all $\gamma_0,\ldots,\gamma_N>0$ 
		such that (A2) holds is open. 
		This follows since $x_J$ depends continuously on $\gamma_0,\ldots,\gamma_N$, and strict inequality is required in {\rm(A2)}.
	\item The condition (A3) ensures that the convex set $K$ from \cref{S54} is bounded. 
		To see this, assume towards a contradiction that there exists a sequence $(x_n)_{n\in\bb N}$ in $K$ with 
		$\|x_n\|\geq n$. By passing to a subsequence if necessary we may assume w.l.o.g.\ that there exists $J\in\mc J$
		such that $x_n\in\Cone\{v_j\DS j\in J\}$ for all $n$ and that the limit 
		$y\DE\lim_{n\to\infty}\frac{x_n}{\|x_n\|}$ exists. Then also $y\in\Cone\{v_j\DS j\in J\}$.
		Write 
		\[
			\frac{x_n}{\|x_n\|}=\sum_{j\in J}\beta_{n,j}v_j,\quad y=\sum_{j\in J}\beta_jv_j.
		\]
		The map $\varphi\DF(\beta_j)_{j\in J}\mapsto\sum_{j\in J}\beta_jv_j$ is a linear bijection of $\bb R^d$ onto
		itself, hence a homeomorphism. Thus 
		\[
			\forall j\in J\DP \lim_{n\to\infty}\beta_{n,j}=\beta_j
		\]
		Using \cref{S34} we obtain 
		\[
			n\leq\|x_n\|=\Big(x_n,\frac{x_n}{\|x_n\|}\Big)\leq\sum_{j\in J}\beta_{n,j}\gamma_j.
		\]
		The right side converges to $\sum_{j\in J}\beta_j\gamma_j$ while the left side is unbounded, 
		and we have reached a contradiction.
	\end{Elist}
\end{Remark}

\begin{proof}[Proof of ``\,$\supseteq$'' in \Cref{S24}]
	Let $H\subseteq\bb R^n$ with $\delta(H)\leq k-1$ be given. The construction of $\Phi\in\HomCSL(\mc Cn,\mc Ck)$ with 
	${\sf H}(\Phi)=H+\Span\{\E[n]\}$ will be carried out in four steps:
	\begin{Itemize}
	\item In Step~\ding{192} we employ \Cref{S19} to built an appropriate base shape. 
	\item In Step~\ding{193} we modify the base shape to construct convex bodies 
		$K_1,\ldots,K_n\in\Moewe[35][-15][9.16]{\bb R^{k-1}}$.
	\item In Step~\ding{194} we show that the homomorphism $\Phi_0\in\HomCSL(\mc Cn,\Moewe[35][-15][9.16]{\bb R^{k-1}})$ with
		$\Phi_0(\{e_j\})=K_j$, $j=1,\ldots,n$, satisfies ${\sf H}(\Phi_0)=H+\Span\{b\}$. 
	\item In Step~\ding{195} we finish the argument by making a minor modificationto consider $\Phi_0$ as a homomorphism into $\mc Ck$. 
	\end{Itemize}
	Let us now proceed to the details of the proof.

	The starting point is to apply \Cref{S18} with the set $H$. This yields $T\in\bb N$ and linearly independent sets 
	$M_0,\ldots,M_T\subseteq\bb R^n$ such that 
	\[
		H=\bigcup_{t=0}^T\Cone M_t.
	\]
	Clearly, $d_t\DE|M_t|\leq k-1$ for all $t$.
	\begin{Steps}
	\item
	We use the number $T$ from above and apply \Cref{S19} with $d\DE k-1$ and $N\DE(4T+1)(k-1)$. This yields
	\begin{equation}
	\label{S21}
		v_0,\ldots,v_N\in\bb R^d,\quad \gamma_0,\ldots,\gamma_N>0,\quad 
		\mc J\subseteq\mc P(\{0,\ldots,N\})\setminus\{\emptyset\},
	\end{equation}
	that satisfy {\rm(A1)}--{\rm(A4)}. 
	By \Cref{S20}(ii) we find $\varepsilon>0$ such that $v_0,\ldots,v_N$ and $\mc J$ from \cref{S21} together with any 
	$\gamma_0',\ldots,\gamma_N'>0$ subject to
	\[
		\forall l\in\{0,\ldots,N\}\DP |\gamma_l'-\gamma_l|\leq\varepsilon
	\]
	fullfill {\rm(A2)}. Let us point out that these data do not depend on the sets $M_0,\ldots,M_T$ but only on their number $T$.

	Choose $J\in\mc J$ and $i\in\{0,\ldots,N\}$ such that 
	\[
		\sum_{j\in J}d_{\mc G}(j,i)\geq N.
	\]
	Since $|J|=d$, at least one summand in this sum must be larger or equal to $\frac Nd=4T+1$. Now choose $i'\in\{0,\ldots,N\}$ with 
	$d_{\mc G}(i',i)=4T$, and let $i_1,\ldots,i_{4T-1}\in\{0,\ldots,N\}$ such that, with $i_0\DE i'$, $i_{4T}\DE i$,
	\[
		(i_0,i_1),(i_1,i_2),\ldots,(i_{4T-2},i_{4T-1}),(i_{4T-1},i_{4T})
	\]
	forms a shortest path in $\mc G$ connecting $i'$ and $i$. For $t\in\{0,\ldots,T\}$ choose $J_t\in\mc J$ with 
	$i_{4t}\in J_t$.

	We have $d_{\mc G}(i_s,i_{s'})=|s-s'|$ for $s,s'\in\{0,\ldots,4T\}$, and it follows that
	\[
		\forall t,t'\in\{0,\ldots,T\},t\neq t'\DP d_{\mc G}(i_{4t},i_{4t'})\geq 4.
	\]
	and in turn
	\[
		\forall t,t'\in\{0,\ldots,T\},t\neq t'\DQ\forall j\in J_t,j'\in J_{t'}\DP d_{\mc G}(j,j')\geq 2.
	\]
	Equivalently we can say that each $J\in\mc J$ intersects at most one of the sets $J_0,\ldots,J_T$. In particular, the
	sets $J_t$ are pairwise disjoint.
	\item
	Write $M_t=\{x_{t,1},\ldots,x_{t,d_t}\}$ and $x_{t,i}=(\xi_{t,i}^{(j)})_{j=1}^n$. We can multiply each $x_{t,i}$ 
	with some positive constant without changing the properties of $M_t$. Hence, we may assume w.l.o.g.\ that 
	\[
		\forall t\in\{0,\ldots,T\},i\in\{1,\ldots,d_t\},j\in\{1,\ldots,n\}\DP 
		\big|\xi_{t,i}^{(j)}\big|\leq\varepsilon.
	\]
	Write $J_t=\{l_{t,1},\ldots,l_{t,d}\}$ and recall that $d_t\leq k-1=d$ for all $t$. We define for $j\in\{1,\ldots,n\}$ 
	and $l\in\{0,\ldots,N\}$
	\[
		\gamma_l^{(j)}\DE
		\begin{cases}
			\gamma_l+\xi_{t,i}^{(j)} \CAS l=l_{t,i}\text{ for some }t\in\{0,\ldots,T\},i\in\{1,\ldots,d_t\}
			\\
			\gamma_l \CASO
		\end{cases}
	\]
	and set 
	\[
		K_j\DE\big\{x\in\bb R^d\DS \forall l\in\{0,\ldots,N\}\DP(x,v_l)\leq\gamma_l^{(j)}\big\}.
	\]
	\item
	Let $\Phi_0\in\HomCSL(\mc Cn,\Moewe[20][0][9.16]{\bb R^d})$ be the homomorphism with $\Phi_0(\{e_j\})=K_j$, 
	$j=1,\ldots,n$. In order to compute ${\sf H}(\Phi_0)$ we use the support function representation 
	$\rho_d\DF\Moewe[18][2]{\bb R^d}\to\bb R^{(\bb R^d)}$. The map $h_{\rho_d\circ\Phi_0}$ acts as 
	\[
		h_{\rho_d\circ\Phi_0}(\omega)=\big(\sigma(K_j,\omega)\big)_{j=1}^n.
	\]
	The action of the support function is known from \cref{S34}, and 
	we obtain for every $J\in\mc J$ and $(\beta_l)_{l\in J}\in[0,\infty)^J$ 
	\begin{align*}
		h_{\rho_d\circ\Phi_0}\Big(\sum_{l\in J}\beta_lv_l\Big)= &\, \Big(\sum_{l\in J}\beta_l\gamma_l^{(j)}\Big)_{j=1}^n
		\\
		= &\, \Big[\sum_{l\in J}\beta_l\gamma_l\Big]\cdot\E[n]+
		\sum_{t=0}^T \sum\limits_{\substack{i\in\{1,\ldots,d_t\}\\ l_{t,i}\in J}}\beta_{l_{t,i}} 
		\underbrace{\big(\xi_{t,i}^{(j)}\big)_{j=1}^n}_{=x_{t,i}}
	\end{align*}
	Since $J$ intersects at most one set $J_t$, $t\in\{0,\ldots,T\}$, we obtain 
	\[
		h_{\rho_d\circ\Phi_0}\big(\Cone\{v_j\DS j\in J\}\big)\subseteq
		\Span\{\E[n]\}+\bigcup_{t=0}^T\Cone\{x_{t,1},\ldots,x_{t,d_t}\}=H+\Span\{\E[n]\}.
	\]
	Property {\rm(A3)} yields 
	\[
		h_{\rho_d\circ\Phi_0}(\bb R^d)=\bigcup_{J\in\mc J} h_{\rho_d\circ\Phi_0}\big(\Cone\{v_j\DS j\in J\}\big)\subseteq 
		H+\Span\{\E[n]\}.
	\]
	On the other hand, 
	\begin{align*}
		\Cone\{x_{t,1},\ldots,x_{t,d_t}\}\subseteq &\, 
		h_{\rho_d\circ\Phi_0}\big(\Cone\{v_j\DS j\in J_t\}\big)+\Span\{\E[n]\}
		\\
		\subseteq &\, h_{\rho_d\circ\Phi_0}(\bb R^d)+\Span\{\E[n]\},
	\end{align*}
	and therefore
	\[
		H=\bigcup_{t=0}^T\Cone\{x_{t,1},\ldots,x_{t,d_t}\}\subseteq h_{\rho_d\circ\Phi_0}(\bb R^d)+\Span\{\E[n]\}.
	\]
	Together, thus 
	\[
		h_{\rho_d\circ\Phi_0}(\bb R^d)+\Span\{\E[n]\}=H+\Span\{\E[n]\},
	\]
	and since $h_{\rho_d\circ\Phi_0}$ is piecewise linear we can repeat the argument that led to \cref{S93} and obtain
	\[
		{\sf H}(\Phi_0)=h_{\rho_d\circ\Phi_0}\big(\bb R^d\big)+\Span\{\E[n]\}.
	\]
	\item
	To complete the proof it remains to add an affine transformation in order to change the codomain of $\Phi_0$. Let $v_j$, 
	$j\in\{0,\ldots,d\}$ be as in \cref{S46}. The set $\bigcup_{j=1}^n K_j$ is bounded, and hence we can choose $\mu>0$ such that 
	\[
		\bigcup_{j=1}^n K_j\subseteq\Conv\big\{\mu v_j\DS j\in\{0,\ldots,d\}\big\}\ED S.
	\]
	Therefore, we can consider $\Phi_0$ as an element of $\HomCSL(\mc Cn,\Moewe[12][8][8.33]{S})$. 

	The maps $\kappa_k\DF\mc Dk\to\Conv(\{0\}\cup\{e_1,\ldots,e_d\})\subseteq\bb R^d$ and 
	\[
		\alpha\DF
		\left\{
		\begin{array}{rcl}
			\Conv\big(\{0\}\cup\{e_1,\ldots,e_d\}\big) & \to & S
			\\
			x & \mapsto & \mu(x-c)
		\end{array}
		\right.
	\]
	are isomorphisms of convex algebras. Thus 
	$\Moewe[14][6][6.25]{\alpha}\circ\Moewe[18][2][6.25]{\kappa_k}\DF\mc Ck\to\Moewe[12][8][8.33]{S}$ is an isomorphism of convex 
	semilattices. Now set 
	\[
		\Phi\DE(\,\Moewe[14][6][6.25]{\alpha}\circ\Moewe[18][2][6.25]{\kappa_k})^{-1}\circ\Phi_0.
	\]
	Then $\Phi\in\HomCSL(\mc Cn,\mc Ck)$ and by \Cref{S30} we have ${\sf H}(\Phi)={\sf H}(\Phi_0)$.
	\end{Steps}
\end{proof}

\noindent

We finish with a corollary of \Cref{S24} which gives information about the behaviour of $\mc H(n,k)$ with varying $k$. 

\begin{corollary}
\label{S67}
	Let $n\in\bb N\setminus\{0\}$. Then 
	\[
		\mc H(n,1)\subsetneq\ldots\subsetneq\mc H(n,n)=\mc H(n,n+1)=\ldots
	\]
\end{corollary}
\begin{proof}
	This is seen by putting together \Cref{S55}, \cref{S58}, and \Cref{S39}.
\end{proof}

\section{Two notions of dimension}

We study two notions of dimension for convex semilattices.

\begin{Definition}
\label{S23}
	Let $\bb X=\langle X,+_p,\oplus\rangle$ be a convex semilattice. We define the \emph{embedding dimension} of $\bb X$ as the number
	\[
		\Edim\bb X\DE\inf\big\{k\in\bb N\setminus\{0\}\DS
		\exists \tau\in\HomCSL(\bb X,\mc Ck)\DP\tau\text{ injective}\big\},
	\]
	where the infimum of the empty set is understood as $\infty$.
\end{Definition}

\begin{Definition}
\label{S69}
	Let $\bb X=\langle X,+_p,\oplus\rangle$ be a convex semilattice. We define the \emph{generator dimension} of $\bb X$ as the number
	\[
		\Gdim\bb X\DE\inf\big\{n\in\bb N\setminus\{0\}\DS
		\exists \tau\in\HomCSL(\mc Cn,\bb X)\DP\tau\text{ surjective}\big\},
	\]
	where the infimum of the empty set is understood as $\infty$.
\end{Definition}

\noindent
Observe that
\begin{equation}
\label{S29}
	\Edim\bb X=1\ \Leftrightarrow\ |X|=1\ \Leftrightarrow\ \Gdim\bb X=1.
\end{equation}

\subsection{About the embedding dimension}

Since $\mc Ck\in\mc W$ (recall \Cref{S80}) the embedding dimension $\Edim\bb X$ can be finite only if $\bb X\in\mc W$. 

For a finitely generated convex semilattice $\bb X$ we can use \Cref{S3} and \Cref{S24} to compute $\Edim\bb X$. 

\begin{Theorem}
\label{S35}
	Let $\bb X\in\mc W$ and assume that we have $n\in\bb N\setminus\{0\}$ and a 
	surjective homomorphism $\Lambda\in\HomCSL(\mc Cn,\bb X)$. 
	\begin{Enumerate}
	\item If $\delta({\sf H}(\Lambda))\neq 2$, then $\Edim\bb X=\delta({\sf H}(\Lambda))$.
	\item If $\delta({\sf H}(\Lambda))=2$, then 
		\[
			\Edim\bb X= 
			\begin{cases}
				2 \CAS\ 
				\parbox[t]{90mm}{${\sf H}(\Lambda)=\big(\Cone\{y_1\}\cup\Cone\{y_2\}\big)+\Span\{\E[n]\}$\\
				with some $y_1,y_2\in\bb R^n$},
				\\[2mm]
				3 \CASO.
			\end{cases}
		\]
	\end{Enumerate}
\end{Theorem}

\noindent
All possible values may occur. We refer to: 
\Cref{S40} for $\Edim\bb X=\delta({\sf H}(\Lambda))\in\bb N\setminus\{0\}$, 
\Cref{S48} for $\Edim\bb X=3$ while $\delta({\sf H}(\Lambda))=2$, 
and \Cref{S105} for $\Edim\bb X=\delta({\sf H}(\Lambda))=\infty$.

Reading \Cref{S35} in another way we see that the number $\delta({\sf H}(\Lambda))$ is determined by $\Edim\bb X$ unless 
$\Edim\bb X=3$ in which case we only know that $\delta({\sf H}(\Lambda))\in\{2,3\}$. 
We do not have an example of a convex semilattice $\bb X$ with $\Edim\bb X=3$ and surjective homomorphisms 
$\Lambda\in\HomCSL(\mc Cn,\bb X)$ and $\Lambda'\in\HomCSL(\mc Cn',\bb X)$ such that $\delta({\sf H}(\Lambda))=2$ and 
$\delta({\sf H}(\Lambda'))=3$.

The connection with the results from the previous sections is made as follows.

\begin{lemma}
\label{S28}
	Let $\bb X\in\mc W$ and assume that we have $n\in\bb N\setminus\{0\}$ and a 
	surjective homomorphism $\Lambda\in\HomCSL(\mc Cn,\bb X)$. Then
	\begin{multline*}
		\big\{k\in\bb N\setminus\{0\}\DS\exists \Psi\in\HomCSL(\bb X,\mc Ck)\text{ injective}\big\}
		\\
		=\big\{k\in\bb N\setminus\{0\}\DS {\sf H}(\Lambda)\in\mc H(n,k)\big\}.
	\end{multline*}
\end{lemma}
\begin{proof}
	Let $k\in\bb N\setminus\{0\}$. We read a diagram of the form
	\[
		\begin{tikzcd}[column sep=large]
			\mc Cn \arrow[r,two heads,"\Lambda"] \arrow[rd,dashed,swap,"\Phi"]
			& \bb X \arrow[d,hookrightarrow,dashed,"\Psi"]
			\\
			& \mc Ck
		\end{tikzcd}
	\]
	in two ways. 

	Assume that we have an injective homomorphism $\Psi\DF\bb X\to\mc Ck$, and set $\Phi\DE\Psi\circ\Lambda$. Then 
	$\ker\Lambda=\ker\Phi$, and by \Cref{S30} thus 
	\[
		{\sf H}(\Lambda)={\sf H}(\Phi)\in\mc H(n,k).
	\]
	Assume that ${\sf H}(\Lambda)\in\mc H(n,k)$. Choose $\Phi\in\HomCSL(\mc Cn,\mc Ck)$ with ${\sf H}(\Phi)={\sf H}(\Lambda)$. 
	\Cref{S31}(ii) provides an injective homomorphism $\Psi\DF\bb X\to\mc Ck$. 
\end{proof}

\begin{proof}[Proof of \Cref{S35}]
	Note that the set ${\sf H}(\Lambda)$ satisfies \cref{S36} and that always $\delta({\sf H}(\Lambda))\geq 1$. 
	\begin{Ilist}
	\item We show that $\Edim\bb X\geq\delta({\sf H}(\Lambda))$. If $\Edim\bb X=\infty$ there is nothing to prove, hence assume that 
		$\Edim\bb X<\infty$. By \Cref{S28} we have ${\sf H}(\Lambda)\in\mc H(n,\Edim\bb X)$. 
		\Cref{S24} implies that $\delta({\sf H}(\Lambda))\leq\Edim\bb X$.
	\item We show that $\Edim\bb X\leq\max\{3,\delta({\sf H}(\Lambda))\}$. If $\delta({\sf H}(\Lambda))=\infty$ there is nothing to 
		prove, hence assume that $\delta({\sf H}(\Lambda))<\infty$. Set $k\DE\max\{3,\delta({\sf H}(\Lambda))\}$, 
		then \Cref{S24} yields ${\sf H}(\Lambda)\in\mc H(n,k)$. By \Cref{S28}, thus, 
		$\Edim\bb X\leq k$.
	\item Assume that $\delta({\sf H}(\Lambda))=1$. Then ${\sf H}(\Lambda)\notin\mc H(n,1)$. 
		By \cref{S56} we have ${\sf H}(\Lambda)=\Span\{\E[n]\}$ and \Cref{S55}(i) yields 
		${\sf H}(\Lambda)\in\mc H(n,1)$. Referring again to \Cref{S28} we obtain $\Edim\bb X=1$. 
	\item Assume that $\delta({\sf H}(\Lambda))=2$. By \Cref{S28} and \Cref{S55}(iii) we have 
		\begin{align*}
			\Edim\bb X=2 &\ \Leftrightarrow\ {\sf H}(\Lambda)\in\mc H(n,2)
			\\
			&\ \Leftrightarrow\ 
			\exists y_1,y_2\in\bb R^n\DP {\sf H}(\Lambda)=\big(\Cone\{y_1\}\cup\Cone\{y_2\}\big)+\Span\{\E[n]\}
		\end{align*}
	\end{Ilist}
\end{proof}

\noindent
As a first application of \Cref{S35} we compute $\Edim\mc Cn$.

\begin{proposition}
\label{S40}
	Let $n\in\bb N\setminus\{0\}$. Then $\Edim\mc Cn=n$. 
\end{proposition}

\noindent
In the proof we use the following lemma.

\begin{lemma}
\label{S44}
	Let $\varphi\in\HomCA(\mc Dn,\mc Dk)$. Then ${\sf H}(\,\Moewe[12][8][5.9]\varphi\,)$ is a linear subspace of $\bb R^n$ and 
	\begin{equation}
	\label{S70}
		\dim\big({\sf H}(\,\Moewe[12][8][5.9]\varphi\,)\big)=\dim\Big(\Span\big\{\varphi(e_j)\DS j\in\{1,\ldots,n\}\big\}\Big).
	\end{equation}
\end{lemma}
\begin{proof}
	Denote $a_j\DE\varphi(e_j)$ and $\Phi\DE\Moewe[12][8][5.9]\varphi$. Then $\Phi(\{e_j\})=\{a_j\}$. To compute 
	${\sf H}(\Phi)$ we use the support function representation $\rho_k\DF\mc Ck\to\bb R^{(\bb R^k)}$. Since 
	\[
		\sigma\big(\{a_j\},\omega\big)=(a_j,\omega),
	\]
	we obtain that $h_{\rho_k\circ\Phi}\DF\bb R^k\to\bb R^n$ acts as 
	\[
		h_{\rho_k\circ\Phi}(\omega)=\big((a_j,\omega)\big)_{j=1}^n.
	\]
	This is a linear map, and its matrix representation is 
	\[
		h_{\rho_k\circ\Phi}=\begin{pmatrix} a_1^T\\ \vdots\\ a_n^T\end{pmatrix}\in\bb R^{n\times k}.
	\]
	The rank of this matrix is $\dim(\Span\{a_1,\ldots,a_n\})$. Hence $h_{\rho_k\circ\Phi}(\bb R^k)$ is a linear subspace 
	of $\bb R^n$ whose dimension is equal to the dimension written on the right side of \cref{S70}. 
	Since $a_j\in\mc Dk$, we have $h_{\rho_k\circ\Phi}(\E[k])=\E[n]$. This shows that 
	$\E[n]\in h_{\rho_k\circ\Phi}(\bb R^k)$, and we conclude that ${\sf H}(\Phi)=h_{\rho_k\circ\Phi}(\bb R^k)$.
\end{proof}

\begin{proof}[Proof of \Cref{S40}]
	The inequality ``$\leq$'' is clear, since we can use $\Id_{\mc Cn}\DF\mc Cn\to\mc Cn$ as a witness. For the proof of the
	reverse inequality note that $\Id_{\mc Cn}=\Moewe[28][-8]{\Id_{\mc Dn}}$. Thus \Cref{S44} and \Cref{S43} yield 
	$\delta({\sf H}(\Id_{\mc Cn}))=\dim {\sf H}(\Id_{\mc Cn})=n$, and \Cref{S35} applied with the number $n$ and 
	$\Lambda\DE\Id_{\mc Cn}$ gives the inequality ``$\geq$''.
\end{proof}

\noindent
The result from this proposition can be lifted to a larger class of algebras.

\begin{corollary}
\label{S98}
	Let $n\in\bb N$, $n\geq 3$, and let $C\subseteq\mc Dn$ be a convex subset with nonempty relative interiour (i.e., 
	nonempty interiour with respect
	to the subspace topology of $\mc Dn$ inherited from $\bb R^n$). Then $\Edim\Moewe[13][8][8.33]C=n$.
\end{corollary}
\begin{proof}
	We have $\Moewe[13][8][8.33]C\subseteq\Moewe[21][-2][8.33]{\mc Dn}=\mc Cn$, and hence 
	$\Edim\Moewe[13][8][8.33]C\leq\Edim\mc Cn=n$.
	
	Since $C$ has nonempty relative interiour, we find a translation $T_a$ and a rescaling $M_\lambda$ such that 
	$(T_a\circ M_\lambda)(\mc Dn)\subseteq C$. The map $\Moewe[17][4][8.33]{T_a}\circ\Moewe[21][-2][8.33]{M_\lambda}$ is an
	automorphism of $\Moewe[21][-2][8.33]{\bb R^n}$ and 
	$(\Moewe[17][4][8.33]{T_a}\circ\Moewe[21][-2][8.33]{M_\lambda})(\mc Cn)\subseteq\Moewe[13][8][8.33]C$. It follows that 
	\[
		n=\Edim\mc Cn=\Edim(\Moewe[17][4][8.33]{T_a}\circ\Moewe[21][-2][8.33]{M_\lambda})(\mc Cn)\leq
		\Edim\Moewe[13][8][8.33]C.
	\]
\end{proof}

\noindent
As a second application of \Cref{S35} we compute the embedding dimension of a product $\mc Cn\times\mc Cm$.

\begin{proposition}
\label{S48}
	Let $n,m\in\bb N\setminus\{0\}$. Then
	\begin{equation}
	\label{S50}
		\Edim\big(\mc Cn\times\mc Cm\big)=
		\begin{cases}
			3 \CAS n=m=2,
			\\[1mm]
			\max\{n,m\} \CASO.
		\end{cases}
	\end{equation}
\end{proposition}
\begin{proof}
	Without loss of generality we assume that $m\leq n$. If $m=1$ then $\mc Cn\times\mc Cm\cong\mc Cn$ and hence \Cref{S40} applies. 
	For the rest of the proof we assume that $m\geq 2$. 

	For notational convenience we denote vectors of dimension $nm$ not as one column of height $nm$, but as
	rectangular matrix with $n$ rows and $m$ columns. Then $\mc D(nm)$ becomes the set of $n\times m$-matrices
	\[
		\mc D(nm)=\Big\{(\alpha_{i,j})_{\substack{i=1,\ldots,n\\ j=1,\ldots,m}}\,\Big|\,
		\alpha_{i,j}\geq 0,\sum_{i=1}^n\sum_{j=1}^m\alpha_{i,j}=1\Big\}.
	\]
	Correspondingly, we denote 
	\[
		e_{l,k}\DE(\alpha_{i,j})_{\substack{i=1,\ldots,n\\ j=1,\ldots,m}}
		\quad\text{with}\quad
		\alpha_{i,j}\DE
		\begin{cases}
			1 \CAS (i,j)=(l,k),
			\\
			0 \CASO.
		\end{cases}
	\]
	Using this notation, $\mc C(nm)$ is free with basis
	\[
		\big\{\{e_{i,j}\}\DSb i\in\{1,\ldots,n\},j\in\{1,\ldots,m\}\big\}.
	\]
	Now let $\Lambda\DF\mc C(nm)\to\mc Cn\times\mc Cm$ be the homomorphism that satisfies 
	\[
		\forall i\in\{1,\ldots,n\},j\in\{1,\ldots,m\}\DP \Lambda(\{e_{i,j}\})=\big(\{e_i\},\{e_j\}\big)
	\]
	The set $\{(\{e_i\},\{e_j\})\DS i\in\{1,\ldots,n\},j\in\{1,\ldots,m\}\}$ generates $\mc Cn\times\mc Cm$, and hence 
	$\Lambda$ is surjective.

	Let $\Xi\DF\bb R^{(\bb R^n)}\times\bb R^{(\bb R^m)}\to\bb R^{\bb R^n\dot\cup\,\bb R^m}$ be defined as
	\[
		\big[\Xi\big((f,g)\big)\big](\omega)\DE
		\begin{cases}
			f(\omega) \CAS \omega\in\bb R^n,
			\\
			g(\omega) \CAS \omega\in\bb R^m.
		\end{cases}
	\]
	Then $\Xi$ is an isomorphism of convex semilattices. 
	Now define 
	\[
		\tau\DE\Xi\circ(\rho_n\times\rho_m)\DF\mc Cn\times\mc Cm\to\bb R^{\bb R^n\dot\cup\,\bb R^m}.
	\]
	Then $\tau$ is an injective homomorphism. 
	\[
		\begin{tikzcd}[column sep=large]
			\mc C(nm) \arrow[r,two heads,"\Lambda"]
			& \mc Cn\times\mc Cm \arrow[r,hookrightarrow,swap,"\rho_n\times\rho_m"]
				\arrow[rr,hookrightarrow,bend left=20,dashed,"\tau"]
			& \bb R^{(\bb R^n)}\times\bb R^{(\bb R^m)} \arrow{r}{\cong}[swap]{\Xi}
			& \bb R^{\bb R^n\dot\cup\,\bb R^m}
		\end{tikzcd}
	\]
	The rest of the argument is a variation of the proof of \Cref{S44}.
	Denote by $\zeta_r\DF\bb R^n\to\bb R^{n\times m}$ and $\zeta_c\DF\bb R^m\to\bb R^{n\times m}$ the linear maps with 
	\[
		\zeta_r(e_i)\DE\sum_{j=1}^m e_{i,j}=
		\begin{pmatrix} 0\\ {\scriptstyle 1\ \cdots\cdots\cdots\ 1}\\ 0\end{pmatrix}
		,\quad
		\zeta_c(e_j)\DE\sum_{i=1}^n e_{i,j}=
		\begin{pmatrix} 0\mkern-7mu & 
			\begin{array}{c}{\scriptstyle 1}\\[-3pt] \vdots\\[-3pt] {\scriptstyle 1}\end{array} & 
			\mkern-7mu 0
		\end{pmatrix}
		,
	\]
	where the $1$'s in the first matrix appear in the $i$-th row, and those in the second matrix in the $j$-th column.
	Clearly, $\zeta_r$ and $\zeta_c$ are injective. 

	We have, for all $i\in\{1,\ldots,n\}$ and $j\in\{1,\ldots,m\}$,
	\[
		\big((\rho_n\times\rho_m)\circ\Lambda\big)(\{e_{i,j}\})]=\big((e_i,\Dummy),(e_j,\Dummy)\big),
	\]
	and hence
	\[
		\big[(\tau\circ\Lambda)(\{e_{i,j}\})\big](\omega)=
		\begin{cases}
			(e_i,\omega) \CAS \omega\in\bb R^n,
			\\
			(e_j,\omega) \CAS \omega\in\bb R^m.
		\end{cases}
	\]
	Therefore
	\[
		h_{\tau\circ\Lambda}(\omega)=
		\begin{cases}
			\big((e_i,\omega)\big)_{\substack{i=1,\ldots,n\\ j=1,\ldots,m}}\CAS \omega\in\bb R^n
			\\[4mm]
			\big((e_j,\omega)\big)_{\substack{i=1,\ldots,n\\ j=1,\ldots,m}}\CAS \omega\in\bb R^m
		\end{cases}
		=
		\begin{cases}
			\zeta_r\big((e_i,\omega)_{i=1}^n\big) \CAS \omega\in\bb R^n
			\\[4mm]
			\zeta_c\big((e_j,\omega)_{j=1}^m\big) \CAS \omega\in\bb R^m
		\end{cases}
	\]
	We see that $h_{\tau\circ\Lambda}|_{\bb R^n}$ and $h_{\tau\circ\Lambda}|_{\bb R^m}$ are linear. 
	Moreover, $\zeta_r(\E[n])=\E[nm]=\zeta_c(\E[m])$. 
	For $\lambda\geq 0$, $\omega\in\bb R^n$, and $\alpha\in\bb R$ we have 
	\[
		\lambda h_{\tau\circ\Lambda}(\omega)+\alpha \E[nm]
		=\lambda\zeta_r\big((\omega,e_i)_{i=1}^n\big)+\frac\alpha m\zeta_r(\E[n])
		=\zeta_r\Big(\lambda(\omega,e_i)_{i=1}^n+\alpha\E[n]\Big),
	\]
	and it follows that
	\[
		\big\{\lambda h_{\tau\circ\Lambda}(\omega)+\alpha\E[nm]\DSb
		\lambda\geq 0,\omega\in\bb R^n,\alpha\in\bb R\big\}=\zeta_r(\bb R^n).
	\]
	Analogously,
	\[
		\big\{\lambda h_{\tau\circ\Lambda}(\omega)+\alpha \E[nm]\DSb
		\lambda\geq 0,\omega\in\bb R^m,\alpha\in\bb R\big\}=\zeta_c(\bb R^m),
	\]
	and therefore
	\[
		{\sf H}(\Lambda)=\zeta_r(\bb R^n)\cup\zeta_c(\bb R^m).
	\]
	\Cref{S42}(iii) and \Cref{S43} yield $\delta({\sf H}(\Lambda))=n$.

	If $n\geq 3$, \Cref{S35} gives $\Edim(\mc Cn\times\mc Cm)=\delta({\sf H}(\Lambda))=n$. If $n=2$ we have to
	further investigate the set ${\sf H}(\Lambda)$.
	The crucial observation is the following: if we have $\beta\geq 0,\alpha\in\bb R$, and 
	$y=(\eta_{ij})_{i,j=1}^2\in\bb R^{2\times 2}$ with $\sum_{i,j=1}^2\eta_{ij}=0$, then 
	\begin{align*}
		\zeta_r(e_1)=\beta y+\alpha\E[4]& \ \Rightarrow\ 
		\alpha=\frac 12 &\mkern-30mu \wedge\ \eta_{11}>0,\eta_{12}>0,\eta_{21}<0,\eta_{22}<0
		\\
		-\zeta_r(e_1)=\beta y+\alpha\E[4]& \ \Rightarrow\ 
		\alpha=-\frac 12 &\mkern-30mu \wedge\ \eta_{11}<0,\eta_{12}<0,\eta_{21}>0,\eta_{22}>0
		\\
		\zeta_c(e_1)=\beta y+\alpha\E[4]& \ \Rightarrow\ 
		\alpha=\frac 12 &\mkern-30mu \wedge\ \eta_{11}>0,\eta_{12}<0,\eta_{21}>0,\eta_{22}<0
		\\
		-\zeta_c(e_1)=\beta y+\alpha\E[4]& \ \Rightarrow\ 
		\alpha=-\frac 12 &\mkern-30mu \wedge\ \eta_{11}<0,\eta_{12}>0,\eta_{21}<0,\eta_{22}>0
	\end{align*}
	Assume we have a representation 
	\[
		{\sf H}(\Lambda)=\bigcup_{l=1}^m\Cone\{y_l\}+\Span\{\E[4]\}.
	\]
	Adding scalar multiples of $\E[4]$ to $y_l$ does not change the set on the right side. 
	Hence, we may assume that for each $y_l$ the sum of its four entries 
	vanishes. Now the above implications show that necessarily $m\geq 4$. \Cref{S35} implies $\Edim(\mc Cn\times\mc Cm)=3$.
\end{proof}

\noindent
It is seen by induction (or by writing out a ``multiindex-version'' of the above proof) that \cref{S50} lifts to finite products.

\begin{corollary}
\label{S45}
	Let $m\geq 2$ and $n_1,\ldots,n_m\in\bb N\setminus\{0\}$. Then 
	\[
		\Edim\Big(\prod_{i=1}^m\mc Cn_i\Big)=
		\begin{cases}
			3 \CAS \parbox[t]{55mm}{$\max\limits_{i=1,\ldots,m}n_i=2$ and there exist at least 
			two indices $i$ with $n_i=2$,}
			\\[8mm]
			\max\limits_{i=1,\ldots,m}n_i \CASO.
		\end{cases}
	\]
\end{corollary}
\begin{proof}
	W.l.o.g.\ assume that $n_1\geq n_2\geq\cdots\geq n_m$. We use induction on $m$. 
	The base case ``$m=2$'' is \cref{S50}. Assume that $m\geq 3$. By \Cref{S40} and the inductive hypothesis we have 
	\begin{align*}
		d\DE &\, \Edim\big(\mc Cn_1\times\mc Cn_2\big)=
		\begin{cases}
			3 \CAS n_1=n_2=2,
			\\
			n_1 \CASO,
		\end{cases}
		\\
		d'\DE &\, \Edim\Big(\prod_{i=3}^m\mc Cn_i\Big)=
		\begin{cases}
			3 \CAS m\geq 4\wedge n_3=n_4=2,
			\\
			n_3 \CASO.
		\end{cases}
	\end{align*}
	We observe two facts. The first is that $d\geq d'$: if $m\geq 4,n_3=n_4=2$ then $n_2\geq 2$ and hence $d\geq 3=d'$, and otherwise 
	$d\geq n_1\geq n_3=d'$. The second is that $(d,d')\neq (2,2)$: if $d'=2$ then $n_3=2$ and hence $n_2\geq 2$ which implies $d\geq 3$.

	We have injective homomorphisms 
	\[
		\mc Cn_1\times\mc Cn_2\to\prod_{i=1}^m\mc Cn_i\cong\big(\mc Cn_1\times\mc Cn_2\big)\times\prod_{i=3}^m\mc Cn_i
		\to\mc Cd\times\mc Cd'.
	\]
	and \cref{S50} yields
	\[
		d=\Edim\big(\mc Cn_1\times\mc Cn_2\big)\leq\Edim\Big(\prod_{i=1}^m\mc Cn_i\Big)\leq
		\Edim\big(\mc Cd\times\mc Cd'\big)=d.
	\]
\end{proof}

\noindent
Let us give an explicit example of an embedding of $\mc C2\times\mc C2$ into $\mc C3$. By composing with $\kappa_3$ and an appropriate
translation and rescaling as usual it is enough to construct an embedding of $\mc C2\times\mc C2$ into $\Moewe[21][-2][8.33]{\bb R^2}$. 

\begin{Example}
\label{S92}
	We define $\Phi\in\HomCSL(\mc C4,\Moewe[21][-2][8.33]{\bb R^2})$ by specifying the images of generators. Here we use the same
	notation and conventions as in the proof of \Cref{S48}. Let $A_{ij}$, $i,j\in\{1,2\}$, be defined as 
	the convex hull of six corner points of a regular octagon according to the following pictures
	
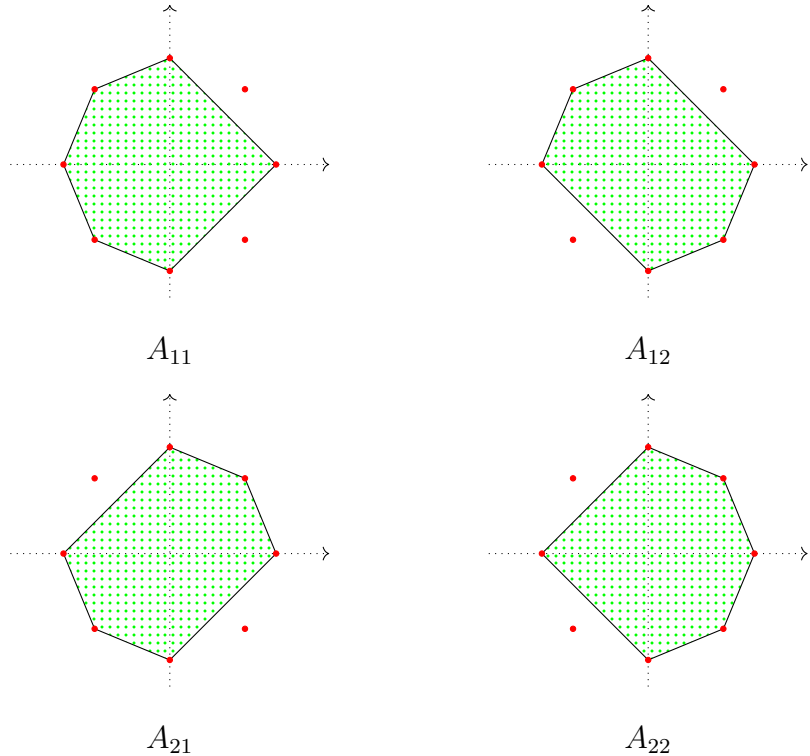
\begin{center}
	\begin{tikzpicture}[x=1pt,y=1pt,scale=1,font=\fontsize{12}{12}]
		\draw[thin,dotted,->] (0,-50)--(0,60);
		\draw[thin,dotted,->] (-60,0)--(60,0);
		\draw[pattern=dots,pattern color=green,thin] 
			({40*cos(0)},{40*sin(0)})--({40*cos(90)},{40*sin(90)})--({40*cos(135)},{40*sin(135)})--
			({40*cos(180)},{40*sin(180)})--({40*cos(225)},{40*sin(225)})--({40*cos(270)},{40*sin(270)})--
			({40*cos(0)},{40*sin(0)});
		\draw[fill,red] ({40*cos(0)},{40*sin(0)}) circle [radius=1];
		\draw[fill,red] ({40*cos(45)},{40*sin(45)}) circle [radius=1];
		\draw[fill,red] ({40*cos(90)},{40*sin(90)}) circle [radius=1];
		\draw[fill,red] ({40*cos(135)},{40*sin(135)}) circle [radius=1];
		\draw[fill,red] ({40*cos(180)},{40*sin(180)}) circle [radius=1];
		\draw[fill,red] ({40*cos(225)},{40*sin(225)}) circle [radius=1];
		\draw[fill,red] ({40*cos(270)},{40*sin(270)}) circle [radius=1];
		\draw[fill,red] ({40*cos(315)},{40*sin(315)}) circle [radius=1];
		\draw (0,-70) node {$A_{11}$};
		\draw[thin,dotted,->] (180,-50)--(180,60);
		\draw[thin,dotted,->] (120,0)--(240,0);
		\draw[pattern=dots,pattern color=green,thin] 
			({180+40*cos(0)},{40*sin(0)})--({180+40*cos(90)},{40*sin(90)})--({180+40*cos(135)},{40*sin(135)})--
			({180+40*cos(180)},{40*sin(180)})--({180+40*cos(270)},{40*sin(270)})--({180+40*cos(315)},{40*sin(315)})--
			({180+40*cos(0)},{40*sin(0)});
		\draw[fill,red] ({180+40*cos(0)},{40*sin(0)}) circle [radius=1];
		\draw[fill,red] ({180+40*cos(45)},{40*sin(45)}) circle [radius=1];
		\draw[fill,red] ({180+40*cos(90)},{40*sin(90)}) circle [radius=1];
		\draw[fill,red] ({180+40*cos(135)},{40*sin(135)}) circle [radius=1];
		\draw[fill,red] ({180+40*cos(180)},{40*sin(180)}) circle [radius=1];
		\draw[fill,red] ({180+40*cos(225)},{40*sin(225)}) circle [radius=1];
		\draw[fill,red] ({180+40*cos(270)},{40*sin(270)}) circle [radius=1];
		\draw[fill,red] ({180+40*cos(315)},{40*sin(315)}) circle [radius=1];
		\draw (180,-70) node {$A_{12}$};
	\end{tikzpicture}
	\end{center}
	\begin{center}
	\begin{tikzpicture}[x=1pt,y=1pt,scale=1,font=\fontsize{12}{12}]
		\draw[thin,dotted,->] (0,-50)--(0,60);
		\draw[thin,dotted,->] (-60,0)--(60,0);
		\draw[pattern=dots,pattern color=green,thin] 
			({40*cos(0)},{40*sin(0)})--({40*cos(45)},{40*sin(45)})--({40*cos(90)},{40*sin(90)})--
			({40*cos(180)},{40*sin(180)})--({40*cos(225)},{40*sin(225)})--({40*cos(270)},{40*sin(270)})--
			({40*cos(0)},{40*sin(0)});
		\draw[fill,red] ({40*cos(0)},{40*sin(0)}) circle [radius=1];
		\draw[fill,red] ({40*cos(45)},{40*sin(45)}) circle [radius=1];
		\draw[fill,red] ({40*cos(90)},{40*sin(90)}) circle [radius=1];
		\draw[fill,red] ({40*cos(135)},{40*sin(135)}) circle [radius=1];
		\draw[fill,red] ({40*cos(180)},{40*sin(180)}) circle [radius=1];
		\draw[fill,red] ({40*cos(225)},{40*sin(225)}) circle [radius=1];
		\draw[fill,red] ({40*cos(270)},{40*sin(270)}) circle [radius=1];
		\draw[fill,red] ({40*cos(315)},{40*sin(315)}) circle [radius=1];
		\draw (0,-70) node {$A_{21}$};
		\draw[thin,dotted,->] (180,-50)--(180,60);
		\draw[thin,dotted,->] (120,0)--(240,0);
		\draw[pattern=dots,pattern color=green,thin] 
			({180+40*cos(0)},{40*sin(0)})--({180+40*cos(45)},{40*sin(45)})--({180+40*cos(90)},{40*sin(90)})--
			({180+40*cos(180)},{40*sin(180)})--({180+40*cos(270)},{40*sin(270)})--({180+40*cos(315)},{40*sin(315)})--
			({180+40*cos(0)},{40*sin(0)});
		\draw[fill,red] ({180+40*cos(0)},{40*sin(0)}) circle [radius=1];
		\draw[fill,red] ({180+40*cos(45)},{40*sin(45)}) circle [radius=1];
		\draw[fill,red] ({180+40*cos(90)},{40*sin(90)}) circle [radius=1];
		\draw[fill,red] ({180+40*cos(135)},{40*sin(135)}) circle [radius=1];
		\draw[fill,red] ({180+40*cos(180)},{40*sin(180)}) circle [radius=1];
		\draw[fill,red] ({180+40*cos(225)},{40*sin(225)}) circle [radius=1];
		\draw[fill,red] ({180+40*cos(270)},{40*sin(270)}) circle [radius=1];
		\draw[fill,red] ({180+40*cos(315)},{40*sin(315)}) circle [radius=1];
		\draw (180,-70) node {$A_{22}$};
	\end{tikzpicture}
	\end{center}
	and let $\Phi$ be the homomorphism with $\Phi(\{e_{i,j}\})=A_{ij}$, $i,j\in\{1,2\}$. 

	In order to compute ${\sf H}(\Phi)$ we use the support function representation 
	$\tau\DF\Moewe[21][-2][8.33]{\bb R^2}\to\bb R^{S^1}$. We have 
	\begin{align*}
		& \forall\theta\in[0,\pi]\DP \sigma(A_{11},e^{i\theta})=\sigma(A_{12},e^{i\theta}) \wedge
		\sigma(A_{21},e^{i\theta})=\sigma(A_{22},e^{i\theta})
		\\
		& \forall\theta\in[0,\pi]\DP \sigma(A_{11},e^{i\theta})=\sigma(A_{21},e^{i\theta}) \wedge
		\sigma(A_{12},e^{i\theta})=\sigma(A_{22},e^{i\theta})
	\end{align*}
	Hence, ${\sf H}(\Phi)\subseteq{\sf H}(\Lambda)$. We have 
	\begin{align*}
		& \sigma(A_{11},e^{i\frac\pi4})<\sigma(A_{21},e^{i\frac\pi4}),
		& \sigma(A_{11},e^{i\frac{3\pi}4})>\sigma(A_{21},e^{i\frac{3\pi}4}),
		\\
		& \sigma(A_{11},e^{i\frac{5\pi}4})>\sigma(A_{12},e^{i\frac{5\pi}4}),
		& \sigma(A_{11},e^{i\frac{7\pi}4})<\sigma(A_{12},e^{i\frac{7\pi}4}).
	\end{align*}
	Hence, ${\sf H}(\Phi)\supseteq{\sf H}(\Lambda)$. 

	\Cref{S31} implies that there exists an injective map $\Psi\in\HomCSL(\mc C2\times\mc C2,\Moewe[21][-2][8.33]{\bb R^2})$ with 
	$\Psi\circ\Lambda=\Phi$.
\end{Example}

\subsection{About the generator dimension}

The generator dimension and the embedding dimension are related as follows.

\begin{proposition}
\label{S51}
	We have 
	\begin{align*}
		\big\{(\Gdim\bb X,\Edim\bb X)\DSb\bb X\in\mc W\big\}= &\, \{(1,1)\}
		\\
		\cup &\, \big\{(n,k)\in(\bb N\cup\{\infty\})\times\bb N\DSb 2\leq k\leq n\big\}
		\\
		\cup &\, \big\{(n,\infty)\DSb n\in\bb N\cup\{\infty\},n\geq 3\big\}.
	\end{align*}
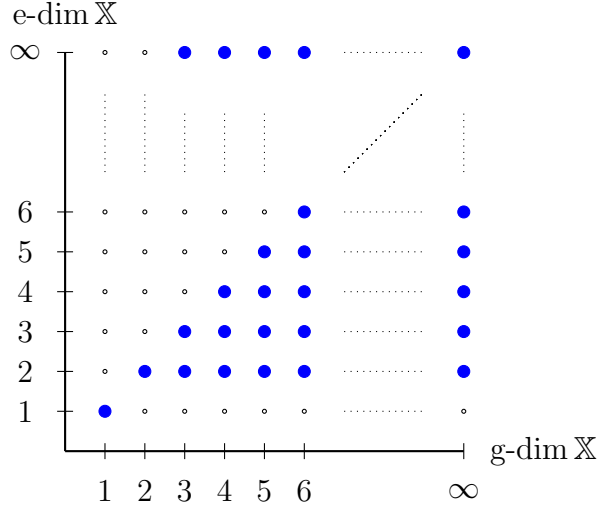
	\begin{center}
	\begin{tikzpicture}[x=1.5pt,y=1.5pt,scale=1,font=\fontsize{12}{12}]
		\draw[thick] (20,20) -- (120,20);
		\draw[thick] (20,20) -- (20,120);
		\draw (140,20) node {$\Gdim\bb X$};
		\draw (20,130) node {$\Edim\bb X$};
		\foreach \x in {1,2,3,4,5,6,10} {
			\draw (\x*10+20,18) -- (\x*10+20,22); 
			\draw (18,\x*10+20) -- (22,\x*10+20);
			\ifthenelse{\x < 10}{\draw (\x*10+20,10) node {$\x$};}{\draw (\x*10+20,10) node {$\infty$};}
			\ifthenelse{\x < 10}{\draw (10,\x*10+20) node {$\x$};}{\draw (10,\x*10+20) node {$\infty$};}
			\foreach \y in {1,2,3,4,5,6,10} {
				\draw (\x*10+20,\y*10+20) circle [radius=0.5];
			}
			\draw[dotted] (90,\x*10+20) -- (110,\x*10+20);
			\draw[dotted] (90,90) -- (110,110);
		}
		\foreach \x in {1,2} {
			\draw[dotted] (\x*10+20,90) -- (\x*10+20,110);
		}
		\foreach \x in {3,4,5,10} {
			\draw[dotted] (\x*10+20,90) -- (\x*10+20,105);
		}
		\draw[color=blue,fill] (30,30) circle [radius=1.5];
		\foreach \x in {2,3,4,5,6,10} {
			\foreach \y in {2,3,4,5,6,10} {
				\ifthenelse{\x=\y \OR \y<\x \OR \(\y=10 \AND \x>2\)}{
				\draw[color=blue,fill] (\x*10+20,\y*10+20) circle [radius=1.5];}{}
			}
		}
	\end{tikzpicture}
	\end{center}
\end{proposition}

\noindent
The inclusion ``$\subseteq$'' is a consequence of our previous results. 

\begin{proof}[Proof of \Cref{S51}, ``\,$\subseteq$'']
	Let $\bb X\in\mc W$.
	Set $n\DE\Gdim\bb X$ and $k\DE\Edim\bb X$ where both numbers are elements of $(\bb N\setminus\{0\})\cup\{\infty\}$.
	By \cref{S29} we have $n=1$ if and only if $k=1$.

	Assume that $n=2$ (and hence $k\geq 2$). 
	Choose $\Omega\neq\emptyset$ such that $\bb X$ embeds into $\bb R^\Omega$, and choose a surjective map 
	$\Lambda\in\HomCSL(\mc C2,\bb X)$. Then ${\sf H}(\Lambda)\in H(2,\Omega)$. By \Cref{S2}(iii),(iv) and \Cref{S55}(ii)
	we have $H(2,\Omega)\subseteq H(2,2)=\mc H(2,2)$. Now \Cref{S28} implies that $k\leq 2$. Together with
	what we already said above, therefore, $k=2$. 

	Assume that $n\geq 3$. If $k=\infty$ there is nothing to prove, hence assume that $k<\infty$. 
	Choose a surjective homomorphism $\Lambda\in\HomCSL(\mc Cn,\bb X)$. By \Cref{S28} we have ${\sf H}(\Lambda)\in\mc H(n,k)$, 
	and by \Cref{S67} it holds that $\mc H(n,k)\subseteq\mc H(n,n)$. Another application of \Cref{S28} yields that $k\leq n$. 
\end{proof}

\noindent
For the proof of the reverse inclusion we give a series of examples. First of all note that the case ``$k=n=1$'' is trivial 
since $\Gdim\mc C1=\Edim\mc C1=1$.
Next we give examples that settle the case ``$\Edim\bb X=2$''.

\begin{Example}
\label{S100}
	Denote again by $\zeta_n^k$ the $n$-th roots of unity, cf.\ \cref{S99}. 
	Let $n\geq 2$ be given and consider the set 
	\[
		X\DE\Conv\Big(\big\{\zeta_{4(n-1)}^l\DSb 2(n-1)\leq l\leq 3(n-1)\big\}\cup\{0\}\Big)\subseteq\bb R^2.
	\]
	\begin{center}
	\begin{tikzpicture}[x=1pt,y=1pt,scale=1.5,font=\fontsize{12}{12}]
		\draw[->] (0,-50)--(0,10);
		\draw[->] (-50,0)--(10,0);
		\draw[pattern=dots,pattern color=green,thick] 
			(0,0)--({40*cos(180)},{40*sin(180)})--({40*cos(202.5)},{40*sin(202.5)})--({40*cos(225)},{40*sin(225)})
			--({40*cos(247.5)},{40*sin(247.5)})--({40*cos(270)},{40*sin(270)})--(0,0)
			;
		\draw[fill,red] (0,0) circle [radius=1];
		\draw[fill,red] ({40*cos(180)},{40*sin(180)}) circle [radius=1];
		\draw[fill,red] ({40*cos(202.5)},{40*sin(202.5)}) circle [radius=1];
		\draw[fill,red] ({40*cos(225)},{40*sin(225)}) circle [radius=1];
		\draw[fill,red] ({40*cos(247.5)},{40*sin(247.5)}) circle [radius=1];
		\draw[fill,red] ({40*cos(270)},{40*sin(270)}) circle [radius=1];
		\draw (20,0) node {$X$};
	\end{tikzpicture}
	\end{center}
	Remembering \Cref{S63} we see that $X$ is a convex semilattice with the operations inherited from $\bb R^2$ 
	(and we write $\bb X$ for this algebra).

	We show that 
	\[
		M\DE\big\{\zeta_{4(n-1)}^l\DSB 2(n-1)\leq l\leq 3(n-1)\big\}
	\]
	is the smallest subset of $X$ that generates $\bb X$. First, observe that
	\[
		0=\zeta_{4(n-1)}^{2(n-1)}\oplus\zeta_{4(n-1)}^{3(n-1)},
	\]
	and hence $M$ is a generating set. 
	Assume we have a generating set $\widetilde M$, and let $l\in\{2(n-1),\ldots,3(n-1)\}$. Choose $x_1,\ldots,x_N\in\Conv\widetilde M$
	such that $\zeta_{4(n-1)}^l=x_1\oplus\cdots,\oplus x_N$. Denote $Q\DE\{\binom\alpha\beta\in\bb R^2\DS\alpha,\beta\leq 0\}$,
	then $x_j\in\zeta_{4(n-1)}^l+Q$ for all $j$. We have $X\cap\zeta_{4(n-1)}^l+Q=\{\zeta_{4(n-1)}^l\}$, and hence 
	$\zeta_{4(n-1)}^l\in\Conv\widetilde M$. Since $\zeta_{4(n-1)}^l$ is an extremal point of $X$, it follows that 
	$\zeta_{4(n-1)}^l\in\widetilde M$. 

	We conclude from the above that $\Gdim\bb X=n$. Let $\tau\DF\mc C2\to\bb R^2$ be the embedding from \Cref{S72}(i), 
	\Cref{S68}. For an appropriate translation $T_a$ and rescaling $M_\lambda$ we have 
	$(T_a\circ M_\lambda)(X)\subseteq\tau(\mc C2)$. Since $T_a$ and $M_\lambda$ are automorphisms of the convex semilattice 
	$\bb R^2$, we have the injective map $\tau^{-1}\circ T_a\circ M_\lambda\in\HomCSL(\bb X,\mc C2)$.

	Now consider the quarter disk 
	\begin{equation}
	\label{S37}
		X\DE\Big\{\binom\alpha\beta\in\bb R^2\DSB \alpha,\beta\leq 0\wedge \alpha^2+\beta^2\leq 1\Big\}
	\end{equation}
	endowed with the operations inherited from $\bb R^2$. The same arguments as above show that 
	\[
		M\DE\Big\{\binom\alpha\beta\in\bb R^2\DSB \alpha,\beta\leq 0\wedge \alpha^2+\beta^2=1\Big\}
	\]
	is the smallest generating set of $\bb X$, whence $\Gdim\bb X=\infty$, and that we have an embedding of $\bb X$ into $\mc C2$.
\end{Example}

\noindent
In order to settle the cases that ``$3\leq\Edim\bb X<\infty$'' we use an example of a different kind. It is based on the following lemma
which is of interest on its own right.

\begin{lemma}
\label{S101}
	Let $d\in\bb N\setminus\{0\}$ and $K\subseteq\bb R^d$ be a compact convex and nonempty subset. Denote by $E(K)$ the set of
	extremal points of $K$, and set
	\[
		E\DE\big\{\{x\}\DSb x\in E(K)\big\}.
	\]
	\begin{Enumerate}
	\item If $M\subseteq\Moewe[14][8][8.33]K$ is a subset that generates $\Moewe[14][8][8.33]K$ as a convex semilattice, then
		$E\subseteq M$.
	\item If $E(K)$ is finite, then $E$ is a generating set for $\Moewe[14][8][8.33]K$.
	\end{Enumerate}
\end{lemma}
\begin{proof}
	Let $x\in E(K)$. Choose $A_1,\ldots,A_N\in\Conv M$ such that $\{x\}=A_1\oplus\ldots\oplus A_N$. Since 
	$A_1\oplus\ldots\oplus A_N\supseteq\bigcup_{j=1}^NA_j$, it follows that $A_i=\{x\}$ for all $i$. Thus $\{x\}\in\Conv M$. 
	Choose $B_1,\ldots,B_l\in M$ and $p_1,\ldots,p_l>0$ with $\sum_{j=1}^lp_j=1$ such that $\{x\}=\sum_{j=1}^lp_jB_j$. 
	Let $b_j\in B_j$, then $\sum_{j=1}^lp_jb_j=x$, and since $x\in E(K)$ it follows that $b_j=x$ for all $j$. We conclude that 
	$B_j=\{x\}$ for all $j$, and hence that $\{x\}\in M$.

	Assume that $E(K)$ is finite, then $K=\Conv E(K)$. Write $E(K)=\{x_1,\ldots,x_l\}$. Let $A\in\Moewe[14][8][8.33]K$, and write 
	$A=\Conv\{a_1,\ldots,a_N\}$ with some $a_i\in K$. Each $a_i$ can be written as a convex combination 
	$a_i=\sum_{j=1}^lp_{ij}x_j$, and we obtain 
	\[
		A=\Big(\sum_{j=1}^lp_{1j}\{x_j\}\Big)\oplus\cdots\oplus\Big(\sum_{j=1}^lp_{Nj}\{x_j\}\Big).
	\]
	Thus $A$ belongs to the subalgebra generated by $E$.
\end{proof}

\begin{Example}
\label{S102}
	Let $k\in\bb N$, $k\geq 3$, and $n\in\bb N\cup\{\infty\}$, $n\geq k$. There exists a compact convex subset $K$ of $\mc Dk$ with 
	nonempty relative interiour and $|E(K)|=n$. For $n=\infty$ use for example some closed ball in $\mc Dk$ with positive radius.
	If $n<\infty$ use induction: start with a smaller copy of $\mc Dk$ which lies in the interiour of $\mc Dk$ and then add corner
	points by shifting midpoints of faces. 

	\Cref{S98} and \Cref{S101} imply that $\Edim\Moewe[14][8][8.33]K=k$ and $\Gdim\Moewe[14][8][8.33]K=n$.
	\phantom{xxx}
\end{Example}

\noindent
It remains to construct $\bb X\in\mc W$ with $\Edim\bb X=\infty$ and $\Gdim\bb X=n$ where $n$ is any given number larger than $2$. 
The case that $n=3$ follows easily from our previous results.

\begin{Example}
\label{S105}
	We define vectors $y_j\in\bb R^3$, $j\in\bb N$, as 
	\[
		y_0\DE\begin{pmatrix} 1 \\ 0 \\ 0\end{pmatrix},\qquad
		y_j\DE\begin{pmatrix} 1-\frac 1j \\ \frac 1j \\ 0\end{pmatrix}\quad\text{for }j\geq 1,
	\]
	and set 
	\[
		H\DE\bigcup_{j\in\bb N}\Cone\{y_j\}+\Span\{\E[3]\}.
	\]
	Since $\lim_{j\to\infty}y_j=y_0$ and the third component of each $y_j$ vanishes, one can easily check that $H$ is closed. 
	By \Cref{S12} we have $H\in H(3,\bb N)$. For each two different indices $j,l\in\bb N$ we have 
	$\Conv\{y_j,y_l\}\nsubseteq H$, and hence $H$ cannot be represented as a finite union of cones. In particular, 
	$\delta(H)=\infty$.

	Choose $\Lambda\in\HomCSL(\mc C3,\bb R^{\bb N})$ such that $H=H(\Lambda)$, and set $\bb X\DE\Lambda(\mc C3)$. Then 
	$\bb X\in\mc W$, $\Gdim\bb X\leq 3$, and $\Edim\bb X=\infty$ by \Cref{S35}. By the already known inclusion ``$\subseteq$'' in
	\Cref{S51} we must have $\Gdim\bb X=3$. 
\end{Example}

\noindent
The case that $n=\infty$ follows immediately.

\begin{Example}
\label{S47}
	Let $\bb X_1$ be the convex semilattice constructed in \Cref{S105}, let $\bb X_2$ be the
	quarter disk \cref{S37} used in \Cref{S100}, and set $\bb X\DE\bb X_1\times\bb X_2$. Since $\HomCSL(\bb X_1,\bb X)$ contains an
	injective map we have $\Edim\bb X=\infty$, and since $\HomCSL(\bb X,\bb X_2)$ contains a surjective map we have 
	$\Gdim\bb X=\infty$.
\end{Example}

\noindent
In order to obtain examples for $3<n<\infty$ we have to argue a little more accurately. 
The example is based on the following construction.

\begin{lemma}
\label{S106}
	Let $\Omega\neq\emptyset$ and let $\bb X$ be a subalgebra of the convex semilattice $\bb R^\Omega$ with
	$X\subseteq[0,\infty)^\Omega$. Consider the disjoint union $\tilde\Omega\DE\Omega\dot\cup\{\star\}$, and let 
	$\iota\DF\bb R^\Omega\to\bb R^{\tilde\Omega}$ be the map defined by 
	\[
		(\iota f)(\omega)\DE
		\begin{cases}
			f(\omega) \CAS \omega\in\Omega
			\\
			0 \CAS \omega=\star
		\end{cases}
	\]
	Then $\iota$ is an injective homomorphism of convex semilattices. 
	Let $\tilde{\bb X}$ be the subalgebra of $\bb R^{\tilde\Omega}$ that is generated by $\iota(X)\cup\{e_\star\}$. 
	\begin{Enumerate}
	\item Let $\tilde M\subseteq\tilde X$ be a generating set for $\tilde{\bb X}$. Then $e_\star\in\tilde M$ and 
		$\iota^{-1}(\tilde M\cap\iota(X))$ is a generating set for $\bb X$.
	\item Let $M\subseteq X$ be a generating set for $\bb X$. Then $\iota(M)\cup\{e_\star\}$ is a generating set for 
		$\tilde{\bb X}$.
	\end{Enumerate}
	Together these items imply that $\Gdim\tilde{\bb X}=\Gdim\bb X+1$.
\end{lemma}
\begin{proof}
	The set $[0,\infty)^\Omega\times[0,1]$ is a subalgebra of $\bb R^{\tilde\Omega}$ that contains $\iota(X)\cup\{e_\star\}$.
	Hence, it also contains $\tilde X$. Let $\tilde M$ be a generating set for $\tilde{\bb X}$. 

	Choose $x_1,\ldots,x_N\in\Conv\tilde M$ such that $e_\star=x_1\oplus\cdots\oplus x_N$. Then $x_i(\omega)=0$ for all $i$ and all
	$\omega\in\Omega$, and there exists at least one $i$ such that $x_i(\star)=1$. For such $i$ we have $x_i=e_\star$, and hence 
	$e_\star\in\Conv\tilde M$. Let $y_1,\ldots,y_l\in\tilde M$ and $p_1,\ldots,p_l>0$ such that $e_\star=\sum_{j=1}^lp_jy_j$. Then
	$y_j(\omega)=0$ for all $j$ and all $\omega\in\Omega$, and $y_j(\star)=1$ for all $j$. This means that $y_j=e_\star$ for all
	$j$, and hence $e_\star\in M$. 

	Let $x\in X$ and repeat the argument.
	Choose $x_1,\ldots,x_N\in\Conv\tilde M$ such that $\iota(x)=x_1\oplus\ldots\oplus x_N$. Then $x_i(\star)=0$ for all $i$. 
	Let $y_{i,1},\ldots,y_{i,l_i}\in\tilde M$ and $p_{i,1},\ldots,p_{i,l_i}>0$ such that $x_i=\sum_{j=1}^{l_i}p_{i,j}y_{i,j}$. 
	Then $y_{i,j}(\star)=0$ for all $i,j$, and we see that $y_{i,j}\in\iota(X)$. Thus $\iota(X)$ is generated by 
	$\tilde M\cap\iota(X)$.

	Item (ii) is clear.
\end{proof}

\begin{Example}
\label{S107}
	Let $n\in\bb N$, $n>3$, be given. we start with the convex semilattice $\bb X$ constructed in \Cref{S105}. By performing an
	appropriate translation, for example $T_a$ where $a\DE\Lambda(\{e_1\})\oplus\Lambda(\{e_2\})\oplus\Lambda(\{e_3\})$, we can 
	achieve that $T_a(X)\subseteq[0,\infty)^{\bb N}$. Now we repeatedly apply \Cref{S106}, in fact, $n-3$ times. This leads to a 
	convex semilattice $\bb Y$ with $T_a(\bb X)\subseteq\bb Y$ and $\Gdim\bb Y=n$. Because of the inclusion we have 
	\[
		\Edim\bb Y\geq\Edim T_a(\bb X)=\Edim\bb X=\infty.
	\]
\end{Example}

\noindent
The proof of \Cref{S51} is complete.

\begin{corollary}
\label{S57}
	Let $n\in\bb N\setminus\{0\}$ and assume we have a convex semilattice $\bb X$ that is generated by $n$ of its elements and that
	$\bb X$ can be embedded into some free convex semilattice. Then $\bb X$ can be embedded into $\mc Cn$.
\end{corollary}
\begin{proof}
	Assume that $\bb X$ can be embedded into $\mc CA$. Since $\mc CA=\Moewe[21][-2][8.33]{\mc DA}$ and $\bb X$ is finitely generated,
	we find a finite subset $A_1$ of $A$ such that the image of $\bb X$ is contained in $\mc CA_1$. Thus $\Edim\bb X<\infty$, and
	by \Cref{S51} therefore $\Edim\bb X\leq\Gdim\bb X=n$.
\end{proof}

\noindent
We give another example that rounds off \Cref{S51}.

\begin{Example}
\label{S104}
	Consider $X\DE\{0,1\}$ with the operations 
	\[
		x\oplus y\DE\max\{x,y\},\qquad x+_py\DE\min\{x,y\}\quad\text{for }p\in(0,1).
	\]
	Then $X$ becomes a convex semilattice $\bb X$, and clearly $\Gdim\bb X=2$. Every convex subset of $\bb R^\Omega$ contains
	either exactly one or infinitely many points. Hence, there cannot exist an injective homomorphism of $\bb X$ into some 
	$\bb R^\Omega$, i.e.\ $\bb X\in\mc W$.
\end{Example}

\subsection{The embedding dimension for convex algebras}

One can of course define embedding and generator dimension in any equational class. 
We briefly discuss the embedding dimension in the equational class of 
convex algebras because it is interesting to compare this with \Cref{S40,S48}.

\begin{proposition}
\label{S49}
	Let $n,m\in\bb N\setminus\{0\}$. 
	\begin{Enumerate}
	\item The set $\HomCA(\mc Dn,\mc Dk)$ contains an injective map if and only if $k\geq n$.
	\item The set $\HomCA(\mc Dn\times\mc Dm,\mc Dk)$ contains an injective map if and only if $k\geq n+m-1$.
	\end{Enumerate}
\end{proposition}
\begin{proof}
	It is clear that $\mc Dn$ embeds in $\mc Dk$ whenever $k\geq n$. Assume we have $k\in\bb N\setminus\{0\}$ and an 
	injective map $\varphi\in\HomCA(\mc Dn,\mc Dk)$. Let $\kappa_k\DF\mc Dk\to\bb R^{k-1}$ be as in 
	\Cref{S65}(ii), and let $\Phi\DF\bb R^n\to\bb R^{k-1}$ be the linear extension of $\kappa_k\circ\varphi$.
	\[
		\begin{tikzcd}[column sep=large]
			\mc Dn \arrow[r,hookrightarrow,"\varphi"] \arrow[d,swap,"\subseteq"]
			& \mc Dk \arrow[d,hookrightarrow,"\kappa_k"]
			\\
			\bb R^n \arrow[r,dashed,swap,"\Phi"] & \bb R^{k-1}
		\end{tikzcd}
	\]
	Since $\Phi|_{\mc Dn}$ is injective, we have $\ker\Phi\cap\{\E[n]\}^\perp=\{0\}$ and hence 
	$\dim\ker\Phi\leq 1$. It follows that 
	\[
		k-1\geq\dim\ran\Phi=n-\dim\ker\Phi\geq n-1,
	\]
	i.e., $k\geq n$.

	We come to the proof of item (ii). If $n=1$ or $m=1$, the assertion follows from (i). 
	Assume throughout the following that $n,m\geq 2$. 

	We construct an embedding of $\mc Dn\times\mc Dm$ into $\mc D(n+m-1)$. The map
	\[
		\Xi\DF\left\{
		\begin{array}{rcl}
			\bb R^{n-1}\times\bb R^{m-1} & \to & \bb R^{n+m-2}
			\\
			\big((\alpha_j)_{j=1}^{n-1},(\beta_i)_{i=1}^{m-1}\big) & \mapsto &
			\big(\alpha_1,\ldots,\alpha_{n-1},\beta_1,\ldots,\beta_{m-1}\big)^T
		\end{array}
		\right.
	\]
	and the rescaling $M_{\frac 12}$ are linear bijections.
	The image of $M_{\frac 12}\circ\Xi\circ(\kappa_n\times\kappa_m)$ is contained in the image of
	$\kappa_{n+m-1}$, and hence there exists a map $\varphi$ with 
	\[
		\begin{tikzcd}[column sep=large]
			\mc Dn\times\mc Dm \arrow[rr,dashed,"\varphi"] 
				\arrow[d,hookrightarrow,swap,"\kappa_n\times\kappa_m"]
			&& \mc D(n+m-1) \arrow[d,hookrightarrow,"\kappa_{n+m-1}"]
			\\
			\bb R^{n-1}\times\bb R^{m-1} \arrow{r}{\cong}[swap]{\Xi}
			& \bb R^{n+m-2} \arrow{r}{\cong}[swap]{M_{\frac 12}}
			& \bb R^{n+m-2}
		\end{tikzcd}
	\]
	Clearly, $\varphi$ is an injective homomorphism.

	Assume we have $k\in\bb N\setminus\{0\}$ and an injective map $\varphi\in\HomCA(\mc Dn\times\mc Dm,\mc Dk)$. Set
	\begin{multline*}
		Y\DE\Conv\Big(\big\{(e_i,e_m)\DS i\in\{1,\ldots,n\}\big\}
		\\
		\cup\big\{(e_n,e_j)\DS j\in\{1,\ldots,m\}\big\}\Big)\subseteq\mc Dn\times\mc Dm,
	\end{multline*}
	then 
	\begin{multline*}
		(\kappa_n\times\kappa_m)(Y)=
		\Conv\Big(\big\{(0,0)\big\}\cup\big\{(e_j,0)\DS j\in\{1,\ldots,n-1\}\big\}
		\\
		\cup\big\{(0,e_i)\DS i\in\{1,\ldots,m-1\}\big\}\Big)\subseteq\bb R^{n-1}\times\bb R^{m-1}.
	\end{multline*}
	Since $\{(e_j,0)\DS j\in\{1,\ldots,n-1\}\}\cup\{(0,e_i)\DS i\in\{1,\ldots,m-1\}\}$ is a basis 
	of $\bb R^{n-1}\times\bb R^{m-1}$, the set $(\kappa_n\times\kappa_m)(Y)$ contains a nonempty open subset of 
	$\bb R^{n-1}\times\bb R^{m-1}$ and there exists a linear map $\Phi\DF\bb R^{n-1}\times\bb R^{m-1}\to\bb R^{k-1}$ with 
	$\Phi\circ(\kappa_n\times\kappa_m)|_Y=\kappa_k\circ\varphi|_Y$.
	\[
		\begin{tikzcd}[column sep=large]
			\mc Dn\times\mc Dm\,\supseteq\mkern-117mu
			& Y\arrow[r,hookrightarrow,"\varphi|_Y"] 
				\arrow[d,hookrightarrow,swap,"(\kappa_n\times\kappa_m)|_Y"]
			& \mc Dk \arrow[d,hookrightarrow,"\kappa_k"]
			\\
			& \bb R^{n-1}\times\bb R^{m-1} \arrow[r,dashed,swap,"\Phi"] & \bb R^{k-1}
		\end{tikzcd}
	\]
	The restriction of $\Phi$ to $(\kappa_n\times\kappa_m)(Y)$ is injective, and it follows that $\Phi$ is 
	injective. Thus, 
	\[
		n+m-2=\dim\big(\bb R^{n-1}\times\bb R^{m-1}\big)\leq\dim\bb R^{k-1}=k-1.
	\]
	It follows that $n+m-1\leq k$.
\end{proof}


\printbibliography

{\footnotesize
\begin{flushleft}
	H.\,Woracek\\
	Institute for Analysis and Scientific Computing\\
	TU Wien\\
	Wiedner Hauptstra{\ss}e\ 8--10/101\\
	1040 Wien\\
	AUSTRIA\\
	email: \texttt{harald.woracek@tuwien.ac.at}\\[5mm]
\end{flushleft}
}

\end{document}